\documentclass[11pt]{article} % use larger type; default would be 10pt
\usepackage[utf8]{inputenc} % set input encoding (not needed with XeLaTeX)

\usepackage[margin=1in]{geometry} % to change the page dimensions
\usepackage{graphicx} % support the \includegraphics command and options

\usepackage{dsfont} % for indicator \indc

\usepackage{booktabs} % for much better looking tables	
\usepackage{array} % for better arrays (eg matrices) in maths
\usepackage{paralist} % very flexible & customisable lists (eg. enumerate/itemize, etc.)
\usepackage{verbatim} % adds environment for commenting out blocks of text & for better verbatim
\usepackage{mathrsfs}
\usepackage{amssymb}
\usepackage{amsthm}
\usepackage{amsmath,amsfonts,amssymb}
\usepackage{esint}
\usepackage{graphics}
\usepackage{enumerate}
\usepackage{mathtools}
\usepackage{xfrac}
\usepackage{subcaption}
\usepackage{stmaryrd}
 \usepackage{mathabx}

\usepackage[dvipsnames]{xcolor}
\usepackage[colorlinks=true, pdfstartview=FitV, linkcolor=blue, citecolor=blue, urlcolor=blue]{hyperref}
\usepackage[normalem]{ulem}

\usepackage{tikz}
\usetikzlibrary{decorations.pathreplacing}
\usetikzlibrary{calc}
\usepackage{pgf}
\numberwithin{equation}{section}
\numberwithin{figure}{section}

\newtheorem{theorem}{Theorem}[section]

\newtheorem{corollary}[theorem]{Corollary}
\newtheorem{proposition}[theorem]{Proposition}
\newtheorem{lemma}[theorem]{Lemma}
\theoremstyle{definition}
\newtheorem{definition}[theorem]{Definition}

\newcommand*{\supp}{\ensuremath{\mathrm{supp\,}}}

\newcommand*{\N}{\ensuremath{\mathbb{N}}}

\newcommand*{\Z}{\ensuremath{\mathbb{Z}}}

\newcommand*{\R}{\ensuremath{\mathbb{R}}}

\newcommand{\eps}{\varepsilon}

\renewcommand*{\tilde}{\widetilde}

\renewcommand{\P}{\ensuremath{\mathbb{P}}}

\newcommand{\s}{\mathbf{s}}

\newcommand{\ep}{\eps}

\DeclareSymbolFont{boldoperators}{OT1}{cmr}{bx}{n}
\SetSymbolFont{boldoperators}{bold}{OT1}{cmr}{bx}{n}
\usepackage{accents}

\newcommand{\T}{\mathbb{T}}

\def\XXint#1#2#3{{\setbox0=\hbox{$#1{#2#3}{\int}$}
\vcenter{\hbox{$#2#3$}}\kern-.5\wd0}}

\let\originalleft\left
\let\originalright\right
\renewcommand{\left}{\mathopen{}\mathclose\bgroup\originalleft}
\renewcommand{\right}{\aftergroup\egroup\originalright}

\renewcommand{\phi}{\varphi}

\newcommand{\nconstant}{Z}

\newcommand{\indc}{\mathds{1}}

\newcommand{\E}{\mathbb{E}}

\newcommand{\sol}{\mathcal{T}}
\newcommand{\sola}{\mathcal{S}}

\DeclareMathOperator{\Var}{Var}

\makeatletter
\pgfmathdeclarefunction{erf}{1}{%
  \begingroup
    \pgfmathparse{#1 > 0 ? 1 : -1}%
    \edef\sign{\pgfmathresult}%
    \pgfmathparse{abs(#1)}%
    \edef\x{\pgfmathresult}%
    \pgfmathparse{1/(1+0.3275911*\x)}%
    \edef\t{\pgfmathresult}%
    \pgfmathparse{%
      1 - (((((1.061405429*\t -1.453152027)*\t) + 1.421413741)*\t 
      -0.284496736)*\t + 0.254829592)*\t*exp(-(\x*\x))}%
    \edef\y{\pgfmathresult}%
    \pgfmathparse{(\sign)*\y}%
    \pgfmath@smuggleone\pgfmathresult%
  \endgroup
}
\makeatother

\newcommand{\F}{\mathcal{F}}
\usepackage{titlesec}

\newcommand{\addperiod}[1]{#1.}
\titleformat{\section}
   {\centering\normalfont\Large}{\thesection.}{0.5em}{}
\titleformat*{\subsection}{\bfseries}
\titleformat{\subsubsection}[runin]
  {\normalfont\bfseries}
  {\thesubsubsection.}
  {0.5em}
  {\addperiod}
\titleformat*{\subsubsection}{\normalfont\itshape}
\titleformat*{\paragraph}{\bfseries}
\titleformat*{\subparagraph}{\large\bfseries}

\title{Turbulent and intermittent phenomena in a universal total anomalous dissipator}

\author{Elias Hess-Childs\thanks{Department of Mathematical Sciences, Carnegie Mellon University.
{\footnotesize \href{mailto:aa@cims.nyu.edu}{ehesschi@andrew.cmu.edu}.}
}
\and 
Keefer Rowan\thanks{Courant Institute of Mathematical Sciences, New York University.
{\footnotesize \href{mailto:keefer.rowan@cims.nyu.edu}{keefer.rowan@cims.nyu.edu}.}
}
}
\date{\today}

\usepackage[nottoc,notlot,notlof]{tocbibind}

\begin{document}

\maketitle

% \setcounter{tocdepth}{2} 
% \tableofcontents

\begin{abstract}
    For all $\alpha \in (0,1)$, we construct an explicit divergence-free vector field $V \in L^\infty([0,1],C^\alpha(\T^2))$ that exhibits universal anomalous (total) dissipation, accelerating dissipation enhancement, Richardson dispersion, anomalous regularization, and spatial intermittency. Additionally, we demonstrate the sharpness of the intermittent Obukhov-Corrsin regime for certain parameter ranges.
\end{abstract}

\section{Introduction}

In this paper, we consider solutions to the drift-diffusion equation
\begin{equation}\label{eq:intro_drift_diffusion_equation}
        \begin{cases}
            \partial_t \theta^\kappa - \kappa \Delta \theta^\kappa + V \cdot \nabla \theta^\kappa = 0&\text{in}\ (0,1)\times \T^2,\\
            \theta^\kappa(0,\cdot) = \theta_0(\cdot)&\text{on}\ \T^2,
        \end{cases}
\end{equation}
where $V(t,x)$ is a divergence-free vector field and $\theta_0$ is in $TV(\T^2),$ the space of Borel measures with finite total variation. We note that~\eqref{eq:intro_drift_diffusion_equation} is the Fokker-Planck equation for the stochastic differential equation
\begin{equation}
\label{eq:stochastic-intro}
\begin{cases}
dX_t^\kappa=V(t,X_t^\kappa)\,dt+\sqrt{2\kappa}\,dw_t,\\
X_0^\kappa=x,
\end{cases}
\end{equation}
where $w_t$ is a standard Brownian motion in $\R^2$.

We are interested in constructing a vector field $V$ that exhibits the phenomena of \textit{passive scalar turbulence}, discussed further in Subsection~\ref{ss:background}. As such, we are interested in low-regularity velocity fields, $V \in L^\infty([0,1], C^\alpha(\T^2))$. For each $\alpha \in (0,1)$, we will construct a corresponding velocity field as described in Subsection~\ref{ss:flow}. We now state the main results we provide about the velocity fields $V$.

This work builds on the construction and results of our previous work~\cite{hess-childs_universal_2025}. Our first statement is essentially identical to~\cite[Theorem 1.1]{hess-childs_universal_2025}, though as our construction of $V$ here meaningfully differs, the result is novel---and more importantly the proof of Theorem~\ref{thm:anomalous-dissipation}, through the more precise statement Theorem~\ref{thm:limit-is-close-pointwise-ell-dependent}, is the starting point for the rest of our main results.

\begin{theorem}[Anomalous total dissipation]
    \label{thm:anomalous-dissipation}
    For all $\alpha \in (0,1)$, letting $V \in L^\infty([0,1], C^\alpha(\T^2))$ be the corresponding incompressible velocity field constructed in Subsection~\ref{ss:flow}, there exists $C(\alpha)>0$ such that for all $\kappa \in (0,1)$ and $\theta_0 \in TV(\T^2)$ such that $\int \theta_0(dx) =0$, if $\theta^\kappa(t,x)$ is the solution to~\eqref{eq:intro_drift_diffusion_equation}, then 
    \[\|\theta^\kappa(1,\cdot)\|_{L^1(\T^2)} \leq C \kappa^{(1-\alpha)^2/12} \|\theta_0\|_{TV(\T^2)}.\]
\end{theorem}

We next note the following corollary, proved as a direct consequence of Theorem~\ref{thm:anomalous-dissipation} in Subsection~\ref{ss:intro-proofs}, giving accelerating dissipation enhancement for large times.

\begin{corollary}[Accelerating dissipation enhancement]
    \label{cor:diffusion-enhancement}
    For all $\alpha \in (0,1)$, extending the velocity field constructed in Subsection~\ref{ss:flow} periodically in time to give an incompressible velocity field $V \in L^\infty([0,\infty),C^\alpha(\T^2))$, there exists $C(\alpha)>0$ such that for all $\kappa \in (0,1)$ and all $\theta_0 \in L^2(\T^2)$ such that $\int \theta_0(x)\,dx =0$, if $\theta^\kappa(t,x)$ is a solution to~\eqref{eq:intro_drift_diffusion_equation} on $[0,\infty) \times \T^2$, then for all $t \in [0,\infty)$ we have the enhanced dissipation bound,
    \[\|\theta^\kappa(t,\cdot)\|_{L^2(\T^2)} \leq Ce^{- C^{-1} \log(\kappa^{-1})(t-1)} \|\theta_0\|_{L^2(\T^2)}.\]
\end{corollary}

Our next result is that $V$ exhibits \textit{Richardson dispersion} for a particle started anywhere on the torus at time $t=0$. As we discuss in Subsection~\ref{ss:background}, this is in some sense a sharp result: the variance could not possibly be larger given the regularity of $V$. Letting $\mathcal{P}(\T^2)$ denote the space of probability measures on $\T^2$, we use the following convention for the variance of a $\T^2$ valued random variable.

\begin{definition}
    For a random variable $X$ with law $\mu\in\mathcal{P}(\T^2)$,
    \[\text{Var}(X):=\text{Var}(\mu):=\inf_{a\in\T^2}\int_{\T^2}|x-a|^2d\mu(x),\]
    where $|x-y|$ denotes the distance between $x$ and $y$ in $\T^2$.
\end{definition}

\begin{theorem}[Richardson dispersion]
    \label{thm:richardson}
     For all $\alpha \in (0,1)$, letting $V \in L^\infty([0,1],C^\alpha(\T^2))$ be the corresponding incompressible velocity field constructed in Subsection~\ref{ss:flow}, there exists $C(\alpha)>0$ so that for all $x \in \T^2$ and $\kappa \in (0,1)$, the solution to~\eqref{eq:stochastic-intro} satisfies
    \[\Var(X_t^\kappa) \geq C^{-1} \big(\kappa t+t^{\frac{2}{1-\alpha}}\big),\qquad t \in [0,1].\]
\end{theorem}

We now give a statement of \textit{anomalous regularization}, which gives that, uniformly in diffusivity, the solution $\theta(t,x)$ gains regularity over the initial data. We note however that this statement is not sharp: as discussed in Subsection~\ref{ss:open-questions}, it is likely possible for some flows to get regularization up to $H^{\frac{1-\alpha}{2}}$, with $H^s$ defined for $s \in \R$ below in Definition~\ref{defn:H-s-spaces}.

\begin{theorem}[Anomalous regularization]
    \label{thm:anomalous-regularization}
    There exists $\gamma \in (0,1/2)$ such that for all $\alpha \in (0,1)$, letting $V\in L^\infty([0,1],C^\alpha(\T^2))$ be the corresponding incompressible velocity field constructed in Subsection~\ref{ss:flow}, there exists $C(\alpha)>0$ such that for all $\kappa \in (0,1)$ and all $\theta_0 \in L^2(\T^2)$ such that $\int \theta_0(x)\,dx =0 $, if $\theta^\kappa(t,x)$ is the solution to~\eqref{eq:intro_drift_diffusion_equation}, then
    \[\|\theta^\kappa\|_{L^2([0,1],H^{(1-\alpha)^2 \gamma}(\T^2))} \leq C \|\theta_0\|_{L^2(\T^2)}.\]
    In particular, $\gamma$ is given (semi-)explicitly in Definition~\ref{defn:big-parameter-defn}.
\end{theorem}

Our next theorem is a statement of spatial intermittency. We give a thorough discussion of intermittency in Subsection~\ref{ss:background}, however in this setting intermittency of the solution can be very loosely thought of as the spatial regularity of the solution depending on the integrability exponent, for example $\theta \in H^\sigma(\T^2)$ but $\theta \not \in C^\sigma(\T^2)$. The following statement is of particular interest as it gives that solutions develop spatially intermittent regularity for all initial data---even spatially smooth initial data. In the next theorem, we make use of Riesz potential spaces $H^{s,p},$ defined below.
\begin{definition}
\label{defn:H-s-spaces}
    For $f : \T^2 \to\R$ with $\int f(x)\,dx =0$ and for any $1 \leq p \leq \infty, s\in \R$, we define the Riesz potential space $H^{s,p}(\T^2)$ by
    \[\|f\|_{H^{s,p}(\T^2)} := \|(-\Delta)^{s/2} f\|_{L^p(\T^2)},\]
    where $(-\Delta)^s$ is defined using the Fourier transform. We denote $H^s(\T^2) := H^{s,2}(\T^2)$, which corresponds to the conventional usage.
\end{definition}

\begin{theorem}[Intermittent regularity]
    \label{thm:intermittency}
    For all $\alpha \in (0,1)$, let $V \in L^\infty([0,1], C^\alpha(\T^2))$ be the corresponding incompressible velocity field constructed in Subsection~\ref{ss:flow}, $\theta_0 \in L^2(\T^2), \theta_0 \ne 0$ such that $\int \theta_0(x)\,dx =0 $, and $\theta^\kappa(t,x)$ be the solution to~\eqref{eq:intro_drift_diffusion_equation}. Then there exists $t_*>0$ depending on $\theta_0$ such that for all $t \in (0,t_*), \beta \in (0,1), p > \beta^{-1}$ 
    \[\lim_{\kappa \to 0}\|\theta^\kappa(t,\cdot)\|_{H^{\beta,p}(\T^2)} = \infty.\]
    In particular,
    \[\lim_{\kappa \to 0} \|\theta^\kappa\|_{L^2([0,1], H^{\beta,p}(\T^2))} = \infty.\]
    Additionally, if
    \[\int_{\{(x,y)\in\T^2:x < \sqrt{2}/2\}} \theta_0(x,y)\,dx\,dy \ne 0,\]
    then $t_*$ can be taken to be $1 - \sigma_0/2$, with $\sigma_0$ defined in Definition~\ref{defn:big-parameter-defn}.
\end{theorem}

\subsection{Background and previous work}
\label{ss:background}

A central feature of the phenomenology of turbulence is anomalous behavior in the limit of vanishing (molecular) dissipation. In fluid turbulence, the most basic of these phenomena is that of the anomalous dissipation of energy, in which the kinetic energy of a fluid is dissipated at a uniform rate even as the molecular viscosity---the ultimate physical mechanism of dissipation---is sent to $0$. See~\cite[Section 5.2]{frisch_turbulence_1995} for a lucid discussion of the phenomenon and its empirical evidence. Other ``anomalous'' behavior of fluid turbulence in the vanishing viscosity limit includes the K41 inertial range statistics~\cite{kolmogorov_degeneration_1941,kolmogorov_dissipation_1941,kolmogorov_local_1941}, Eulerian spontaneous stochasticity~\cite{lorenz_predictability_1969,mailybaev_spontaneously_2016,mailybaev_spontaneous_2016}, and intermittent corrections to the K41 statistics~\cite[Chapter 8]{frisch_turbulence_1995} as well as~\cite{cheskidov_euler_2014,cheskidov_volumetric_2023, rosa_intermittency_2025}.

Demonstrating the presence of these anomalous phenomena in mathematical models of fluids is an exceedingly challenging problem which is far from a satisfactory resolution. The best results in that direction come from the convex integration literature: see for example~\cite{de_lellis_h-principle_2012,de_lellis_dissipative_2013,buckmaster_anomalous_2015,isett_proof_2018,novack_intermittent_2023} and in particular the reviews~\cite{buckmaster_convex_2019,de_lellis_weak_2022}.

In order to study a problem more tractable than proper fluid turbulence, we instead consider \textit{passive tracer turbulence}. A passive tracer is an object that is advected by the fluid velocity field but does not affect the fluid. The two tracers we are interested in are \textit{passive particles} and \textit{passive scalars}. For the velocity field $V$ and a diffusivity $\kappa \geq 0$, passive particles are solutions to the SDE~\eqref{eq:stochastic-intro} and passive scalars are solutions to the PDE~\eqref{eq:intro_drift_diffusion_equation}. Passive particles model objects like dust particles: pushed around by the fluid field and also subject to thermal fluctuations but not (meaningfully) affecting the fluid field. Passive scalars model something like a dye concentration: being mixed by the fluid field and subject to thermal diffusion but also not affecting the fluid. Passive tracers are known to exhibit a variety of anomalous turbulent phenomena---and so are an important object of study in turbulence theory---while also allowing us to simplify the problem by working with a model ``turbulent'' velocity field $V$ that we can directly specify in place of a true solution to the fluid equations. See~\cite{falkovichParticlesFieldsFluid2001} for a review of passive tracer turbulence in the physics literature.   

In this paper, we are interested in the phenomena of anomalous dissipation, accelerating dissipation enhancement, Richardson dispersion, anomalous regularization, and intermittency in passive tracers advected by incompressible flows. We discuss each of these phenomena below.

\subsubsection{Anomalous dissipation}

Anomalous dissipation for passive scalars happens when the $L^2$ norm of the solution to the advection diffusion equation is dissipated at a uniform rate in the vanishing diffusivity limit, despite the $L^2$ norm being formally conserved in the non-diffusive transport equation. This phenomenon is central to passive scalar turbulence and constitutes an essential input into the inertial range statistical theories of~\cite{obukhov_structure_1949} and~\cite{corrsin_spectrum_1951} that form passive scalar analogs to K41 theory. Anomalous dissipation is the best studied of the passive tracer turbulence phenomena in the mathematical literature---starting with~\cite{drivas_anomalous_2022}, there has been a profusion of examples:~\cite{colombo_anomalous_2023,armstrong_anomalous_2025,burczak_anomalous_2023,elgindi_norm_2024,rowan_anomalous_2024,hofmanova_anomalous_2023,rowan_accelerated_2024,johansson_anomalous_2024} among others. See~\cite[Section 1.2.1]{hess-childs_universal_2025} for a discussion of these works various contributions.

The vector field we construct in Subsection~\ref{ss:flow} exhibits anomalous dissipation for all initial data and in fact exhibits the stronger phenomenon of asymptotic total dissipation, where the entirety of the $L^2$ mass is dissipated by time $t=1$ in the vanishing diffusivity limit. The construction of a vector field exhibiting such a phenomenon was the central focus of our previous work~\cite{hess-childs_universal_2025}, out of which the current paper was developed. The vector field we construct here differs from that constructed in~\cite{hess-childs_universal_2025}, and as such Theorem~\ref{thm:anomalous-dissipation} is not a direct consequence of~\cite[Theorem 1.1]{hess-childs_universal_2025}. However, the ideas for this part are largely the same, so we defer to~\cite[Section 1]{hess-childs_universal_2025} for a thorough discussion of anomalous dissipation and asymptotic total dissipation.

\subsubsection{Accelerating dissipation enhancement}

Accelerating dissipation enhancement is---to the best of our knowledge---a new phenomenon to the enhanced dissipation literature. Enhanced dissipation describes how a divergence-free flow can interact with the dissipation due to the Laplacian in order to speed up the rate of dissipation. Going back to~\cite{constantin_diffusion_2008}, the phenomenon of dissipation enhancement is now fairly well understood. A variety of dissipation enhancing flows have been constructed~\cite{bedrossian_enhanced_2017,zelati_stochastic_2021,albritton_enhanced_2022,coti_zelati_enhanced_2023,villringer_enhanced_2024} and it has been shown in~\cite{feng_dissipation_2019,zelati_relation_2020,cooperman_exponentially_2025} that all mixing flows---for example those constructed in~\cite{bedrossian_almost-sure_2022,blumenthal_exponential_2023,myers_hill_exponential_2022,elgindi_optimal_2023}---exhibit dissipation enhancement, that is the first time the $L^2$ norm of any zero-mean data is split in half happens well before the dissipation time scale of $\kappa^{-1}.$ In~\cite{bedrossian_almost-sure_2021,cooperman_harris_2024,navarro-fernandez_exponential_2025}, sharper estimates are proven for long time scales. In particular, they show estimates of the form
\begin{equation}
\label{eq:unif-dissipation-enhancement}
\|\theta^\kappa_t\|_{L^2} \leq C(\kappa) e^{-\gamma t} \|\theta_0\|_{L^2},
\end{equation}
where $\theta^\kappa$ is a solution to~\eqref{eq:intro_drift_diffusion_equation} for their suitably chosen velocity fields $V$, $\theta_0$ is an arbitrary zero-mean choice of initial data, and $\gamma>0$ is a $\kappa$-independent constant.

The estimate~\eqref{eq:unif-dissipation-enhancement} says that the asymptotic exponential rate of decay is $\kappa$ independent. In contrast to our setting, the flows considered in~\cite{bedrossian_almost-sure_2021,cooperman_harris_2024,navarro-fernandez_exponential_2025} are uniformly spatially Lipschitz while our flow is only uniformly spatially $\alpha$-H\"older. However our estimate of \textit{accelerating dissipation enhancement} given by Corollary~\ref{cor:diffusion-enhancement},
    \[\|\theta^\kappa(t,\cdot)\|_{L^2(\T^2)} \leq Ce^{- C^{-1} \log(\kappa^{-1})(t-1)} \|\theta_0\|_{L^2(\T^2)},\]
gives that the asymptotic exponential rate of decay actually goes to $\infty$ as $\kappa \to 0$. That is, as we decrease the dissipation strength of dissipation, the rate of decay (at least for sufficiently large times) increases. This result is essentially direct from the asymptotic total dissipation given in Theorem~\ref{thm:anomalous-dissipation} and its short proof is given at the end of this section. We note that we could have similarly shown the result for the vector field constructed for~\cite[Theorem 1.1]{hess-childs_universal_2025}.

There are a variety of interesting open questions about the maximal asymptotic rate of decay for an advection-diffusion equation and Batchelor scale formation. A very strong conjecture would be that for all suitably regular flows $V$, for any $\kappa>0$, there exists some mean-zero initial data $\theta_0$ such that for the solution $\theta^\kappa$ to~\eqref{eq:intro_drift_diffusion_equation}, we have that
\begin{equation}
\label{eq:batchelor-conjecture}
\liminf_{t\rightarrow \infty} t^{-1}\log\|\theta^\kappa_t\|_{L^2(\T^2)} \geq - \gamma
\end{equation}
for some $\gamma>0$ independent of $\kappa$. Corollary~\ref{cor:diffusion-enhancement} shows that this cannot be true for velocity fields $V \in L^\infty_t C^\alpha_x$, though interestingly the failure is only by a (relatively small) logarithmic factor. \cite{miles_diffusion-limited_2018} provides a numerical study touching on similar problems and~\cite{hairer_lower_2024} shows that for $V$ taken to be a solution to the stochastically-forced Navier--Stokes equations, $\liminf_{t\to\infty}t^{-1}\log \|\theta^\kappa_t\|_{L^2} \geq - \kappa^{-a}$ for some $a>0$, showing that the failure of~\eqref{eq:batchelor-conjecture} is by at most some algebraic rate in $\kappa$.

\subsubsection{Richardson dispersion}

Richardson dispersion refers to the superballistic (explosive) separation of near particle pairs under advection by a turbulent fluid. More precisely, Richardson's law~\cite{richardson_atmospheric_1926} states that if $R_t$ is the displacement between two particles advected by a turbulent velocity field with approximately $1/3$ spatial regularity, then typically $R_t^2\approx t^3$ at sufficiently large times. This implies that the initial separation of particles and the magnitude of molecular diffusion is inconsequential to the growth of their displacement at large enough times, and---in the infinite Reynolds number limit---that arbitrarily close particle pairs separate in finite time.

Despite being one of the first quantitative phenomenological predictions of turbulence, there are few fluid models for which Richardson's law is well understood. In particular, Richardson-type behavior has been thoroughly investigated in the applied literature for rough transport noise~\cite{GawedzkiVergassola2000,falkovichParticlesFieldsFluid2001,Fannjiang2003}---often called the Kraichnan model due to~\cite{kraichnanSmallScaleStructure1968}. To the best of our knowledge, no vector fields with time regularity greater than that of white noise have been shown to exhibit this phenomenon.

Theorem~\ref{thm:richardson} can be seen as a Richardson-type law for passive particles advected by $V$: if $X_t^\kappa$ and $Y_t^\kappa$ are independent and identically distributed solutions to~\eqref{eq:stochastic-intro}, then
\[\mathbb{E}[|X_t^\kappa-Y_t^\kappa|^2]\geq \Var(X_t^\kappa)\geq C^{-1}t^{\frac{2}{1-\alpha}}.\]
That is, if two particles begin at the same position and are advected by $V$ while being subject to arbitrarily small independent diffusions, then the expectation of their squared displacement grows like $t^{\frac{2}{1-\alpha}}$. We note that Richardson's law corresponds to the $\alpha = \tfrac{1}{3}$ case.

For particles with non-zero initial separation, the following is also a straightforward corollary of Theorem~\ref{thm:richardson}, proved in Appendix~\ref{appendix:variance}.
\begin{corollary}\label{cor:richardson}
Let $X_t^\kappa$ and $Y_t^\kappa$ denote two solutions of~\eqref{eq:stochastic-intro} with independent driving noises and arbitrary initial conditions. Then, letting $R_t^\kappa:=|X_t^\kappa-Y_t^\kappa|$ and $R_0=R^\kappa_0$, there exists $C(\alpha)>0$ so that
\[C^{-1}(R_0^2+\kappa t+ t^{\frac{2}{1-\alpha}})\leq \mathbb{E}[(R_t^\kappa)^2]\leq C(R_0^2+\kappa t+ t^{\frac{2}{1-\alpha}}).\]
\end{corollary}
As a consequence, the magnitude of the mean-squared displacement of two particles is seen to be independent of their initial separation when $t\gtrsim R_0^{1-\alpha}$ and independent of the magnitude of the diffusion when $t\gtrsim\kappa^{\frac{1-\alpha}{1+\alpha}}$.

The Richardson scaling given by Theorem~\ref{thm:richardson} can be viewed as a quantitative statement of \textit{spontaneous stochasticity}: describing the lack of concentration of solutions to~\eqref{eq:stochastic-intro} in the vanishing diffusivity limit and thus the non-uniqueness to the associated zero-diffusivity ODE. 

Spontaneous stochasticity is intimately related to other features of scalar turbulence. It has been extensively studied in the Kraichnan model~\cite{Bernard1998,chaves_lagrangian_2003,jan_integration_2002}, and---for backwards Lagrangian trajectories---has been shown to be equivalent to anomalous dissipation through the fluctuation dissipation formula~\cite{drivas_lagrangian_2017,drivas_lagrangian_2017-1,Eyink_Drivas_2018}. This was exploited in~\cite{johansson_anomalous_2024} to construct an anomalous dissipating autonomous-in-time vector field via spontaneous stochasticity. The results of~\cite{drivas_lagrangian_2017} together with Theorem~\ref{thm:anomalous-dissipation} imply that the the variance of the backward stochastic trajectories for $V$ from time $1$ to time $0$ can be uniformly bounded below as $\kappa \to 0$, while Theorem~\ref{thm:richardson} gives a much more precise statement of the exact scaling of the variance of the forward trajectories over time.

Theorem~\ref{thm:richardson} additionally shows that $V$ maximally increases the variance of Lagrangian trajectories. Indeed, the following proposition---proved in Appendix~\ref{appendix:variance}---shows that the variance of a particle advected by an $L^\infty_tC^\alpha_x$ velocity field with diffusivity $\kappa$ can not be greater than $\kappa t+ t^{\frac{2}{1-\alpha}}$.

\begin{proposition}\label{prop:variance_upper_bound}
Let $u\in L^\infty([0,1],C^\alpha(\T^2))$ and $X_t^{\kappa}$ be a solution to the stochastic differential equation~\eqref{eq:stochastic-intro} with $V$ replaced by $u$. Then there exists $C(\alpha)>0$ so that for all $\kappa\in(0,1)$, $x\in\T^2$, and $t\in[0,1]$
\[\Var(X_t^\kappa)\leq C\big(\kappa t+(\|u\|_{L^\infty([0,1],C^\alpha(\T^2))}t)^{\frac{2}{1-\alpha}}\big).\]
\end{proposition}

Finally, we note that it was remarked in~\cite{hess-childs_universal_2025} that the constructed vector field was ``maximally spreading'' in that the solution to the associated stochastic differential equation $X_t^\kappa$ satisfied
\[\lim_{\kappa\to 0} \Var(X_{1/2}^\kappa)=0\quad\text{and}\quad \lim_{\kappa\to 0}\Var(X_{1/2+t}^\kappa)=Ct^{\frac{2}{1-\alpha}}\ \text{for }t\geq 0.\]
That is, in the vanishing diffusivity limit, the variance of the Lagrangian trajectory grows like $(t-t_0)^\frac{2}{1-\alpha}$ after time $t_0=\frac{1}{2}$. In contrast, here we show a non-asymptotic lower bound on the variance and no longer require the initial period of time $[0,1/2]$ to regularize the problem---instead providing a uniform lower bound on the variance for all positive times.

\subsubsection{Anomalous regularization}

In fluid turbulence, the velocity $u^\nu$ of a fluid with viscosity $\nu$ is expected to live uniformly in (about) $C^{1/3}_x$ in the $\nu \to 0$ limit. Similarly, in passive scalar turbulence the passive scalar solution $\theta^\kappa$ is expected to live uniformly in (about) $C^{(1-\alpha)/2}_x$---if the advecting flow $V$ is $C^\alpha_x$---in the $\kappa \to 0$ limit. It is well understood why these objects cannot live in higher regularity spaces: by the positive side of Onsanger's conjecture~\cite{constantin_onsagers_1994}, if $u^\nu$ had any more regularity, then there could be no anomalous dissipation. A similar argument applies to the passive scalars as well: if they had any more regularity, there would be no anomalous dissipation~\cite[Theorem 4]{drivas_anomalous_2022}. Further, it is understood, at least on a heuristic level, that the advective term of the equation tends to produce small scales in the solution, limiting the spatial regularity of the solution (after sufficient time has elapsed to develop a full turbulence cascade).

What is less well understood is why the solutions should retain any degree of regularity in the vanishing diffusivity limit. When $\kappa>0$, the presence of the Laplacian alone is sufficient to ensure some degree of regularity, e.g.\ by the energy identity $\theta^\kappa \in L^2_t H^1_x$ for all $\kappa>0$. However, this bound degenerates as $\kappa \to 0$. Additionally, when $\kappa=0$, it is known that for arbitrary velocity fields in $ C^\alpha(\T^2)$ there is no guarantee that even smooth initial data retains any degree of Sobolev regularity at positive times~\cite{alberti_loss_2019}.

One answer to this problem is that there is some mechanism of \textit{anomalous regularization} which smooths out solutions, uniformly in diffusivity. It is then the balance between the anomalous regularization and the roughness induced by the turbulence cascade that causes solutions to live at the correct regularity uniformly in diffusivity. See~\cite{drivas_self-regularization_2022} for a discussion of this idea in the context of fluid turbulence.

For passive scalars, anomalous regularization in the Kraichnan model was shown in~\cite{galeati_anomalous_2024} building off of ideas of regularization by noise in the SDE and SPDE literature~\cite{flandoli_well-posedness_2010,flandoli_random_2011,gess_regularization_2018}. In the case of exactly spatially self-similar transport noise, anomalous regularization estimates also follow as in~\cite{zelati_statistically_2023}.

Theorem~\ref{thm:anomalous-regularization} is---as far as we are aware---the first statement of anomalous regularization of passive scalars induced by a vector field with more time regularity than white noise. We note that our previous result~\cite[Corollary 1.3]{hess-childs_universal_2025} had a form of anomalous regularization, but it required waiting unit time to manifest. Theorem~\ref{thm:anomalous-regularization} gives uniform-in-diffusivity regularization for all $L^2$ initial data all the way back to the initial time. However, our current result is (seemingly) not optimal in the attained regularity. We expect that, at least for some vector fields, one should get anomalous regularization all the way up to the ``Obukhov--Corrsin'' regularity of $L^2([0,1], H^{\frac{1-\alpha}{2}}(\T^2)),$ which is beyond our current result. We note that~\cite{galeati_anomalous_2024} does attain the (almost) optimal regularity in their setting as discussed in~\cite[Remark 1.2]{galeati_anomalous_2024}.

\subsubsection{Intermittency:\ Overview}

\label{sss:intermitency-intro}
Intermittency describes a broad collection of phenomena that typically pertain to corrections to the self-similar description of turbulence provided by K41 theory (in the fluid setting) and the Obukhov--Corrsin theory (in the passive scalar setting). The most classical manifestation of intermittency is corrections to the scaling predictions of the structure functions in fluid turbulence. The $p$th absolute structure function $S_p:\T^2\to\R$ for a (possibly random, e.g.\ as induced by a stochastic forcing) fluid velocity $u$, assumed to be taken in the vanishing viscosity limit, is defined by
\[S_p(\ell) := \lim_{T \to \infty}\E\bigg[\frac{1}{T} \int_0^T \int_{\T^2} |u(t,x+\ell) - u(t,x)|^p\,dx\,dt\bigg].\]
Then K41 theory describes a turbulent fluid (in particular assuming that $u$ anomalously dissipates energy, see~\cite[Chapter 6]{frisch_turbulence_1995}) as a $1/3$ regular, statistically self-similar velocity field. In that case, the structure functions scale as
\[S_p(\ell) \approx |\ell|^{(\frac{1}{3} - \zeta_p)p},\]
with $\zeta_p =0$. However, numerical and experimental evidence suggests this scaling is in fact false and rather $\zeta_3= 0$ and $ \zeta_p>0$ for $p>3$: for an overview of the phenomenon and its evidence from the physics perspective, see~\cite[Chapter 8]{frisch_turbulence_1995}.

Intermittency in fluid turbulence is far from being completely understood and the precise values of the ``intermittency corrections'' $\zeta_p$ are unknown. In the passive scalar setting, studies of the Kraichnan model in the physics literature have provided predictions of intermittency for the passive scalar structure function~\cite{gawedzki_anomalous_1995,bernard_slow_1998} perturbatively in the spatial regularity $\alpha \to 0$ that have been numerically verified~\cite{frisch_intermittency_1998}. However, rigorous demonstrations of intermittency in passive scalar turbulence are still lacking.

We note that absolute structure functions are strongly related to spatial Besov regularity spaces, defined for $s \in (0,1)$ by
\[\|f\|_{B^{s,p}_\infty(\T^2)} = \sup_{\ell \in \T^2} |\ell|^{-s} \bigg(\int_{\T^2} |f(x + \ell) - f(x)|^p\,dx\bigg)^{1/p}.\]
We can think of $B^{s,p}_\infty$ as one (of the many possible, $H^{s,p}$ being another) valid choices for a space with $s \in (0,1)$ many derivatives between $L^p$ and $W^{1,p}$. This idea is made precise by the theory of interpolation spaces~\cite{bergh_interpolation_1976}.

We can then see the intermittency prediction that $\zeta_p > 0$ for $p>3$, as saying that $u \in B^{1/3,3}_\infty$ but $u \not \in B^{1/3, 3+\ep}_\infty$. That is: \textit{the spatial regularity of $u$ depends on the integrability exponent it is being measured at.} It is precisely this notion of intermittency we have in mind when we say that Theorem~\ref{thm:intermittency} is a statement of spatial intermittency. Theorem~\ref{thm:anomalous-regularization} says that $\|\theta^\kappa(t,\cdot)\|_{H^{(1-\alpha)^2 \gamma}}$ stays well controlled uniformly in $\kappa$, but Theorem~\ref{thm:intermittency} says that for an initial interval of times $t \in (0,t_*)$, $\|\theta^\kappa(t,\cdot)\|_{H^{\beta,p}(\T^2)} \to \infty$ as $\kappa \to 0$ for any $p> \beta^{-1}$, in particular $\|\theta^\kappa(t,\cdot)\|_{H^{(1-\alpha)^2 \gamma,p}} \to \infty$ for $p$ large enough. Since these results hold for any initial data, we demonstrate that even smooth initial data immediately develops intermittent spatial regularity.

\subsubsection{Intermittency:\ Spatial and temporal}

Let us consider further what it means for a function to be ``intermittent'' in the sense given by Theorem~\ref{thm:anomalous-regularization} and Theorem~\ref{thm:intermittency}. Ignoring the time dependency for now and considering instead the slightly more ``hands on'' $B^{s,p}_\infty$ norms, we would say a function $f : \T^2 \to \R$ is intermittent if for some $1 \leq p < q \leq \infty$,
\[\|f\|_{B^{s,p}_\infty} < \infty \quad \text{and} \quad \|f\|_{B^{s,q}_\infty} =\infty.\]
Let's consider the extreme case of $p=1$ and $q=\infty$. Then, unpacking the definition of $B^{s,p}_\infty$, this says that ``typically'' in $x$, $f(x+\ell) - f(x)$ is order $|\ell|^s$, but (for an appropriate sequence of $\ell \to 0$) there exists a small positive measure set of $x$ where it is asymptotically larger that $|\ell|^s$. This suggests why we call the phenomenon intermittency: the spatial roughness of $f$ is distributed ``intermittently'' in space, concentrating on a sparse spatial region. More precisely, we call this sort of intermittency \textit{spatial intermittency}.

We can also consider \textit{temporal intermittency}, where the spatial roughness of the function $f$ is distributed intermittently in time.\footnote{It might be more accurate to call this ``temporal intermittency of the spatial regularity'' in order to contrast it with intermittency of the temporal regularity. As we never consider temporal regularity in this work, we stick with temporal intermittency for simplicity.} There is of course myriad ways of quantifying this phenomenon, but let us consider the particularly simple choice of considering spaces of the form $L^p_t H^{s,q}_x$. Then we will say a function $f : [0,1] \times \T^2 \to \R$ is temporally intermittent if for some $s \in (0,1), q \in [1,\infty]$, there exists $1 \leq p_1 < p_2 \leq \infty$ such that 
\[f \in L^{p_1}_t H^{s,q}_x \quad \text{and} \quad f \not \in L^{p_2}_t H^{s,q}_x.\]
Similarly to the spatial case, what this is saying is that the times where $f$ is spatially rough are ``intermittent'', concentrating on a sparse set.

While we are not aware of any passive scalar turbulence constructions that demonstrate spatial intermittency for smooth initial data, temporally intermittency is present throughout the anomalous dissipation literature---though largely unremarked upon. In particular, the works~\cite{colombo_anomalous_2023,elgindi_norm_2024} construct anomalous dissipation examples with the velocity field $u \in L^\infty_t C^\alpha_x$ and the solutions $\theta^\kappa$ obeying uniform-in-diffusivity bounds in $L^2_t C^\beta_x$ for all $\beta< \frac{1-\alpha}{2}$; that is, they essentially show the Obukhov--Corrsin regularity of the solutions. Since $C^\beta$ embeds into  $H^{\beta,q}$ for all $q \in [1,\infty)$, they in particular have that $\theta^\kappa$ is uniformly bounded in $L^2_t H^\beta_x$. The following proposition however shows that their solutions $\theta^\kappa$ cannot be uniformly bounded $L^\infty_t H^{s}_x$ for any $s>0$.

\begin{proposition}
\label{prop:time-intermittency}
    Let $u \in L^\infty([0,1] \times \T^2)$ with $\nabla \cdot u =0$ and $\theta_0 \in L^2(\T^2)$. For all $\kappa>0$, let $\theta^\kappa$ solve~\eqref{eq:intro_drift_diffusion_equation} with $V$ replaced by $u$. Suppose for some $\theta \in L^\infty([0,1], L^2(\T^2))$ and some sequence of $\kappa_j \to 0$, $\theta^{\kappa_j} \stackrel{L^\infty_t L^2_x}{\rightharpoonup} \theta$. Suppose also $E(t) := \|\theta(t,\cdot)\|_{L^2}^2$ is not continuous. Then, for any $s>0,$
    \[\lim_{j\to\infty} \|\theta^{\kappa_j}\|_{L^\infty([0,1],H^{s}(\T^2))} = \infty.\]
\end{proposition}

The proof of this proposition---which is essentially an application of the Aubin--Lions lemma---is provided in Section~\ref{ss:intro-proofs}. This proposition applies to~\cite{colombo_anomalous_2023,elgindi_norm_2024}---as well as~\cite{drivas_anomalous_2022,hess-childs_universal_2025} among others---as these anomalous dissipation examples are all built on mixing flows culminating at a singular time. As such, if $\kappa_j$ is a sequence along which the flow exhibits anomalous dissipation, and $t_*$ is the singular time of the flow, directly inspecting the construction it is apparent that the limiting zero-diffusivity solution $\theta$ has its $L^2$ norm jump down at time $t=t_*$. That is, all of the anomalous dissipation happens at the singular time, hence the limiting energy profile $E(t)$ is discontinuous. Then, according to Proposition~\ref{prop:time-intermittency}, we cannot possibly have $\theta^\kappa$ uniformly bounded in $L^\infty_t H^{s}_x$ for any $s>0$. However, since $\theta^\kappa$ is uniformly bounded in $L^2_t H^\beta_x$ for any $\beta < \frac{1-\alpha}{2}$ in~\cite{colombo_anomalous_2023,elgindi_norm_2024}, we get temporal intermittency in the sense described above. In fact, this discussion suggests that some form of time intermittency is actually generic to anomalous dissipation examples built around mixing flows concentrating on a singular time.

\subsubsection{Intermittency:\ Dissipation regularity}
\label{sss:dissipation-regularity}

Proposition~\ref{prop:time-intermittency} is an example of a result that gives a criterion for intermittent regularity given some properties of the limiting zero-diffusivity solution. In Proposition~\ref{prop:time-intermittency} we consider only the fairly coarse information of the energy profile in time, $E(t)$. In~\cite{rosa_intermittency_2025}, building on~\cite{de_rosa_support_2024,de_rosa_intermittency_2024}, this idea is much more thoroughly pursued. By using the (much richer) space-time dissipation distribution, criteria relating the (potentially fractional) dimension of the support of the dissipation distribution and intermittent regularity is given in~\cite{rosa_intermittency_2025}. \cite{rosa_intermittency_2025} primarily focuses on the case of the Euler equations, in which their work can be viewed as a sharpening of the positive side of the Onsanger conjecture as proved by~\cite{constantin_onsagers_1994}. However, we will only be interested in the application of their ideas to the passive scalar setting. We now recall their main result on passive scalar intermittency.

\begin{theorem}[{\cite[Corollary 5.4]{rosa_intermittency_2025}}]
\label{thm:intermittent-obukov-corrsin}
    Let $p\in [1,\infty], q \in [2,\infty],\sigma,\beta \in(0,1),$ $u : [0,1] \times \T^2 \to \R^2$, and $\theta : [0,1] \times \T^2 \to \R$ with $u \in L^p([0,1],B^{\sigma,p}_\infty(\T^2)),$ $\nabla \cdot u =0,$ and $\theta \in L^q([0,1], B^{\beta,q}_\infty(\T^2))$. Suppose that on $[0,1] \times \T^2$, $\theta$ solves
    \[\partial_t \theta + \nabla \cdot (u \theta) =0.\]
    Define the \textit{dissipation distribution} $D$ associated to $(u,\theta)$ as
    \[D := \partial_t \theta^2 + \nabla \cdot (u \theta^2).\]
    Suppose that $D$ is a Radon measure on $[0,1] \times \T^2$ and that there exists a set $S$ such that $D$ is supported on $S$---that is $D(A \cap S) = D(A)$ for all Borel sets $A$---such that $S$ has Hausdorff dimension $\gamma \in [0,3].$ If $D\ne 0$ then we must have
    \[\frac{2\beta}{1-\sigma} \leq 1 - \Big(1 - \frac{2}{q} - \frac{1}{p}\Big)(3-\gamma).\]
\end{theorem}

Let us now interpret the above theorem. The dissipation distribution $D$ is essentially a measure of anomalous dissipation and can also be thought of as a failure for the solution $\theta$ to be ``renormalized'' in the sense of DiPerna-Lions~\cite{diperna_ordinary_1989}. In the simplest case, $\theta$ arises as the strong limit of the vanishing diffusivity solutions to the advection-diffusion equation~\eqref{eq:intro_drift_diffusion_equation}: $\theta^{\kappa_j} \stackrel{L^2_t L^2_x}{\to} \theta$. In this case, using measure tightness, one can see that $D$ is in fact a negative finite space-time Radon measure and is the distributional limit of the sequence of space-times measures $2\kappa_j |\nabla \theta^{\kappa_j}|^2$. For any positive $\kappa$, we note that $(\theta^\kappa)^2$ solves the equation
\[\partial_t (\theta^\kappa)^2 - \kappa\Delta (\theta^\kappa)^2 + u \cdot \nabla (\theta^\kappa)^2 = - 2 \kappa |\nabla \theta^\kappa|^2,\]
so $2 \kappa_j |\nabla \theta^{\kappa_j}|^2$ is the space-time measure indicating where the $L^2$ mass of $\theta^\kappa$ is being dissipated and thus $D$, as its limit, is the space-time measure indicating where anomalous dissipation is happening. 

Then we ask that $D$ is supported in a set $S \subseteq [0,1] \times \T^2$, taken as small as possible. The conclusion is only novel in the setting that $\gamma<3$, i.e.\ that $S$ has positive (Hausdorff) co-dimension. Otherwise the result reduces to the usual (non-intermittent) Obukhov--Corrsin regularity statement (see~\cite[Section 5]{drivas_anomalous_2022} for a clear and rigorous presentation). 

In total, we see that Theorem~\ref{thm:intermittent-obukov-corrsin} gives additional restrictions on the possible regularity of the passive scalar solution and advecting flow pair, conditional on having non-trivial anomalous dissipation isolated to some positive co-dimension space-time set. These regularity restrictions are intermittent in the sense that possible regularity is highly restricted when $\frac{2}{q} + \frac{1}{p} \ll 1$ but is much less restricted when $\frac{2}{q} + \frac{1}{p} =1$. That is: spatial regularity is highly dependent on integrability.

We note that in our setting, as will be made clear by the construction in Subsection~\ref{ss:flow}, that the dissipation measure associated to our solutions and flows is supported on countably many timeslices $\{t\} \times \T^2$, and as such will have codimension at least $1$. Thus Theorem~\ref{thm:intermittent-obukov-corrsin} does apply to give non-trivial intermittent regularity restrictions on our solutions. However, in Theorem~\ref{thm:intermittent-obukov-corrsin} the spatial and temporal integrability exponents are always matched. While natural for the setting of~\cite{rosa_intermittency_2025}, this makes it difficult to disentangle the notions of ``spatial intermittency'' and ``temporal intermittency'' as discussed above. On the other hand, Theorem~\ref{thm:intermittency}, which follows from more direct analysis of the construction, gives a somewhat clearer picture of the intermittency present in our construction, for which any initial data becomes meaningfully spatially intermittent on some open interval of time going back to the initial time $0$.

We note however that~\cite[Section 5.5]{rosa_intermittency_2025} leaves open the sharpness of Theorem~\ref{thm:intermittent-obukov-corrsin}. In Appendix~\ref{appendix:intermittency}, we show that the construction of~\cite{alberti_exponential_2019}---which also constitutes the central ingredient to our construction of the flow $V$ in Subsection~\ref{ss:flow}---directly verifies the (essential) sharpness of Theorem~\ref{thm:intermittent-obukov-corrsin} in a subset of the parameter range, namely when $p=\infty, \gamma =2$ and $q, \sigma,\beta$ subject only to the restrictions of Theorem~\ref{thm:intermittent-obukov-corrsin}. 

\begin{theorem}
    \label{thm:sharpness-of-theo-for-p-equal-infty}
    For all $q \in [1,\infty], \sigma,\beta \in (0,1)$ such that
    \[\frac{2\beta}{1-\sigma} < 1 - \Big(1 - \frac{2}{q} - \frac{1}{\infty}\Big)(3-2) = \frac{2}{q},\]
    there exists $u \in L^\infty([0,1],B^{\sigma,\infty}_\infty(\T^2))$ with $\nabla \cdot u =0$ and $\theta \in L^q([0,1], B^{\beta,q}_\infty(\T^2))$ such that on $[0,1] \times \T^2$, $\theta$ solves
    \[\partial_t \theta + \nabla \cdot (u \theta) =0,\]
    and associated to $(u,\theta)$ we have the dissipation distribution
    \[D := \partial_t \theta^2 + \nabla \cdot (u \theta^2),\]
    where $D \ne 0$ is a non-trivial negative Radon measure supported on $S:=\{1/2\} \times \T^2$ which has Hausdorff dimension $2$.
\end{theorem}

We note that we show only all ``non-critical'' regularities, not getting up to the case of equality. While this only covers a fairly restricted range of parameters, this result does cover a region of parameters where Theorem~\ref{thm:intermittent-obukov-corrsin} makes non-trivial ``intermittent'' corrections to the usual Obukhov-Corrsin regularity. We also note that we allow $q$ to go down to $1$, going somewhat beyond the range covered by Theorem~\ref{thm:intermittent-obukov-corrsin}. The argument for Theorem~\ref{thm:sharpness-of-theo-for-p-equal-infty}, given in Appendix~\ref{appendix:intermittency}, follows straightforwardly from direct analysis of the construction of~\cite{alberti_exponential_2019} and interpolation. It is also clear from directly studying the construction used that new ideas are needed in order to demonstrate sharpness at the opposite extreme setting of $q =\infty$ and $p \in [1,\infty)$. The construction relies on a singular mixing time, similar to the setting of Proposition~\ref{prop:time-intermittency}, which means that the solution will never live in $L^\infty B^{\beta,\infty}_\infty$ since any positive regularity norm will always diverge at the final time.

\subsection{Open questions}
\label{ss:open-questions}

The primary weakness of the current construction is that the vector field is ``tailor-made'' in order to prove passive scalar turbulence phenomena. A very interesting---though exceedingly difficult---direction is showing phenomena such as anomalous dissipation and Richardson dispersion for ``generic fluid-like models'', though even appropriate models fitting that description are unclear.

There are however still improvements to be made even for synthetically constructed vector fields. One way our tailor-made flow could be a more representative physical model of passive tracer turbulence is making it (more) time homogeneous. The statements of Richardson dispersion and anomalous regularization are both only for solutions started at initial time $t=0$; the same cannot be said if we take an arbitrary initial time $t \in (0,1)$. Another weakness is the lack of sharpness in the regularity attained in the statement of anomalous regularization given by Theorem~\ref{thm:anomalous-regularization}. Ideally, we would get anomalous regularization all the way up to the maximal regularity allowable by Theorem~\ref{thm:intermittent-obukov-corrsin}, or at the very least get $H^{\frac{1-\alpha}{2}-}(\T^2)$ regularity at (almost) all times $t>0.$ 

Demonstrating sharp anomalous regularization as well as Richardson dispersion for a velocity field that is (more) time homogeneous, so that the results can be shown for data started from any initial time, would be a meaningful advance on the results presented here. It is worth noting that the construction and argument of~\cite{armstrong_anomalous_2025} is a likely candidate for solving this problem. That said, the construction presented here has the advantage that the solutions to~\eqref{eq:intro_drift_diffusion_equation} have a much more explicit form than those considered in~\cite{armstrong_anomalous_2025}.

In terms of intermittency, understanding under what conditions we can expect and prove intermittent behavior remains a deep and difficult open problem for more realistic or generic fluid models.

\subsection{Proof of Corollary~\ref{cor:diffusion-enhancement} and Proposition~\ref{prop:time-intermittency}}
\label{ss:intro-proofs}

We quickly prove Corollary~\ref{cor:diffusion-enhancement} as a consequence of Theorem~\ref{thm:anomalous-dissipation}.

\begin{proof}[Proof of Corollary~\ref{cor:diffusion-enhancement}]
    Applying Riesz-Thorin interpolation to Theorem~\ref{thm:anomalous-dissipation} together with the $L^\infty$ and $L^2$ contractivity of advection-diffusion equations and using that the velocity field is taken to be periodic in time, we get that for all $n \in \N$ and $r \in (0,1)$ that
    \[\|\theta^\kappa(n+1+r,\cdot)\|_{L^2(\T^2)} \leq \|\theta^\kappa(n+1,\cdot)\|_{L^2(\T^2)}\leq C \kappa^{(1-\alpha)^2/24} \|\theta^\kappa(n,\cdot)\|_{L^2} = C e^{-C^{-1} \log(\kappa^{-1})} \|\theta^\kappa(n,\cdot)\|_{L^2}.\]
    Iterating this bound, we conclude.
\end{proof}

We now provide the simple proof of Proposition~\ref{prop:time-intermittency}.

\begin{proof}[Proof of Proposition~\ref{prop:time-intermittency}]
    We assume for the sake of contradiction that $\|\theta^{\kappa_j}\|_{L^\infty_t H^s_x} \nrightarrow \infty.$ Then, up to relabelling, we can assume that $\sup_j \|\theta^{\kappa_j}\|_{L^\infty_t H^s_x} < \infty$ and $\theta^{\kappa_j} \stackrel{L^\infty_t L^2_x}{\rightharpoonup} \theta.$  We want to apply the Aubin--Lions lemma to get that $(\theta^{\kappa_j})_{j \in \N}$ is precompact in $C^0_tL^2_x$. We use $H^s_x$ as the ``strong'' space and note that $H^s_x$ embeds compactly into $L^2_x$. To conclude the precompactness in $C^0_t L^2_x$, we note that
    \[\|\dot \theta^{\kappa_j}\|_{L^\infty_t H^{-2}_x} \leq \|\kappa_j \Delta \theta^{\kappa_j}\|_{L^\infty_t H^{-2}_x} + \|\nabla \cdot (u \theta^{\kappa_j})\|_{L^\infty_t H^{-2}_x} \leq (\kappa_j + \|u\|_{L^\infty_{t,x}}) \|\theta^{\kappa_j}\|_{L^\infty_t L^2_x} \leq  (\kappa_j + \|u\|_{L^\infty_{t,x}}) \|\theta_0\|_{ L^2_x} .\]
    Then, as $L^2_x$ embeds continuously into $H^{-2}_x$, Aubin--Lions gives the desired precompactness. Taking perhaps a further subsequence, we have that $\theta^{\kappa_j} \stackrel{C^0_t L^2_x}{\to} \theta$, hence $\theta \in C^0_t L^2_x$, contradicting our hypothesis that $E(t) := \|\theta(t,\cdot)\|_{L^2_x}$ is not continuous.
\end{proof}
\subsection{Acknowledgments}

We would like to thank Scott Armstrong, Theodore Drivas, and Vlad Vicol for illuminating discussions. EHC was partially supported by NSF grant DMS-2342349.

\section{Construction of the flow and overview of the argument}

This construction is essentially a more elaborate version of that in~\cite{hess-childs_universal_2025}. As such, we recommend that a reader highly motivated to completely understand the construction start with~\cite{hess-childs_universal_2025}.  We first recall the following notation from~\cite[Section 2]{hess-childs_universal_2025}. We note we define $B$ (and hence $\T^2$) in such a way so that it is congruent to each of its halves.

\begin{definition}
Let
    \begin{align*}
    B &:= [0,\sqrt{2}] \times [0,1],\\
    \Theta_0(x,y) &:= \indc_{\{x < \sqrt{2}/2\}}(x,y),\\
    \mathcal{R} &:= \frac{1}{\sqrt{2}} \begin{pmatrix} 0 & -1 \\ 1 & 0\end{pmatrix},\\
    A_n&:=\mathcal{R}^nB,\\
    \Lambda_n&:=\mathcal{R}^n\Big\{(\sqrt{2}i,j):0\leq i\leq 2^{\lfloor\frac{n}{2}\rfloor},0\leq j\leq 2^{\lfloor\frac{n+1}{2}\rfloor}\Big\},
    \end{align*}
    and throughout we identify $\T^2$ as $B$ with periodic boundaries.
\end{definition}

\begin{definition}
\label{def:sol-operator}
    Given a vector field $u\in L^\infty([0,\infty)\times B)$, boundary data $f \in L^\infty([0,\infty) \times \partial B)$ and $\kappa> 0$, for any $0 \leq s \leq t$, $\sol_{s,t}^{u,\kappa,f} : TV(B) \to L^1(B)$ denotes the solution operator to
    \begin{equation}\label{eq:drift_diff_equation}
    \begin{cases}
    \partial_t \theta - \kappa \Delta \theta + u \cdot \nabla \theta=0&[s,t]\times B,
    \\ \theta(\cdot,x) = f(\cdot,x)&x\in \partial B.
    \end{cases}
    \end{equation}
    For $\kappa=0$, we use $\sol^{u,0,f}_{s,t}$ to denote the solution to~\eqref{eq:drift_diff_equation}, although in this case we require $u \in L^1([0,\infty), W^{1,\infty}(B))$.
    If $u$ is tangent to $\partial B$, then $\sol_{s,t}^{u,0,f}$ is independent of the boundary data $f$. We denote $\sol^{u,\kappa,\T^2}_{s,t}$ the solution operator to the problem with periodic boundary data.
\end{definition}

We let $v : [0,1] \times B \to \R^2$ be the flow constructed in~\cite[Section 2.2]{hess-childs_universal_2025}, which we note is based on (and is essentially identical to) that of~\cite{alberti_exponential_2019}. For all $a_0, a_1 \in \R$, define 
\begin{equation}
\label{eq:two-cell-data}
\theta_{a_0,a_1} := (a_1 - a_0) \Theta_0 + a_0.
\end{equation}
We note that from~\cite[Section 2]{hess-childs_universal_2025}, the zero-diffusivity transport solution associated to $v$ for the initial data $\theta_{a_0,a_1}$ is uniquely defined (and independent of boundary data as $v$ is tangent to the boundary $\partial B$) and we in particular have that for $t \in [1/2,1],$
\[\sol_{0,t}^{v,0,\T^2} \theta_{a_0,a_1} = \tfrac{a_0+a_1}{2}.\]
That is, $\theta_{a_0,a_1}$ is perfectly mixed by time $\frac{1}{2}$. See Figure~\ref{fig:two-cell} for a diagram of the action of transport by $v$ from time $0$ to $1$. We recall further the following properties of $v$.

\begin{theorem}[{\cite[Lemma 2.8 and Corollary 3.6]{hess-childs_universal_2025}}]
\label{thm:two-cell-dissipator}
    For all $\alpha \in (0,1)$, there exists a flow $v :[0,1] \times B \to \R^2$ such that 
    \begin{enumerate}
        \item $\nabla \cdot v =0,$
        \item $v$ is tangential to the boundary of the box $\partial B$,
        \item $v \in L^\infty([0,1], C^\alpha(B)),$
        \item $v$ extends to a continuous periodic function $\R^2 \to\R^2$,
        \item\label{item:Lipschitz_bound} There exists $M>0$, independent of $\alpha$, so that for all $t \in [0,1/2)$, we have the Lipschitz bound $\|\nabla v(t,\cdot)\|_{L^\infty(B)} \leq \frac{M}{1-\alpha}(1/2-t)^{-1}.$
    \end{enumerate}
    Further, there exists $C(\alpha)>0$ such that for all $f \in L^\infty([0,1] \times \partial B)$, $\kappa \in (0,1)$, and $t \in [0,1]$, we have the bound
    \begin{equation}
    \label{eq:two-cell-dissipation}
    \big\|(\sol^{v,\kappa,f}_{0,t} - \sol_{0,t}^{v,0,f})\theta_{a_0,a_1} \big\|_{L^1(B)} \leq C \big(|a_0 - a_1| +\|f-a_0\|_{L^\infty([0,1] \times \partial B)}\big) \kappa^{(1-\alpha)/12} |t-\tfrac{1}{2}|^{-1/2}.
    \end{equation}
\end{theorem}

\begin{figure}[htbp]
    \centering
\begin{tikzpicture}

% Define the overall rectangle's dimensions
\def\width{0.88*0.3535}  % bottom length
\def\height{0.88*0.5}        % side length

\fill[black!100!white] (33*\width, 0) rectangle (37*\width,4*\height);

\fill[black!0!white] (37*\width, 0) rectangle (41*\width,4*\height);

\draw[thick] (33*\width, 0) rectangle (41*\width, 4*\height);

\draw[->,very thick,black] (43*\width,2*\height) -- (46*\width,2*\height);

%1 cell

\fill[black!50!white] (48*\width, 0) rectangle (56*\width,4*\height);

\draw[thick] (48*\width, 0) rectangle (56*\width, 4*\height);

\end{tikzpicture}
    \caption{The action of transport by $v$.}
    \label{fig:two-cell}
\end{figure}
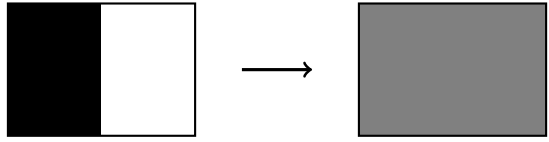

The degeneracy in~\eqref{eq:two-cell-dissipation} as $t \to \frac{1}{2}$ does not explicitly appear in~\cite{hess-childs_universal_2025}. However, for $t \leq \frac{1}{2}$, the bound is direct from using~\cite[Proposition 3.1]{hess-childs_universal_2025} to control times up to $t_n$ where $\frac{1}{2} - t_n \approx \kappa^{\frac{3(1-\alpha)}{2(\alpha+2)}}$ and then using the trivial bound $\big\|\sol^{v,\kappa,f}_{0,t} \theta_0 - \sol_{0,t}^{v,0,f}\big\|_{L^1(B)}  \leq 2$ for $t \in [t_n,\frac{1}{2}].$ For $t \geq \frac{1}{2}$, the degeneracy is direct from inspecting~\cite[Proof of Theorem 2.4]{hess-childs_universal_2025}. Additionally, Item~\ref{item:Lipschitz_bound} follows from inspecting the construction of $v$ and using that the flow $U$ in~\cite[Theorem 2.6]{hess-childs_universal_2025} is in $L^\infty([0,\infty), W^{1,\infty}(B)).$

\subsection{Definition of the flow}

\label{ss:flow}

We now define the flow $V$ that we will use throughout this paper. It is built using (rescaled in time and space copies of) $v$ as the essential building block. We first introduce the two-parameter sequence of times $s^i_j$ that we will use to define $V$.

We note that throughout the paper, dependence on $\alpha \in (0,1)$ is often suppressed, though we note that $V,\sigma_j,$ and $s^i_j$ depend explicitly on $\alpha$ and hence objects like $\sol^{V,\kappa,\T^2}_{0,t}$ depend implicitly on $\alpha$. All prefactor constants $C>0$ will be allowed to depend on $\alpha$, though we keep $\alpha$ dependence in all exponents, such as $\kappa^{(1-\alpha)^2/12}$, explicit.

\begin{definition}
\label{defn:big-parameter-defn}
    For $j \in \N$ and $i \leq j$, let
    \begin{align*}
    \sigma_j &:= \frac{2^{-(1-\alpha)j/2}}{\nconstant},\\
    s^i_j &:= 1 - \sum_{k<i} \Big(\sigma_k +  \sum_{k \leq \ell } \sigma_\ell\Big) - \sum_{i \leq \ell < j} \sigma_\ell,\\
    s^i_\infty &:= 1 - \sum_{k<i} \Big(\sigma_k +  \sum_{k \leq \ell } \sigma_\ell\Big) - \sum_{i \leq \ell} \sigma_\ell,\\
    Z &:= \sum_{k=0}^\infty \Big(2^{-(1-\alpha)k/2} + \sum_{k \leq \ell} 2^{-(1-\alpha)\ell/2} \Big),\\
    \gamma &:= \frac{1}{8(M+7)},
    \end{align*}
    where the $M$ in the definition of $\gamma$ is as in Theorem~\ref{thm:two-cell-dissipator}.
\end{definition}

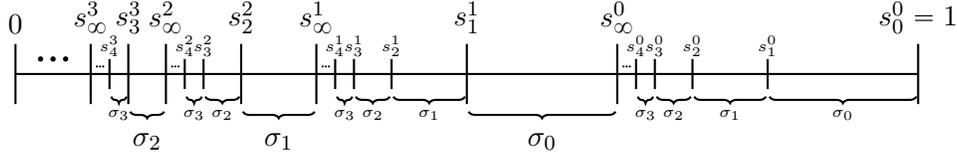
\begin{figure}[htbp]
    \centering

\begin{tikzpicture}[x=12cm, y=1cm]
  % Draw the unit interval
  \draw[thick] (0,0) -- (1,0);
  \draw[thick] (0,0.4) node[above] {$0$} -- (0,-0.4);
  \draw[thick] (1,0.4) node[above] {$s^0_0=1$} -- (1,-0.4);

  \pgfmathsetmacro{\ellipa}{2^(-2)/9}
  \pgfmathsetmacro{\ellipb}{2^(-2)/6}
  \pgfmathsetmacro{\ellipc}{2^(-1)/9}

  \fill (\ellipa, 0.2) circle (1pt);
  \fill (\ellipb, 0.2) circle (1pt);
  \fill (\ellipc, 0.2) circle (1pt);

  \pgfmathsetmacro{\ellipa}{2/3+2^(-4)/9}
  \pgfmathsetmacro{\ellipb}{2/3+2^(-4)/6}
  \pgfmathsetmacro{\ellipc}{2/3+2^(-3)/9}

  \fill (\ellipb, 0.1) circle (0.4pt);
  \fill (\ellipa, 0.1) circle (0.4pt);
  \fill (\ellipc, 0.1) circle (0.4pt);

  \pgfmathsetmacro{\ellipa}{2/6+2^(-4)/9}
  \pgfmathsetmacro{\ellipb}{2/6+2^(-4)/6}
  \pgfmathsetmacro{\ellipc}{2/6+2^(-3)/9}

  \fill (\ellipb, 0.1) circle (0.4pt);
  \fill (\ellipa, 0.1) circle (0.4pt);
  \fill (\ellipc, 0.1) circle (0.4pt);

  \pgfmathsetmacro{\ellipa}{2/12+2^(-4)/9}
  \pgfmathsetmacro{\ellipb}{2/12+2^(-4)/6}
  \pgfmathsetmacro{\ellipc}{2/12+2^(-3)/9}

  \fill (\ellipb, 0.1) circle (0.4pt);
  \fill (\ellipa, 0.1) circle (0.4pt);
  \fill (\ellipc, 0.1) circle (0.4pt);

  \pgfmathsetmacro{\ellipa}{2/24+2^(-4)/9}
  \pgfmathsetmacro{\ellipb}{2/24+2^(-4)/6}
  \pgfmathsetmacro{\ellipc}{2/24+2^(-3)/9}

  \fill (\ellipb, 0.1) circle (0.4pt);
  \fill (\ellipa, 0.1) circle (0.4pt);
  \fill (\ellipc, 0.1) circle (0.4pt);

  % Ticks and labels for i = 0 to 4
  \foreach \i in {1,2,3} {
    \pgfmathsetmacro{\powtwo}{2^(-\i)}

    % Tick and label for s^i_i
    \draw[thick] (\powtwo,0.4) node[above]{$s^{{\i}}_{{\i}}$} -- (\powtwo,-0.4) ;
    }
\foreach \i in {0,1,2,3} {
    \pgfmathsetmacro{\fracpowtwo}{2^(1 - \i)/3}

    % Tick and label for s^i_\infty
    \draw[thick] (\fracpowtwo,0.4) node[above]{$s^{{\i}}_{{\infty}}$} -- (\fracpowtwo,-0.4);
    
  }

  \foreach \j in {1,2,3,4} {

  \pgfmathsetmacro{\powtwo}{2/3+2^(-\j)/3}

  \draw[thick] (\powtwo,0.2) node[above,yshift=-3pt]{\tiny $s^0_{\j}$} -- (\powtwo,-0.2);

  }

\foreach \j in {2,3,4} {

  \pgfmathsetmacro{\powtwo}{2/6+2^(-\j+1)/6}

  \draw[thick] (\powtwo,0.2) node[above,yshift=-3pt]{\tiny $s^1_{\j}$} -- (\powtwo,-0.2);

  }

\foreach \j in {3,4} {

  \pgfmathsetmacro{\powtwo}{2/12+2^(-\j+2)/12}

  \draw[thick] (\powtwo,0.2) node[above,yshift=-3pt]{\tiny $s^2_{\j}$} -- (\powtwo,-0.2);

  }

\foreach \j in {4} {

  \pgfmathsetmacro{\powtwo}{2/24+2^(-\j+3)/24}

  \draw[thick] (\powtwo,0.2) node[above,yshift=-3pt]{\tiny $s^3_{\j}$} -- (\powtwo,-0.2);

  }

\draw[decorate, decoration={brace, mirror}, thick]
    (0.5+0.002, -0.5) -- (0.6667-0.002, -0.5)
    node[midway, below=4pt] {$\sigma_0$};

\draw[decorate, decoration={brace, mirror}, thick]
  (0.25+0.002, -0.5) -- (0.3333-0.002, -0.5)
  node[midway, below=4pt] {$\sigma_1$};

\draw[decorate, decoration={brace, mirror}, thick]
  (0.125+0.002, -0.5) -- (0.1667-0.002, -0.5)
  node[midway, below=4pt] {$\sigma_2$};

\draw[decorate, decoration={brace, mirror}, thick]
  (5/6+0.002, -0.3) -- (1-0.002, -0.3)
  node[midway, below=1pt] {\tiny $\sigma_0$};

\draw[decorate, decoration={brace, mirror}, thick]
  (0.75+0.002, -0.3) -- (5/6-0.002, -0.3)
  node[midway, below=1pt] {\tiny $\sigma_1$};

\draw[decorate, decoration={brace, mirror}, thick]
  (17/24+0.002, -0.3) -- (3/4-0.002, -0.3)
  node[midway, below=1pt] {\tiny $\sigma_2$};

\draw[decorate, decoration={brace, mirror}, thick]
(33/48+0.002, -0.3) -- (17/24-0.002, -0.3)
  node[midway, below=1pt] {\tiny $\sigma_3$};

\draw[decorate, decoration={brace, mirror}, thick]
  (5/12+0.002, -0.3) -- (0.5-0.002, -0.3)
  node[midway, below=1pt] {\tiny $\sigma_1$};

  \draw[decorate, decoration={brace, mirror}, thick]
  (3/8+0.002, -0.3) -- (5/12-0.002, -0.3)
  node[midway, below=1pt] {\tiny $\sigma_2$};

\draw[decorate, decoration={brace, mirror}, thick]
  (17/48+0.002, -0.3) -- (3/8-0.002, -0.3)
  node[midway, below=1pt] {\tiny $\sigma_3$};

\draw[decorate, decoration={brace, mirror}, thick]
  (5/24+0.002, -0.3) -- (1/4-0.002, -0.3)
  node[midway, below=1pt] {\tiny $\sigma_2$};

\draw[decorate, decoration={brace, mirror}, thick]
  (3/16+0.002, -0.3) -- (5/24-0.002, -0.3)
  node[midway, below=1pt] {\tiny $\sigma_3$};

\draw[decorate, decoration={brace, mirror}, thick]
  (5/48+0.002, -0.3) -- (1/8-0.002, -0.3)
  node[midway, below=1pt] {\tiny $\sigma_3$};
\end{tikzpicture}

 \caption{The definition of the singular times $s^i_j$ and $s^i_\infty$.}\label{fig:s^i_j-def}
\end{figure}
    
Note that for all $i \in \N$, $\lim_{j \to \infty} s^i_j = s^i_\infty$, $Z$ is chosen so that $\lim_{i \to\infty} s^i_j = 0$ uniformly in $j$, and the collection of intervals
\[\big\{ [s^i_{j+1}, s^i_j)  : i \in \N, i \leq j\big\} \cup \big\{[s^{i+1}_{i+1}, s^i_\infty] : i \in \N\big\}\]
form a partition of the interval $(0,1)$. We will thus define $V$ separately on each of these intervals.

\begin{definition}
    Letting $v$ be given by in Theorem~\ref{thm:two-cell-dissipator} and identifying it with its periodic extension on $\R^2$, we define the divergence-free vector field $V: [0,1] \times \T^2 \to \R^2$ by
    \begin{equation}
        \label{eq:V-def}
        V(t,x) := 
        \begin{cases}
            \sigma_j^{-1}\mathcal{R}^{j}v\big(\sigma_j^{-1} (t-s^i_{j+1}), \mathcal{R}^{-j} x\big) & t \in [s^i_{j+1}, s^i_j), i \in \N, i \leq j,\\
            0 & t \in [s^{i+1}_{i+1}, s^i_\infty], i \in \N.
        \end{cases}
    \end{equation}
\end{definition}

The following is direct from the scaling of $v$ used to define $V$.
\begin{lemma}
    For $V$ defined by~\eqref{eq:V-def}, $V \in L^\infty([0,1], C^\alpha(\T^2)).$
\end{lemma}

The motivation for defining $V$ in this way should become clear throughout the remainder of this section. A primary motivation for the definition is the particularly simple behavior of the vector field in the vanishing diffusivity limit, described by the limiting solution operator given by Definition~\ref{def:limiting-soln-op}.

\subsection{Overview of the argument}
\label{ss:overview}

The first, and most technically involved, step of the argument is understanding the behavior of the solution operator $\sol^{V,\kappa,\T^2}_{0,t}$ under the vanishing diffusivity limit $\kappa \to 0.$ To this end, we construct a \textit{limiting solution operator} $\sola_t$ in Definition~\ref{def:limiting-soln-op}. On $(0,1],\, \sola_t$ follows the zero-diffusivity ODE trajectories up until singular ``perfect mixing times'', after which the solution is replaced with its local averages; the action of $\sola_t$ is depicted in Figure~\ref{fig:Sol_t}. In Section~\ref{s:convergence-to-ideal}, we (essentially) show that as $\kappa \to 0, \sol^{V,\kappa,\T^2}_{0,t}\to \sola_t$ in a zero regularity sense (e.g.\ as operators $L^2(\T^2) \to L^2(\T^2)$). This convergence alone is enough to immediately conclude the anomalous total dissipation of Theorem~\ref{thm:anomalous-dissipation}. We will discuss the details of the proof of convergence in Subsection~\ref{ss:convergence}, though we note that the argument is essentially a refinement of the argument of~\cite{hess-childs_universal_2025}.

That said, there are two key differences in this work that are necessary in order to prove the statements of Richardson dispersion in Theorem~\ref{thm:richardson} and anomalous regularization in Theorem~\ref{thm:anomalous-regularization}. The first is that we are more careful in our ``bookkeeping''. We track pointwise bounds of the difference of the solution operators $\sol^{V,\kappa, \T^2}_{0,t} - \sola_t$ at all times $t \in [0,1]$ instead of just the final time $t=1$ and maintain sharper estimates throughout. This allows us to retain nontrivial bounds all the way down to $t \approx \kappa^{\frac{1-\alpha}{1+\alpha}}$, the relevant timescale for Richardson dispersion. As well, good pointwise estimates are an essential ingredient for lifting the statement of anomalous dissipation to one of anomalous regularization. 

The second key difference with~\cite{hess-childs_universal_2025} is a difference in the construction of the velocity field $V$. The velocity field in~\cite{hess-childs_universal_2025} starts with a ``pause'', an open time interval during which the velocity field is $0$. This pause necessarily factored into the estimates, as we used the smoothing of the heat kernel in this time to provide an initial regularity of the solution prior to the flow meaningfully acting. However, such a pause entirely rules out the possibility of an anomalous regularization estimate of the form of the one given in Theorem~\ref{thm:anomalous-regularization}, since, as $\kappa \to 0$, there will be no regularization on this time interval, and so the $L^2([0,1], H^s(\T^2))$ norm of any zero-regularity initial date will diverge as $\kappa \to 0.$ Similarly, the presence of the pause prevents an estimate of Richardson dispersion like that of Theorem~\ref{thm:richardson}, since on the pause interval, the variance must only be that of pure diffusion, $\kappa t$, as opposed to the ``super-ballistic'' growth, $t^{\frac{2}{1-\alpha}}$.

Thus, as is shown in Figure~\ref{fig:s^i_j-def}, we get rid of the initial pause here and rather have the velocity field $V$ have non-trivial action immediately. This creates additional difficulties in proving the convergence of $\sol^{V,\kappa,\T^2}_{0,t} \to \sola_t$, but we defer the discussion of these to Subsection~\ref{ss:convergence}. 
 
Following the proof that $\sol^{V,\kappa,\T^2}_{0,t}$ converges in the appropriate zero-regularity sense to $\sola_t$, in Section~\ref{s:richardson} we prove the statement of Richardson dispersion given by Theorem~\ref{thm:richardson}. The main idea is that by directly analyzing $\sola_t$, which has an explicit representation, we see the correct variance growth for Richardson dispersion in the zero-diffusivity limit. Thus to conclude, it suffices to show that the true solution, given by the solution operator $\sol^{V,\kappa,\T^2}_{0,t}$, is sufficiently close to the ideal solution and that this closeness causes the variance of the true solution to also have the right size. This is somewhat subtle as we need to show closeness of the two solution operators at the time $t \approx \kappa^{\frac{1-\alpha}{1+\alpha}} \to 0$ and our estimate on the closeness of the solution operators degenerates as $t \to 0$. However, we show sufficiently sharp bounds in Theorem~\ref{thm:limit-is-close-pointwise-ell-dependent} to take the diagonal limit and conclude.

We next turn our attention to anomalous regularization. In Section~\ref{s:fractional-regularity-spaces}, we introduce the tools of interpolation of fractional regularity spaces that are essential to our argument. In Section~\ref{s:anomalous-regularization} we proceed to the proof, which we now sketch. Fix some $\theta_0 \in L^2(\T^2)$ with $\|\theta_0\|_{L^2} = 1$ and let $\theta^\kappa : [0,1] \times \T^2 \to \R$ be the associated solution to~\eqref{eq:intro_drift_diffusion_equation}, that is $\theta^\kappa(t,\cdot) = \sol^{V,\kappa,\T^2}_{0,t} \theta_0(\cdot).$ We then let $\theta^0(t,\cdot) := \sola_t \theta_0(\cdot)$, so that $\theta^\kappa \to \theta^0$ as $\kappa \to 0$. The idea to prove the anomalous regularization is to combine three facts. The first is the quantitative bound with algebraic rate (provided by Corollary~\ref{cor:is-close-to-ideal-pointwise}),
\begin{equation}
\label{eq:overview-interpolation-bound-1}
\|\theta^\kappa - \theta^0\|_{L^2_{t,x}} \leq C \kappa^{(1-\alpha)^2/96}.
\end{equation}
The second, by the explicit form of the limiting solution operator $\sola_t$, and as proved in Proposition~\ref{prop:ideal-soln-operator-regularity} using interpolation tools, is that
\[\|\theta^0\|_{L^2_t H^{\frac{1-\alpha}{8(M+1)}}_x} \leq C,\]
where $M$ is as in Theorem~\ref{thm:two-cell-dissipator}. The third, which follows from the standard energy identity for drift-diffusion equations, is that
\begin{equation}
\label{eq:intro-energy-identity}
\|\theta^\kappa\|_{L^2_tH^1_x} \leq C\kappa^{-1/2}.
\end{equation}
Then for $s \in (0,1)$, we write using interpolation,
\begin{equation}
\label{eq:overview-interpolation-bound-2}
\|\theta^\kappa\|_{L^2_t H^{s\frac{1-\alpha}{8(M+1)}}_x} = \|\theta^0\|_{L^2_t  H^{s\frac{1-\alpha}{8(M+1)}}} + \|\theta^\kappa - \theta^0\|_{L^2_t  H^{s\frac{1-\alpha}{8(M+1)}}_x} \leq C+ C\|\theta^\kappa - \theta^0\|_{L^2_{t,x}}^{1-s} \|\theta^\kappa - \theta^0\|_{L^2_t  H^{\frac{1-\alpha}{8(M+1)}}_x}^s.
\end{equation}
Then we note that
\begin{equation}
\label{eq:overview-interpolation-bound-3}
\|\theta^\kappa - \theta^0\|_{L^2_t  H^{\frac{1-\alpha}{8(M+1)}}_x} \leq \|\theta^\kappa\|_{L^2_t  H^{\frac{1-\alpha}{8(M+1)}}_x} +\|\theta^0\|_{L^2_t  H^{\frac{1-\alpha}{8(M+1)}}_x} \leq \|\theta^\kappa\|_{L^2_{t,x}}^{1-\frac{1-\alpha}{8(M+1)}}  \|\theta^\kappa\|_{L^2_t  H^1_x}^{\frac{1-\alpha}{8(M+1)}} +C \leq C \kappa^{-\frac{1-\alpha}{16(M+1)}}.
\end{equation}
Combining~\eqref{eq:overview-interpolation-bound-1},~\eqref{eq:overview-interpolation-bound-2}, and~\eqref{eq:overview-interpolation-bound-3} and choosing $s$ correctly, we see that for the appropriate $s \in (0,1)$, we have that $\|\theta^\kappa\|_{L^2_t H^{s\frac{1-\alpha}{8(M+1)}}_x}  \leq C$ for all $\kappa \in (0,1)$, exactly as desired.

We note here that it is absolutely essential that we have an algebraic rate of convergence in~\eqref{eq:overview-interpolation-bound-1} in order to compete with the algebraic rate of blow up in~\eqref{eq:intro-energy-identity}.

Finally, let us discuss the proof of intermittent regularity as given by Theorem~\ref{thm:intermittency}, which is completed at the end of Section~\ref{s:anomalous-regularization}. With $ \theta^\kappa$ and $\theta^0$ as above, we have that for all $t \in (0,1)$, $\theta^\kappa(t,\cdot) \stackrel{L^2}{\rightharpoonup} \theta^0(t,\cdot).$  Then, by norm lower semicontinuity, it suffices to show that for all $t < t_*$, $\|\theta^0(t,\cdot)\|_{H^{\beta,p}(\T^2)}  = \infty$ with $t_*$ the largest time for which $\theta^0 \ne 0$. The explicit form of $\sola_t$ shows that for $t< t_*$, $\theta^0$ is piecewise constant with codimension $1$ interfaces. A function of that form must have $\|\theta^0(t,\cdot)\|_{H^{\beta,p}(\T^2)}  = \infty$, as provided by Proposition~\ref{prop:piecewise-constant-implies-not-in-some-spaces}, allowing us to conclude.

\subsection{Overview of convergence to the limiting solution operator}
\label{ss:convergence}

We now discuss in more detail how we prove the convergence to the limiting solution operator, which we define shortly. We first need the following notation.
\begin{definition}
\label{def:Pi}
    For $j \in \N$, $\mathcal{F}_j$ denotes the linear space of functions that are piecewise constant on the boxes $\{A_j + x_j\}_{x_j \in \Lambda_j}$. That is,
    \[\mathcal{F}_j := \mathrm{span}\big\{\indc_{A_j + x_j} : x_j \in \Lambda_j\big\}.\]

    For $j \in \N$, we define $\Pi_j: L^2(\T^2) \to L^2(\T^2)$ to be the orthogonal projection onto $\mathcal{F}_j$. Note that $\Pi_j$ extends to an operator $L^1(\T^2) \to L^1(\T^2)$ and $\int \Pi_j\theta(x)\,dx =\int\theta(x)\,dx$.
\end{definition}

We note that $\Pi_j$ acts by replacing $\theta(x)$ with its average on each box $A_j + x_j$. That is, for $x \in (A_j + x_j)^\circ$ for some $x_j \in \Lambda_j$, we have that 
 \[\Pi_j \theta(x) = \frac{1}{|A_j|}\int_{A_j + x_j} \theta(y)\,dy.\]
It also holds that the projectors are consistent in the sense that if $0 \leq j \leq n <\infty$, then
\[\Pi_j \Pi_n = \Pi_j = \Pi_n \Pi_j.\]

We now use $\Pi_j$ to define the limiting solution operator $\sola_t$. It will then be the central goal of Section~\ref{s:convergence-to-ideal}, culminating in Theorem~\ref{thm:limit-is-close-pointwise-ell-dependent} and Corollary~\ref{cor:is-close-to-ideal-pointwise} below, to show that $\sol^{V,\kappa,\T^2}_{0,t} \to\sola_t$ as $\kappa\rightarrow 0$. 
 
\begin{definition}
\label{def:limiting-soln-op}
    For $t \in [0,1]$, we define the \textit{limiting solution operator} $\sola_t : L^1(\T^2) \to L^1(\T^2)$ as follows.
    \begin{equation*}
        \sola_t  := \begin{cases}
        \Pi_i& t\in[s^i_{i+1}+\sigma_i/2,s_i^i],\\
        \Pi_{i+1}&t\in[s^{i+1}_{i+1},s^i_{i+1}],\\
        \sol^{V,0,\T^2}_{s^i_{i+1},t}\Pi_{i+1} &t\in[s^i_{i+1},s^i_{i+1}+\sigma_i/2).\\
        \end{cases}
    \end{equation*}
    See Figure~\ref{fig:Sol_t} for a diagram of the action of $\sola_t$.
\end{definition}

\begin{figure}[htbp]
    \centering
\begin{tikzpicture}[x=1cm, y=8cm]

% === Original macros (preserved) ===
\def\width{0.9*0.3535}  % ≈ 0.31815
\def\height{0.9*0.5}    % = 0.45

% === Number line and annotations ===
\draw[thick] (0,0) -- (0,-1);
\draw[thick] (-0.4,0) node[left] {$s^0_0=1$} -- (0.4,0);
\draw[thick] (-0.4,-1) node[left] {$0$} -- (0.4,-1);

\pgfmathsetmacro{\mid}{1}
\pgfmathsetmacro{\rightb}{11/12}
\pgfmathsetmacro{\leftb}{5/6}

\draw[thick] (-0.4,\leftb-1) node[left]{$s^0_1$} -- (0.4,\leftb-1);
\draw[thick] (-0.4,\rightb-1) node[left]{\tiny $s^0_1+\sigma_0/2$} -- (0.4,\rightb-1);

\draw[decorate, decoration={brace}, thick]
(0.5,\mid-1-0.002) -- (0.5,\rightb-1+0.002) node[midway, right=1pt] {$\Pi_0$}; 
\draw[decorate, decoration={brace}, thick]
(0.5,\rightb-1-0.002) -- (0.5,\leftb-1+0.002) node[midway, right=1pt] {$\mathcal{T}_{s^0_1,t}^{V,0,\mathbb{T}^2}\Pi_1$};

\pgfmathsetmacro{\mid}{2^(-1)}
\pgfmathsetmacro{\rightb}{11*2^(-1)/12}
\draw[decorate, decoration={brace}, thick]
(0.5,\leftb-1-0.002) -- (0.5,\rightb-1+0.002) node[midway, right=1pt] {$\Pi_1$};

\pgfmathsetmacro{\leftb}{5*2^(-1)/6}
\draw[thick] (-0.4,\leftb-1) node[left]{$s^1_2$} -- (0.4,\leftb-1);
\draw[thick] (-0.4,\rightb-1) node[left]{\tiny $s^1_2+\sigma_1/2$} -- (0.4,\rightb-1);

\pgfmathsetmacro{\mid}{2^(-1)}
\draw[thick] (-0.4,\mid-1) node[left]{$s^1_1$} -- (0.4,\mid-1);

\draw[decorate, decoration={brace}, thick]
(0.5,\rightb-1-0.002) -- (0.5,\leftb-1+0.002) node[midway,right=1pt] {$\mathcal{T}_{s^1_2,t}^{V,0,\mathbb{T}^2}\Pi_2$};

\pgfmathsetmacro{\mid}{2^(-2)}
\draw[decorate, decoration={brace}, thick]
(0.5,\leftb-1-0.002) -- (0.5,\mid-1+0.002) node[midway, right=1pt] {$\Pi_2$}; 
\draw[thick] (-0.4,\mid-1) node[left]{$s^2_2$} -- (0.4,\mid-1);

\pgfmathsetmacro{\ellipa}{\mid/2-1+0.03} 
\pgfmathsetmacro{\ellipb}{\mid/2-1}
\pgfmathsetmacro{\ellipc}{\mid/2-1-0.03}

\fill (0.2,\ellipa) circle (1pt);
\fill (0.2,\ellipb) circle (1pt);
\fill (0.2,\ellipc) circle (1pt);

% === Boxes (drawn to the right, not warped) ===

  % Define box dimensions
  \def\width{0.3535*16/14}
  \def\height{1/14}

  % Vertical spacing between groups
  \def\dy{5*\height}

  % Total height from top of top box to bottom
  \def\totalshift{14*\height}

  % Shift so top-left of top box is at (2,0)
  \begin{scope}[shift={(6,-\totalshift)}]

    %%%%% Group 1: 1 cell (top)
    \def\yoffset{2*\dy}
    \fill[black!35!white] (0, \yoffset) rectangle (8*\width, \yoffset + 4*\height);
    \draw[thick]           (0, \yoffset) rectangle (8*\width, \yoffset + 4*\height);

    %%%%% Group 2: 2 cells (middle)
    \def\yoffset{1*\dy}
    \fill[black!60!white] (0, \yoffset) rectangle (4*\width, \yoffset + 4*\height);
    \fill[black!10!white] (4*\width, \yoffset) rectangle (8*\width, \yoffset + 4*\height);
    \draw[thick]           (0, \yoffset) rectangle (8*\width, \yoffset + 4*\height);

    %%%%% Group 3: 4 cells (bottom)
    \def\yoffset{0}
    \fill[black!35!white] (0, \yoffset) rectangle (4*\width, \yoffset + 2*\height);
    \fill[black!85!white] (0, \yoffset + 2*\height) rectangle (4*\width, \yoffset + 4*\height);
    \fill[black!15!white] (4*\width, \yoffset) rectangle (8*\width, \yoffset + 2*\height);
    \fill[black!5!white]  (4*\width, \yoffset + 2*\height) rectangle (8*\width, \yoffset + 4*\height);
    \draw[thick]           (0, \yoffset) rectangle (8*\width, \yoffset + 4*\height);

% Coordinates of the midpoint of the \Pi_0 brace (in scope coordinates)
\pgfmathsetmacro{\bracex}{-4.8} % 0.5 - 2 (horizontal shift)
\pgfmathsetmacro{\bracey}{-1/24 + 14*\height} % vertical position: midpoint of brace + total shift

% Coordinates of top-left and bottom-left corners of the top box
\pgfmathsetmacro{\topboxy}{2*\dy + 4*\height}
\pgfmathsetmacro{\botboxy}{2*\dy}

% Draw dashed lines from brace midpoint to top/bottom of top box
\draw[dashed, thick] (\bracex, \bracey) -- (0, \topboxy);
\draw[dashed, thick] (\bracex, \bracey) -- (0, \botboxy);

% Horizontal position of Pi_1 label (x = 0.5 in original figure)
\pgfmathsetmacro{\bracexone}{-4.8}  % since 0.5 - 2 = -1.5

% Correct y-coordinate where Pi_1 is placed
\pgfmathsetmacro{\braceyone}{-0.3542 + 14*\height}

% Top and bottom of middle box
\pgfmathsetmacro{\topmidboxy}{1*\dy + 4*\height}
\pgfmathsetmacro{\botmidboxy}{1*\dy}

% Dashed lines
\draw[dashed, thick] (\bracexone, \braceyone) -- (0, \topmidboxy);
\draw[dashed, thick] (\bracexone, \braceyone) -- (0, \botmidboxy);

% Horizontal coordinate: x = 0.5 in original => -1.5 in shifted scope
\pgfmathsetmacro{\bracextwo}{-4.8}

% Vertical coordinate of Pi_2 label
\pgfmathsetmacro{\braceytwo}{-0.6667 + 14*\height}

% Top and bottom of bottom box
\pgfmathsetmacro{\topbotboxy}{0*\dy + 4*\height}
\pgfmathsetmacro{\botbotboxy}{0}

% Draw dashed lines
\draw[dashed, thick] (\bracextwo, \braceytwo) -- (0, \topbotboxy);
\draw[dashed, thick] (\bracextwo, \braceytwo) -- (0, \botbotboxy);

  \end{scope}

\end{tikzpicture}
    \caption{The definition of $\sola_t$ and $\sola_t\theta_0$ for a specific initial data.}
    \label{fig:Sol_t}
\end{figure}
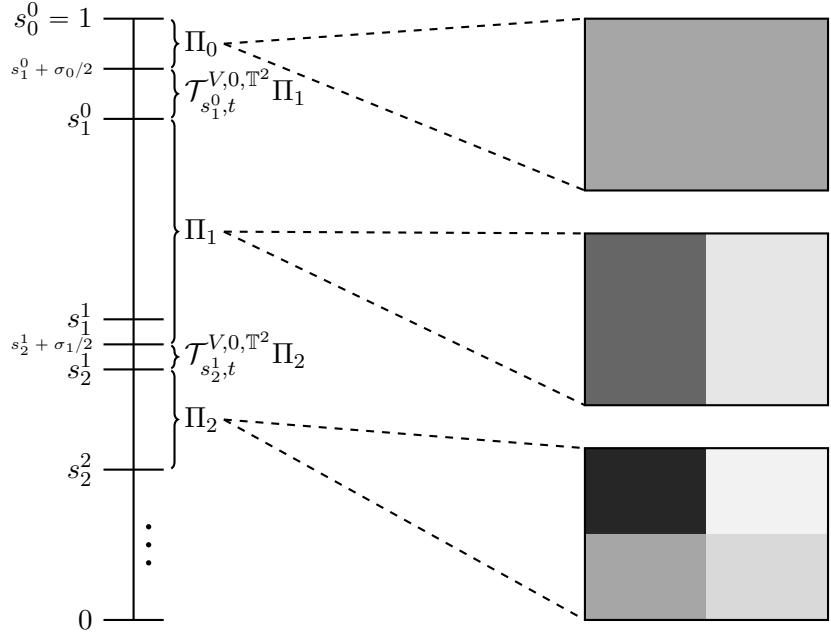

We will also use the following notation for the heat semigroup.

\begin{definition}
    \label{def:Phi-defn}
    For all $t \geq 0$, we denote the heat semigroup by $e^{t\Delta}$, so that $e^{t\kappa\Delta} = \sol^{0,\kappa,\T^2}_{0,t}$.
\end{definition}

We now give the first statement that (loosely) $\sol^{V,\kappa,\T^2}_{0,t} \stackrel{\kappa \to 0}{\to} \sola_t$. We note that Theorem~\ref{thm:anomalous-dissipation} is a direct corollary of the following result, though Theorem~\ref{thm:limit-is-close-pointwise-ell-dependent} will be more useful for further results.

\begin{theorem}
\label{thm:limit-is-close-pointwise-ell-dependent}
    There exists $C(\alpha)>0$ such that for all $t \in [0,1]$ and $ \kappa \in (0,1)$, if $t \in [s^{i+1}_{i+1}, s^i_i]$ for $i \in \N$ and $\ell > i$, we have the bound
    \[\Big\|\sol^{V,\kappa,\T^2}_{0,t} -\sola_t e^{\kappa \sigma_\ell \Delta}\sol^{V,\kappa,\T^2}_{0,s^{\ell+1}_{\ell+1}}\Big\|_{TV(\T^2) \to L^1(\T^2)} \leq \begin{cases}
      C (2^{\ell(1+\alpha)/2} \kappa)^{\frac{(1-\alpha)^2}{12}}  & t \in [s^{i+1}_{i+1},s^i_{i+1}],\\
         C (2^{\ell(1+\alpha)/2} \kappa)^{\frac{(1-\alpha)^2}{12}}\Big|\frac{t- s^i_{i+1} - \sigma_i/2}{\sigma_i}\Big|^{-1/2}& t \in [s^i_{i+1}, s^i_{i}].
    \end{cases} \]
\end{theorem}

From Theorem~\ref{thm:limit-is-close-pointwise-ell-dependent}, we also get the following simpler statement, which controls a smaller set of times and more regular initial data.

\begin{corollary}
    \label{cor:is-close-to-ideal-pointwise}
    There exists $C(\alpha)>0$ such that for all $\kappa \in (0,1)$ and $t \in [\kappa^{(1-\alpha)/8},1] $, we have the bound
    \begin{equation}\label{eq:soln-operator-close-to-ideal}
    \big\|\sol^{V,\kappa,\T^2}_{0,t} -\sola_t\big\|_{L^2(\T^2) \to L^2(\T^2)} \leq \begin{cases}
     C\kappa^{(1-\alpha)^2/48}  & t \in [s^{i+1}_{i+1},s^i_{i+1}],\\
         C\kappa^{(1-\alpha)^2/48}\Big|\frac{t- s^i_{i+1} - \sigma_i/2}{\sigma_i}\Big|^{-1/4}& t \in [s^i_{i+1}, s^i_{i}].
    \end{cases}
    \end{equation}
    In particular, for all $\kappa\in(0,1)$ it holds that
    \begin{equation}\label{eq:soln-operator-close-to-ideal-integrated}
    \big\|\sol^{V,\kappa,\T^2}_{0,t} -\sola_t\big\|_{L^2(\T^2) \to L^2([0,1],L^2(\T^2))} \leq C \kappa^{\frac{(1-\alpha)^2}{96}}.
    \end{equation}
\end{corollary}

Before moving on to the sketch of the argument for Theorem~\ref{thm:limit-is-close-pointwise-ell-dependent} and Corollary~\ref{cor:is-close-to-ideal-pointwise}, let us interpret their statements and explain why we need both. Theorem~\ref{thm:limit-is-close-pointwise-ell-dependent} is much more optimized in how small we can take $t$ and still get a nontrivial bound. This is the statement we will need in order to establish the Richardson dispersion result as we will need to go all the way down to the timescale where the advection and diffusion are of the same magnitude, $t \approx \kappa^{\frac{1-\alpha}{1+\alpha}}$. Inspecting the bound of Theorem~\ref{thm:limit-is-close-pointwise-ell-dependent} shows that at this timescale, the error we have is $O(1)$; however by taking $t \geq K \kappa^{\frac{1-\alpha}{1+\alpha}}$, we can make the error as small as we want, allowing us to conclude the statement of variance growth needed for Richardson dispersion. 

On the other hand, Corollary~\ref{cor:is-close-to-ideal-pointwise} has a simpler error bound on the right hand side and actually shows closeness to the true solution operator $\sola_t$ as opposed to a the more complicated operator appearing in Theorem~\ref{thm:limit-is-close-pointwise-ell-dependent}. However, in order to get this simpler statement, we need to first sacrifice some optimality in time and consider $L^2$ initial data instead of the more general $TV$ initial data; however this will make no difference for proving the anomalous regularization.

\subsubsection{A sketch of the argument}

The main difficulty is in proving Theorem~\ref{thm:limit-is-close-pointwise-ell-dependent}. Corollary~\ref{cor:is-close-to-ideal-pointwise} essentially follows from Theorem~\ref{thm:limit-is-close-pointwise-ell-dependent} together with Lemma~\ref{lem:soln-operator-doesnt-change-averages} which gives that $\sola_t e^{\kappa \sigma_\ell \Delta}\sol^{V,\kappa,\T^2}_{0,s^{\ell+1}_{\ell+1}} \approx \sola_t$ for $L^2$ data. We thus only discuss the proof of Theorem~\ref{thm:limit-is-close-pointwise-ell-dependent}.

In the proof of Theorem~\ref{thm:limit-is-close-pointwise-ell-dependent}, we split $[0,1]$ into the following time intervals:
\[[0,s^{\ell+1}_{\ell+1}],\; [s^{\ell+1}_{\ell+1}, s^\ell_\ell],\; [s^{i+1}_{i+1}, s^i_{i+1}],\;\text{and}\; [s^i_{i+1},s^i_i],\]
where $0 \leq i < \ell$. We will only focus on understanding the endpoints of the intervals as the intermediate times follow from some additional bookkeeping. On the first time interval, $[0,s^{\ell+1}_{\ell+1}],$ the diffusion is in some sense dominant, and we have no means of controlling the solution uniformly in diffusivity. It is for this reason we ``give up'', and why the solution operator $\sol^{V,\kappa,\T^2}_{0,s^{\ell+1}_{\ell+1}}$ shows up in Theorem~\ref{thm:limit-is-close-pointwise-ell-dependent}. 

For the next step, a key ingredient of the argument is provided by Proposition~\ref{prop:rescaled-two-cell-dissipators}, which uses that on the time interval $[s^i_{j+1}, s^i_j]$, $V$ is just space and time rescaled copies of the flow $v$. Thus by rescaling Theorem~\ref{thm:two-cell-dissipator}, Proposition~\ref{prop:rescaled-two-cell-dissipators} gives (essentially) that
\[\sol^{V,\kappa,\T^2}_{s^i_{j+1},s^i_j} \Pi_{j+1} \approx \Pi_j.\]
That is---up to a small error---running $\sol^{V,\kappa,\T^2}_{s^i_{j+1},s^i_j}$ on data that is piecewise constant on scale $2^{-(j+1)/2}$ averages that data up to scale $2^{-j/2}.$ We note however that the error in this estimate degenerates as $j \to \infty$.

It is on the time interval $[s^{\ell+1}_{\ell+1}, s^\ell_\ell]$ that we first use the structure of $V$ in a nontrivial way. We write the solution operator in the following way, using that $V = 0$ on $[s^{\ell+1}_{\ell+1}, s^\ell_\infty],$
\[\sol_{0,s^\ell_\ell}^{V,\kappa,\T^2} = \sol^{V,\kappa,\T^2}_{s^\ell_{\ell+1}, s^\ell_\ell}\cdots\sol^{V,\kappa,\T^2}_{s^\ell_n, s^\ell_{n-1}}\sol^{V,\kappa,\T^2}_{s^\ell_\infty, s^\ell_n} e^{\kappa \sigma_\ell \Delta}\sol^{V,\kappa,\T^2}_{0,s^{\ell+1}_{\ell+1}}.\]
We want to use that $\sol^{V,\kappa,\T^2}_{s^i_{j+1},s^i_j} \Pi_{j+1} \approx \Pi_j$ to iteratively replace that terms $\sol^{V,\kappa,\T^2}_{s^\ell_{\ell+1}, s^\ell_\ell}\cdots\sol^{V,\kappa,\T^2}_{s^\ell_n, s^\ell_{n-1}}$ with $\Pi_\ell$, however to do this, we need $\sol^{V,\kappa,\T^2}_{s^\ell_\infty, s^\ell_n} e^{\kappa \sigma_\ell \Delta} \approx \Pi_n$. We get that this is true---provided $n$ is sufficiently large depending on $\ell$---by combining three estimates. The first is the smoothing due to the heat kernel $\|e^{s\Delta}\|_{L^1 \to W^{1,1}} \leq C s^{-1/2}$; the second given by Lemma~\ref{lem:drift-diffusion-small} which gives that a drift-diffusion equation with small diffusion and drift small in $L^1_t L^\infty_x$ essentially leaves solutions invariant in $L^1$ if the initial data has $W^{1,1}$ regularity; and the third is given by~\eqref{eq:pi_projecting-bound} in Lemma~\ref{lem:pi-bounds} which gives that $\Pi_n$ essentially leaves $W^{1,1}$ functions invariant in $L^1$, provided $n$ is sufficiently large. Combining the first two bounds gives that $\sol^{V,\kappa,\T^2}_{s^\ell_\infty, s^\ell_n} e^{\kappa \sigma_\ell \Delta} \approx e^{\kappa \sigma_\ell \Delta}$ and then the first and third bound give that $e^{\kappa \sigma_\ell \Delta} \approx \Pi_n e^{\kappa \sigma_\ell \Delta}.$ Summarizing the above discussion, we have for $n$ sufficiently large depending on $\ell$
\[\sol_{0,s^\ell_\ell}^{V,\kappa,\T^2} \approx \sol^{V,\kappa,\T^2}_{s^\ell_{\ell+1}, s^\ell_\ell}\cdots\sol^{V,\kappa,\T^2}_{s^\ell_n, s^\ell_{n-1}}\Pi_n e^{\kappa \sigma_\ell \Delta}\sol^{V,\kappa,\T^2}_{0,s^{\ell+1}_{\ell+1}},\]
where we are implicitly using that the solution operator is an $L^1$ contraction, so that if $\theta \approx \phi$, then $\sol^{V,\kappa,\T^2}_{s,t} \theta \approx \sol^{V,\kappa,\T^2}_{s,t}\varphi$. Now we can iteratively use that $\sol^{V,\kappa,\T^2}_{s^i_{j+1},s^i_j} \Pi_{j+1} \approx \Pi_j$ with the above to get that for $n$ sufficiently small, 
\[\sol_{0,s^\ell_\ell}^{V,\kappa,\T^2} \approx \Pi_\ell e^{\kappa \sigma_\ell \Delta}\sol^{V,\kappa,\T^2}_{0,s^{\ell+1}_{\ell+1}}= \sola_{s^\ell_\ell}  e^{\kappa \sigma_\ell \Delta}\sol^{V,\kappa,\T^2}_{0,s^{\ell+1}_{\ell+1}}.\]
This then deals with the time interval $[s^{\ell+1}_{\ell+1}, s^\ell_\ell]$, with the precise constants appearing from optimizing the choice of $n$ in the above sketch. This whole argument is recorded in Proposition~\ref{prop:first-time-close-to-proj}. 

We now consider the intervals $[s^{i+1}_{i+1},s^i_{i+1}]$ and $[s^i_{i+1},s^i_i]$. Inductively we can suppose that for some $i<\ell$ we have controlled up to time $s^{i+1}_{i+1}$ and so we have that
\[\sol_{0,s^{i+1}_{i+1}}^{V,\kappa,\T^2} \approx \Pi_{i+1} e^{\kappa \sigma_\ell \Delta}\sol^{V,\kappa,\T^2}_{0,s^{\ell+1}_{\ell+1}}= \sola_{s^{i+1}_{i+1}}  e^{\kappa \sigma_\ell \Delta}\sol^{V,\kappa,\T^2}_{0,s^{\ell+1}_{\ell+1}}.\]
We want to show that
\[\sol_{0,s^{i+1}_{i}}^{V,\kappa,\T^2}  \approx \Pi_{i+1} e^{\kappa \sigma_\ell \Delta}\sol^{V,\kappa,\T^2}_{0,s^{\ell+1}_{\ell+1}}= \sola_{s^{i}_{i+1}}  e^{\kappa \sigma_\ell \Delta}\sol^{V,\kappa,\T^2}_{0,s^{\ell+1}_{\ell+1}}\] \text{and} \[ \sol_{0,s^{i}_{i}}^{V,\kappa,\T^2}  \approx \Pi_{i} e^{\kappa \sigma_\ell \Delta}\sol^{V,\kappa,\T^2}_{0,s^{\ell+1}_{\ell+1}}= \sola_{s^{i}_{i}}  e^{\kappa \sigma_\ell \Delta}\sol^{V,\kappa,\T^2}_{0,s^{\ell+1}_{\ell+1}}.\]
We note that the second follows from the first, as---assuming the first---we have that 
\[\sol_{0,s^{i}_{i}}^{V,\kappa,\T^2} = \sol_{s^i_{i+1},s^{i}_{i}}^{V,\kappa,\T^2} \sol_{0,s^{i}_{i+1}}^{V,\kappa,\T^2} \approx \sol_{s^i_{i+1},s^{i}_{i}}^{V,\kappa,\T^2}\Pi_{i+1} e^{\kappa \sigma_\ell \Delta}\sol^{V,\kappa,\T^2}_{0,s^{\ell+1}_{\ell+1}} \approx \Pi_{i} e^{\kappa \sigma_\ell \Delta}\sol^{V,\kappa,\T^2}_{0,s^{\ell+1}_{\ell+1}},\]
where we use that by Proposition~\ref{prop:rescaled-two-cell-dissipators}, $\sol^{V,\kappa,\T^2}_{s^i_{j+1},s^i_j} \Pi_{j+1} \approx \Pi_j$.

Thus it only remains to see that $\sol_{0,s^{i+1}_{i}}^{V,\kappa,\T^2}  \approx \Pi_{i+1} e^{\kappa \sigma_\ell \Delta}\sol^{V,\kappa,\T^2}_{0,s^{\ell+1}_{\ell+1}}$. By our inductive hypothesis, we have that
\[\sol_{0,s^{i+1}_{i}}^{V,\kappa,\T^2}  = \sol_{s_{i+1}^{i+1}, s^i_{i+1}}^{V,\kappa,\T^2}\sol_{0,s^{i+1}_{i+1}}^{V,\kappa,\T^2} \approx \sol_{s_{i+1}^{i+1}, s^i_{i+1}}^{V,\kappa,\T^2} \Pi_{i+1} e^{\kappa \sigma_\ell \Delta}\sol^{V,\kappa,\T^2}_{0,s^{\ell+1}_{\ell+1}}.\]
Thus to conclude our sketch, it suffices to see that 
\[ \sol_{s_{i+1}^{i+1}, s^i_{i+1}}^{V,\kappa,\T^2} \Pi_{i+1}  \approx \Pi_{i+1}.\]
This fact is recorded in Proposition~\ref{prop:no-change-by-partial-total-dissipator}. The key observation for the proof of this fact is that on the time interval $[s^{i+1}_{i+1}, s^i_{i+1}]$, the velocity field $V$ is tangent to the interfaces across which the projection $\Pi_{i+1}$ varies. As such, the velocity field---at least formally when the diffusivity is $0$---does not alter the solution at all once it is piecewise constant on scale $i+1$. That is, formally, we have that
\[\Pi_{i+1} \sol^{V,0,\T^2}_{s^{i+1}_{i+1}, t} = \sol^{V,0,\T^2}_{s^{i+1}_{i+1}, t}\Pi_{i+1}  = \Pi_{i+1}.\]
This is no longer exactly true in the positive diffusivity case that we are interested in. However, carefully estimating the error due to the diffusivity, using as the central input Lemma~\ref{lem:constant_error}, we get that the estimate is at least approximately true: $\sol_{s_{i+1}^{i+1}, s^i_{i+1}}^{V,\kappa,\T^2} \Pi_{i+1}  \approx \Pi_{i+1}$ as desired.

\section{Convergence to the limiting solution operator and anomalous total dissipation}

\label{s:convergence-to-ideal}

This section primarily comprises a variety of technical estimates which form the key steps of the proof of Theorem~\ref{thm:limit-is-close-pointwise-ell-dependent}. In particular, Proposition~\ref{prop:rescaled-two-cell-dissipators}, which gives that $\sol_{s^i_{j+1},s^i_j}^{V,\kappa,\T^2} \Pi_{j+1} \approx \Pi_j$, Proposition~\ref{prop:first-time-close-to-proj}, which gives that $\sol^{V,\kappa,\T^2}_{s^{i+1}_{i+1},s^i_i} \approx \Pi_i e^{\kappa \sigma_i \Delta},$ and Proposition~\ref{prop:no-change-by-partial-total-dissipator}, which gives that $\sol^{V,\kappa,\T^2}_{s^{i+1}_{i+1},t} \Pi_{i+1} \approx \Pi_{i+1}$ for $t \in [s^{i+1}_{i+1}, s^i_{i+1}]$. Following the proof of theses propositions---and the preceding lemmas necessary for their proof---we conclude the section with the proofs of Theorem~\ref{thm:limit-is-close-pointwise-ell-dependent} and Corollary~\ref{cor:is-close-to-ideal-pointwise}.

We start by recording some bounds on the projection operators $\Pi_j$, defined in Definition~\ref{def:Pi}. The elementary proof is given in Appendix~\ref{appendix:lemmas}.

\begin{lemma}\label{lem:pi-bounds}
    There exists $C>0$ such that for all $j \in \N$,
    \begin{align}
    \label{eq:pi-l1-l1-bound}
    \|\Pi_j\|_{L^1(\T^2) \to L^1(\T^2)} & = 1,\\
    \label{eq:pi-l1-linfty-bound}
    \|\Pi_j\|_{L^1(\T^2) \to L^\infty(\T^2)}& \leq C2^j, \\
    \label{eq:pi-l1-bv-bound}
    \|\Pi_j\|_{L^1(\T^2) \to BV(\T^2)} & \leq C 2^{j/2},\\
    \label{eq:pi_projecting-bound}
    \|1-\Pi_j\|_{W^{1,1}(\T^2)\to L^1(\T^2)} &\leq  C2^{-j/2}.
    \end{align}
\end{lemma}

We next note a bound of drift-diffusion equations. We will need this bound in order to control contributions due to boundary data when we split our domain into many small grid cells $A_j+x_j$ and solve the drift-diffusion equation on each cell separately. To better understand the utility of this lemma, inspect its usage in the proofs of Proposition~\ref{prop:rescaled-two-cell-dissipators} and Lemma~\ref{lem:Pi_k_preservation}. The proof is deferred to Appendix~\ref{appendix:lemmas}.

\begin{lemma}
\label{lem:box-bounds-sharp}
    There exists $C>0$ so that for any divergence-free $u \in L^1([0,1],L^\infty(\T^2))$, initial data $\theta \in L^1(\T^2),$ $\kappa\in(0,1)$, $N>0$, and $j \in \N$ such that
    \[\kappa^{1/2} \leq 2^{-j/2} \quad \text{and} \quad \|u\|_{L^1([0,1],L^\infty(\T^2))} \leq N 2^{-j/2},\]
    then
    \[\sum_{x_j \in \Lambda_j} 2^{-j}\sup_{t \in [0,1] } \sup_{y \in A_j+x_j} \big|\sol_{0,t}^{u,\kappa,\T^2} \Pi_j \theta(y)\big| \leq C (N+1)^2 \|\Pi_j \theta\|_{L^1(\T^2)}.\]
\end{lemma}

The following proposition is easiest to interpret when $t = s^i_j$, in which case it gives that
\[\sol_{s^i_{j+1}, s^i_j}^{V,\kappa,\T^2}\Pi_{j+1} \approx \Pi_j,\]
up to a precisely controlled error. The proposition also gives precise bounds for intermediate times $t \in [s^i_{j+1}, s^i_j].$ For $t \in [s^i_{j+1}, s^i_{j+1} + \sigma_j/2)$, we are showing the dynamics of the drift-diffusion equation are close to the dynamics of the pure transport equation, with zero diffusion. This estimate necessarily diverges at the singular time $t = s^i_{j+1} + \sigma_j/2$, as at this time the transport dynamics live on infinitesimally small scales. On the time interval $[s^i_{j+1}+\sigma_j/2, s^i_j]$, the estimate gives that the drift-diffusion dynamics are close to the projection operator. This estimate is however also divergent at $t =s^i_{j+1} + \sigma_j/2$, essentially as we need to wait some time for the smoothing of the heat kernel to take effect, which is the origin of the $t^{-1/2}$ scaling of the divergence. This estimate follows essentially from carefully rescaling Theorem~\ref{thm:two-cell-dissipator}.

\begin{proposition}
    \label{prop:rescaled-two-cell-dissipators}
    There exists $C(\alpha)>0$ such that for all $i \in \N, j \geq i$, $\kappa \in (0,1)$, and $t \in [s^i_{j+1}, s^i_j]$,
    \[\big\|(\sol_{s^i_{j+1}, t}^{V,\kappa,\T^2} - \sol_{s^i_{j+1},t}^{V,0,\T^2})\Pi_{j+1}\big\|_{L^1(\T^2) \to L^1(\T^2)} \leq   C(2^{j(1+\alpha)/2}\kappa)^{(1-\alpha)/12} \big|\sigma_j^{-1}(t-s^i_{j+1} -\sigma_j/2)\big|^{-1/2},\]
    where we take the convention that for $t \in [s^i_{j+1}, s^i_j]$,
    \[\sol_{s^i_{j+1},t}^{V,0,\T^2} = \begin{cases}
        \sol_{s^i_{j+1},t}^{V,0,\T^2}& t \in [s^i_{j+1},s^i_{j+1} + \sigma_j/2), \\
        \Pi_j & t \in [s^i_{j+1} + \sigma_j/2, s^i_j].
    \end{cases}\]
\end{proposition}

\begin{proof}
    Let $f$ be an arbitrary element of $C^\infty(\T^2)$. Then we can find $(a^{x_j}_0, a^{x_j}_1)_{x_j \in \Lambda_j}$ such that if we define $\theta_{a^{x_j}_0, a^{x_j}_1} : B \to \R$ as in~\eqref{eq:two-cell-data} and take $\theta_{a^{x_j}_0, a^{x_j}_1}(y) =0$ for $y \not \in B$, then we have that
     \[\Pi_{j+1} f(y) = \sum_{x_j \in \Lambda_j}\theta_{a^{x_j}_0, a^{x_j}_1}\big(\mathcal{R}^{-j}(y-x_j)\big).\]
      Then, using the definition of $V$,
    \[\sol^{V,\kappa,\T^2}_{s^i_{j+1}, t} \Pi_{j+1} f(y) = \sum_{x_j \in \Lambda_j} \Big(\sol^{v, \sigma_j 2^j\kappa,g^{x_j}}_{0, \sigma_j^{-1}(t- s^i_{j+1})} \theta_{a_0^{x_j},a^{x_j}_1}\Big)\big(\mathcal{R}^{-j}(y-x_j)\big),\]
    for the boundary data $g^{x_j} : [0,1] \times \partial B \to \R$ given by
    \[g^{x_j}(t,y) = \sol^{V,\kappa,\T^2}_{s^i_{j+1}, s^i_{j+1} + t \sigma_j} \Pi_{j+1} f\big(\mathcal{R}^{-j} (y-x_j)\big).\]
     Thus, using Theorem~\ref{thm:two-cell-dissipator}, we have that
    \begin{align}\notag
        &\Big\|\Big(\sol_{s^i_{j+1}, t}^{V,\kappa,\T^2} - \sol_{s^i_{j+1},t}^{V,0,\T^2}\Big)\Pi_{j+1}f\Big\|_{L^1}\leq C\sum_{x_j \in \Lambda_j} 2^{-j}\Big\|\Big(\sol^{v, \sigma_j 2^j\kappa,g^{x_j}}_{0,\sigma_j^{-1}(t- s^i_{j+1})}  - \sol^{v, 0,g^{x_j}}_{0,\sigma_j^{-1}(t- s^i_{j+1})}\Big)\theta_{a_0^{x_j},a_1^{x_j}}\Big\|_{L^1}
        \\&\quad\leq C \sum_{x_j \in \Lambda_j}2^{-j} \big(|a^{x_j}_1| + |a^{x_j}_0| + \|g^{x_j}\|_{L^\infty([0,1] \times \partial B)}\big) (\sigma_j 2^j\kappa)^{(1-\alpha)/12} \big|\sigma_j^{-1}(t-s^i_{j+1} -\sigma_j/2)\big|^{-1/2}.
        % \notag\\&\qquad\leq  C\Big( \|\Pi_{j+1} f\|_{L^1} + \sum_{x_j \in \Lambda_j}2^{-j} \|g^{x_j}\|_{L^\infty([0,1] \times \partial B)}\Big) (2^{j(1+\alpha)/2}\kappa)^{(1-\alpha)/12} \big|\sigma_j^{-1}(t-s^i_{j+1} -\sigma_j/2)\big|^{-1/2} .
        \label{eq:two-cell-main-align}
    \end{align}
    Then
    \begin{equation}
        \label{eq:two-cell-extra-term} \sum_{x_j \in \Lambda_j}2^{-j} \big(|a^{x_j}_1| + |a^{x_j}_0| + \|g^{x_j}\|_{L^\infty([0,1] \times \partial B)}\big) \leq  C \|\Pi_{j+1} f\|_{L^1} + \sum_{x_j \in \Lambda_j}2^{-j}  \|g^{x_j}\|_{L^\infty([0,1] \times \partial B)}.
    \end{equation}
    We note that
    \[\|g^{x_j}\|_{L^\infty([0,1] \times \partial B)} \leq \sup_{t \in [0,1]}\sup_{y \in A_j + x_j}\Big| \sol_{s^i_{j+1}, s^i_{j+1} + \sigma_j t}^{V,\kappa,\T^2} \Pi_{j+1} f(y) \Big|=  \sup_{t \in [0,1]}\sup_{y \in A_j + x_j} \Big|\sol_{0,t}^{\tilde V, \sigma_j \kappa,\T^2} \Pi_{j+1} f(y)\Big|,\]
    where 
    \[\tilde V(t,x) = V(\sigma_j t + s^i_{j+1}, x),\]
    so that $\int_0^1 \|\tilde V(t,\cdot)\|_{L^\infty_x}\,dt = \int_{s^i_{j+1}}^{s^i_j} \|V(t,\cdot)\|_{L^\infty_x}\,dt\leq C2^{-j/2}$. We also have that $(\sigma_j \kappa)^{1/2} \leq 2^{-j/2}$ as otherwise $\sigma_j 2^j \kappa \geq 1$ and the estimate follows trivially. Thus we can use  Lemma~\ref{lem:box-bounds-sharp} to give that
    \begin{equation}
    \label{eq:boundary-terms-2cell}
    \sum_{x_j \in \Lambda_j} 2^{-j} \|g^{x_j}\|_{L^\infty([0,1] \times \partial B)} \leq \|\Pi_j f\|_{L^1} \leq \|f\|_{L^1}.
    \end{equation}
    Combining~\eqref{eq:two-cell-main-align},~\eqref{eq:two-cell-extra-term}, and~\eqref{eq:boundary-terms-2cell} and bounding $\|\Pi_{j+1} f\|_{L^1} \leq \|f\|_{L^1}$, we conclude.
\end{proof}

The following lemma, whose proof is deferred to Appendix~\ref{appendix:lemmas}, gives that the advection-diffusion equation essentially leaves $W^{1,1}$ data invariant in $L^1$ provided that the diffusion is sufficiently weak and the drift is sufficiently small. It is used in the proof of Lemma~\ref{lem:BV-partial-total-dissipator} below.

\begin{lemma}
\label{lem:drift-diffusion-small}
    Let $u \in L^1([0,1], L^\infty(\T^2))$ with $\nabla \cdot u=0.$ Then 
    \[\big\|\sol^{u,\kappa,\T^2}_{0,1} - 1\big\|_{W^{1,1}(\T^2) \to L^1(\T^2)} \leq \|u\|_{L^1([0,1],L^\infty(\T^2))} + \sqrt{\pi \kappa}.\]
\end{lemma}

The following lemma essentially combines Proposition~\ref{prop:rescaled-two-cell-dissipators} and Lemma~\ref{lem:drift-diffusion-small} to show that 
\[\sol^{V,\kappa,\T^2}_{s^i_\infty, s^i_j} \approx \Pi_j,\]
provided the initial data is in $W^{1,1}$. The argument follows by splitting $[s^i_\infty, s^i_j]$ into $[s^i_\infty, s^i_n) \cup [s^i_n, s^i_j]$ for $n \geq j \geq i$ and using Lemma~\ref{lem:drift-diffusion-small} on the former interval and iteratively using Proposition~\ref{prop:rescaled-two-cell-dissipators} on the latter. The estimate below is the key ingredient to Proposition~\ref{prop:first-time-close-to-proj}. 

\begin{lemma}
\label{lem:BV-partial-total-dissipator}
    There exists $C(\alpha)>0$ such that for all $\theta \in W^{1,1}(\T^2),$ $\kappa\in(0,1)$, $i \in \N$ and $n \geq j \geq i$, we have the bound
    \[\Big\|\Big(\sol_{s^i_\infty, s^i_j}^{V,\kappa,\T^2} - \Pi_j\Big)\theta\Big\|_{L^1(\T^2)} \leq C\big(2^{-n/2} + (2^{-n(1-\alpha)/2} \kappa)^{1/2}\big)\|\theta\|_{W^{1,1}(\T^2)} + C (2^{n(1+\alpha)/2} \kappa)^{(1-\alpha)/12} \|\theta\|_{L^1(\T^2)}.\]
\end{lemma}
\begin{proof}
    Let $n \geq j$ and write 
    \[\sol_{s^i_\infty, s^i_j}^{V,\kappa,\T^2} - \Pi_j = \sol_{s^i_n, s^i_j}^{V,\kappa,\T^2} \sol_{s^i_\infty,s^i_n}^{V,\kappa,\T^2} - \Pi_j = \sol_{s^i_n, s^i_j}^{V,\kappa,\T^2} \Big(\sol_{s^i_\infty,s^i_n}^{V,\kappa,\T^2}-1\Big) + \sol_{s^i_n, s^i_j}^{V,\kappa,\T^2} (1- \Pi_n) +  \sol_{s^i_n, s^i_j}^{V,\kappa,\T^2} \Pi_n  - \Pi_j.\]
    Further we decompose
    \[\sol_{s^i_n, s^i_j}^{V,\kappa,\T^2} \Pi_n  - \Pi_j
= \sum_{k=j}^{n-1} \sol^{V,\kappa,\T^2}_{s^i_k, s^i_j} \Big(\sol^{V,\kappa,\T^2}_{s^i_{k+1}, s^i_k} \Pi_{k+1} - \Pi_k\Big). 
    \]
    Thus in total, we get
    \begin{align*}
    \Big\|\Big(\sol_{s^i_\infty, s^i_j}^{V,\kappa,\T^2} - \Pi_j\Big)\theta\Big\|_{L^1} &\leq \Big\|\sol_{s^i_\infty,s^i_n}^{V,\kappa,\T^2}-1\Big\|_{W^{1,1} \to L^1} \|\theta\|_{W^{1,1}} + \|1- \Pi_n\|_{W^{1,1} \to L^1} \|\theta\|_{W^{1,1}} 
    \\&\qquad+ \sum_{k=j}^{n-1} \Big\|\sol^{V,\kappa,\T^2}_{s^i_{k+1}, s^i_k} \Pi_{k+1} - \Pi_k\Big\|_{L^1 \to L^1} \|\theta\|_{L^1}
    \\&\leq C\big(2^{-n/2} + (\sigma_n \kappa)^{1/2} + 2^{-n/2}\big)\|\theta\|_{W^{1,1}} + C \sum_{k=j}^{n-1} (\sigma_k 2^k \kappa)^{(1-\alpha)/12} \|\theta\|_{L^1}
    \\&\leq  C\big(2^{-n/2} + (2^{-n(1-\alpha)/2} \kappa)^{1/2}\big)\|\theta\|_{W^{1,1}} + C (2^{n(1+\alpha)/2} \kappa)^{(1-\alpha)/12} \|\theta\|_{L^1}
    \end{align*}
    where we use Lemma~\ref{lem:drift-diffusion-small},~\eqref{eq:pi_projecting-bound} in Lemma~\ref{lem:pi-bounds}, and Proposition~\ref{prop:rescaled-two-cell-dissipators}.
\end{proof}

The following proposition combines the smoothing of the heat kernel together with an optimization of Lemma~\ref{lem:BV-partial-total-dissipator} to show that
\[\sol^{V,\kappa,\T^2}_{s^{i+1}_{i+1}, s^i_i} \approx \Pi_i e^{\kappa \sigma_i \Delta}.\]
This estimate is used essentially to prove the ``base case'' of Theorem~\ref{thm:limit-is-close-pointwise-ell-dependent}, getting control of the solution operator $\sol^{V,\kappa,\T^2}_{0,t}$ at the time $t = s^\ell_\ell.$ 

\begin{proposition}
\label{prop:first-time-close-to-proj}
    There exists $C(\alpha)>0$ such that for all $\kappa\in(0,1)$ and $i \in \N,$
    we have the bound
    \[\Big\|\sol^{V,\kappa,\T^2}_{s^{i+1}_{i+1}, s^i_i} - \Pi_i e^{\kappa \sigma_i\Delta}\Big\|_{L^1(\T^2) \to L^1(\T^2)} \leq   C \big( 2^{i(1+\alpha)/2} \kappa\big)^{(1-\alpha)^2/8}.\]
\end{proposition}
\begin{proof}
    We can suppose without loss of generality that $2^{-i/2} \geq \kappa^{\frac{1}{1+\alpha}}$ as otherwise the bound trivializes (as we can always bound the operator norm by $2$). We then note that 
    \[\sol^{V,\kappa,\T^2}_{s^{i+1}_{i+1}, s^i_i} - \Pi_i e^{\kappa \sigma_i \Delta} = \Big(\sol^{V,\kappa,\T^2}_{s^{i}_{\infty}, s^i_i} - \Pi_i\Big) e^{\kappa \sigma_i\Delta}.\]
    By the usual estimate on the heat kernel and Young's convolution inequality we have that
    \[\|e^{s \Delta}\|_{L^1 \to W^{1,1}} \leq Cs^{-1/2}.\]
    Combining this with Lemma~\ref{lem:BV-partial-total-dissipator}, we have that for any $n \geq i$, 
    \begin{align*}
    \Big\|\sol^{V,\kappa,\T^2}_{s^{i+1}_{i+1}, s^i_i} - \Pi_i e^{\kappa \sigma_i \Delta}\Big\|_{L^1 \to L^1} &\leq  C\big(2^{-n/2} + 2^{-n/2}( 2^{n(1+\alpha)/2} \kappa)^{1/2}\big) (\sigma_i \kappa)^{-1/2} + C (2^{n(1+\alpha)/2} \kappa)^{(1-\alpha)/12}
    \\&\leq  C2^{-n/2} 2^{i(1-\alpha)/4} \kappa^{-1/2} + C (2^{n(1+\alpha)/2} \kappa)^{(1-\alpha)/12},
    \end{align*}
    where we suppose without loss of generality that $2^{n(1+\alpha)/2}\kappa \leq 1$, as otherwise the bound becomes trivial.
    We then set
    \[n:= \big \lceil \tfrac{2}{3+\alpha} \log_{\sqrt{2}}(\kappa^{-1}) + \tfrac{1-\alpha}{3+\alpha} i\big \rceil,\]
    for which we note that $2^{-i/2} \geq \kappa^{\frac{1}{1+\alpha}}$ implies $n \geq i$ as required. For this $n$, we compute
    \begin{align*}
    \Big\|\sol^{V,\kappa,\T^2}_{s^{i+1}_{i+1}, s^i_i} - \Pi_i e^{\kappa \sigma_i\Delta}\Big\|_{L^1 \to L^1} &\leq  C \kappa^{\frac{1-\alpha}{2(3+\alpha)}} 2^{\frac{(1 -\alpha)(1+\alpha)}{4(3+\alpha)}i} + C \big( 2^{\frac{(1-\alpha)(1+\alpha)}{4(3+\alpha)}i} \kappa^{\frac{1-\alpha}{2(3+\alpha)}}\big)^{(1-\alpha)/6}
    \\& \leq  C \big( 2^{i(1+\alpha)/2} \kappa\big)^{(1-\alpha)^2/8},
    \end{align*}
    as desired.
\end{proof}

The following lemma, whose proof is deferred to Appendix~\ref{appendix:lemmas}, shows that if the drift $u$ is tangent to the boundary of the domain $\partial B$, then the drift-diffusion equation essentially leaves constants invariant---independent of boundary data---provided the diffusion is sufficiently small and the Lipschitz norm of the velocity field is not too large. This lemma is the key ingredient to Lemma~\ref{lem:Pi_k_preservation}.

\begin{lemma}\label{lem:constant_error}
For any divergence-free vector field $u \in L^\infty([0,1],W^{1,\infty}(B))$ tangent to $\partial B$, for all $\beta<\frac{1}{2}$, there exists a constant $C(\beta)>0$ so that for all $\kappa\in(0,1)$, boundary data $f\in L^\infty([0,1],L^\infty(\partial B))$ and $t\in[0,1]$
\[\big\|\sol^{u,f,\kappa}_{0,t}a-a\big\|_{L^1(B)}\leq C(t\kappa)^\beta(\|f\|_{L^\infty([0,1]\times\partial B)}+|a|) \big((t\|\nabla u\|_{L^\infty([0,1],L^\infty(B))})^{1-\beta}+1\big).\]
\end{lemma}

The following lemma shows that for $t \in [s^i_{j+1}, s^i_j]$ 
\[\sol^{V, \kappa,\T^2}_{s^i_{j+1},t} \Pi_k \approx \Pi_k,\]
that is---on this time interval---functions that are piecewise constant down to length-scale $2^{-k/2}$ are essentially invariant under $\sol^{V, \kappa,\T^2}_{s^i_{j+1},t}$. The proof follows from a careful iteration of Lemma~\ref{lem:constant_error}, breaking up the time interval $[s^i_{j+1}, s^i_j]$ into many smaller time intervals. This is necessary as $\|\nabla u(t,\cdot)\|_{L^\infty_x}$ blows up at time $t = s^i_{j+1} + \sigma_{j+1}/2.$ This lemma provides the key estimate for Proposition~\ref{prop:no-change-by-partial-total-dissipator}.

We will use the following notation throughout the remainder of the paper.

\begin{definition}
    \label{def:average-defn} For a function $f: B \to \R$ and a positive measure subset $A \subseteq B$, we denote the average of $f$ over $A$ by $(f)_A$,
    \[(f)_A := \frac{1}{|A|} \int_A f(x)\,dx.\]
\end{definition}

\begin{lemma}\label{lem:Pi_k_preservation} There exists $C(\alpha)>0$ so that for all $i\in\N$, $j\geq k\geq i$, $\kappa\in(0,1)$, and $t\in[s^i_{j+1},s^i_{j}]$,
\[\Big\|\sol_{s^i_{j+1},t}^{V,\kappa,\T^2}\Pi_k-\Pi_k\Big\|_{L^1(\T^2)\rightarrow L^1(\T^2)}\leq C(2^{k-(1-\alpha)j/2}\kappa)^{1/3}\]
\end{lemma}

\begin{proof}
As the bound holds trivially otherwise, without loss of generality we assume that $\sigma_j2^k\kappa\leq 1$.

For an arbitrary function $\theta$ and $x_k\in\Lambda_k$, let $a^{x_k}=(\theta)_{A_k+x_k}$. Fixing $x_k$, then for all $x\in A_k+x_k$
\[\sol_{s^i_{j+1},t}^{V,\kappa,\T^2}\Pi_k\theta(x)=\sol_{0,\sigma_j^{-1}(t-s^i_{j+1})}^{U,\sigma_j2^k\kappa,g^{x_k}}a^{x_k}(\mathcal{R}^{-k}(x-x_k))\]
for boundary data $g^{x_k}:[0,T]\times \partial B\rightarrow\R$ given by
\[g^{x_k}(t,y)=\sol_{s^i_{j+1}, \sigma_jt+s^i_{j+1}}^{V,\kappa,\T^2} a^{x_k}(\mathcal{R}^{-k}(y-x_k)),\]
and vector field $U:[0,1]\times B\rightarrow\R^2$ given by
\[U(t,x)=\sigma_j \mathcal{R}^{-k}V(\sigma_jt+s^i_{j+1},\mathcal{R}^{k}x+x_k))=R^{j-k}v(t,\mathcal{R}^{-(j-k)}x).\]
where we've used the definition of $V$.

Using that $\sol^{V,\kappa,\T^2}_{s,t}$ is an $L^1$ contraction we thus find that
\[\Big\|\sol_{s^i_{j+1},t}^{V,\kappa,\T^2}\Pi_k\theta-\Pi_k\theta\Big\|_{L^1}\leq \sum_{x_k\in\Lambda_k} 2^{-k}\Big\|\sol_{0,\sigma_j^{-1}(t-s^i_{j+1})}^{U,\sigma_j2^k\kappa,g^{x_k}}a^{x_k}-a^{x_k}\Big\|_{L^1}.\]
We now claim that for all $t\in[s^i_j,s^i_{j+1}]$,
\begin{equation}\label{eq:mass_escape_claim}
  \Big\|\sol_{0,\sigma_j^{-1}(t-s^i_{j+1})}^{U,\sigma_j2^k\kappa,g^{x_k}}a^{x_k}-a^{x_k}\Big\|_{L^1}\leq C(\sigma_j 2^{k}\kappa)^{1/3}(\|g^{x_k}\|_{L^\infty([0,1]\times\partial b)}+|a^{x_k}|).  
\end{equation}
Given this we can conclude the lemma as in the proof of Proposition~\ref{prop:rescaled-two-cell-dissipators}. Combining the above two displays, it then holds that
\[\Big\|\sol_{s^i_{j+1},t}^{V,\kappa,\T^2}\Pi_k\theta-\Pi_k\theta\Big\|_{L^1}\leq C(2^k\sigma_j\kappa)^{1/3}\sum_{x_k\in\Lambda_k} 2^{-k}(\|g^{x_j}\|_{L^\infty([0,1]\times\partial B)}+|a^{x_k}|).\]
Thus it suffices to prove that
\[\sum_{x_k\in\Lambda_k} 2^{-k}(\|g^{x_j}\|_{L^\infty([0,1]\times\partial B)}+|a^{x_k}|)\leq C\|\theta\|_{L^1}.\]
It trivially holds that
\[\sum_{x_k\in\Lambda_k} 2^{-k}|a^{x_k}|=\|\Pi_k\theta\|_{L^1}\leq \|\theta\|_{L^1},\]
while 
\[\|g^{x_k}\|_{L^\infty([0,1] \times \partial B)} \leq \sup_{t \in [0,1]}\sup_{y \in A_k + x_k} \Big|\sol_{s^i_{j+1}, \sigma_jt+s^i_{j+1}}^{V,\kappa,\T^2} a^{x_k}\Big| =  \sup_{t \in [0,1]}\sup_{y \in A_k + x_k} \Big|\sol_{0,t}^{\tilde V, \sigma_j \kappa,\T^2} \Pi_k\theta\Big|,\]
where 
\[\tilde V(t,x) := V(\sigma_j t + s^i_{j+1}, x).\]
Since
\[\int_0^1 \big\|\tilde V(t,\cdot)\big\|_{L^\infty}\,dt = \int_{s^i_{j+1}}^{s^i_j} \|V(t,\cdot)\|_{L^\infty}\,dt\leq C2^{-j/2}\leq C 2^{-k/2},\]
and $(\sigma_j \kappa)^{1/2} \leq 2^{-k/2}$ by assumption, Lemma~\ref{lem:box-bounds-sharp} implies that
\begin{equation}
\sum_{x^k \in \Lambda_k} 2^{-k} \|g^{x_k}\|_{L^\infty([0,1] \times \partial B)} \leq \|\Pi_k \theta\|_{L^1} \leq \|\theta\|_{L^1},
\end{equation}
and we are done.

It thus only remains to verify~\eqref{eq:mass_escape_claim}. We will proceed by proving the bound on three regions of time. When $0\leq\sigma_j^{-1}(t-s^i_{j+1})<\frac{1}{2}$, when $\sigma_j^{-1}(t-s^i_{j+1})=\frac{1}{2},$ and when $\frac{1}{2}<\sigma_j^{-1}(t-s^i_{j+1})\leq 1$.

First, for $0\leq\sigma_j^{-1}(t-s^i_{j+1})<\frac{1}{2}$, let $t_i=\frac{1}{2}-2^{-i-1}$ and let $n$ be such that 
\[t_n\leq\sigma_j^{-1}(t-s^i_{j+1})\leq t_{n+1}.\]
Then we have that
\begin{equation}\label{eq:mass_escape_split}
   \Big\|\sol_{0,\sigma_j^{-1}(t-s^i_{j+1})}^{U,\sigma_j2^k\kappa,g^{x_k}}a^{x_k}-a^{x_k}\Big\|_{L^1}\leq \sum_{i=1}^{n}\Big\|\sol_{t_{i-1},t_i}^{U,\sigma_j2^k\kappa,g^{x_k}}a^{x_k}-a^{x_k}\Big\|_{L^1}+\Big\|\sol_{t_n,\sigma_j^{-1}(t-s^i_{j+1})}^{U,\sigma_j2^k\kappa,g^{x_k}}a^{x_k}-a^{x_k}\Big\|_{L^1}.
\end{equation}
Lemma~\ref{lem:constant_error} with $\beta=\frac{1}{3}$ implies that for all $i$
\begin{align*}\Big\|\sol_{t_{i-1},t_i}^{U,\sigma_j2^k\kappa,g^{x_k}}a^{x_k}-a^{x_k}\Big\|_{L^1}&\leq C(\sigma_j 2^{k}2^{-i-1}\kappa)^{\frac{1}{3}}(\|g^{x_k}\|_{L^\infty([0,1]\times\partial B)}+|a^{x_k}|)
\\&\qquad\qquad\qquad\qquad \times\big((2^{-i-1}\|\nabla U\|_{L^\infty([t_{i-1},t_i],L^\infty(B)})^{\frac{2}{3}}+1\big),
\end{align*}
where we have used that $t_i-t_{i-1}=2^{-i-1}$. The definition of $U$ and Item~\ref{item:Lipschitz_bound} in Theorem~\ref{thm:two-cell-dissipator} imply that
\[\|\nabla U\|_{L^\infty([t_{i-1},t_i],L^\infty(B)}\leq \|\nabla v\|_{L^\infty([t_{i-1},t_i],L^\infty(B)}\leq \frac{M}{1-\alpha}2^{i+1},\]
thus combining the above displays, there exists a constant $C(\alpha)>0$ such that
\[\Big\|\sol_{t_{i-1},t_i}^{U,\sigma_j2^k\kappa,g^{x_k}}a^{x_k}-a^{x_k}\Big\|_{L^1}\leq C(\sigma_j 2^k2^{-i-1}\kappa)^{\frac{1}{3}} (\|g^{x_k}\|_{L^\infty([0,1]\times\partial b)}+|a^{x_k}|).\]
We similarly find that
\[\Big\|\sol_{t_n,\sigma_j^{-1}(t-s^i_{j+1})}^{U,\sigma_j2^k\kappa,g^{x_k}}a^{x_k}-a^{x_k}\Big\|_{L^1}\leq C(\sigma_j 2^k2^{-n-2}\kappa)^{\frac{1}{3}} (\|g^{x_k}\|_{L^\infty([0,1]\times\partial b)}+|a^{x_k}|).\]
Inserting these estimates into~\ref{eq:mass_escape_split}, in total we have that
\begin{align*}
    \Big\|\sol_{0,\sigma_j^{-1}(t-s^i_{j+1})}^{U,\sigma_j2^k\kappa,g^{x_k}}a^{x_k}-a^{x_k}\Big\|_{L^1}&\leq C(\sigma_j 2^k\kappa)^{\frac{1}{3}}(\|g^{x_k}\|_{L^\infty([0,1]\times\partial b)}+|a^{x_k}|)\sum_{i=1}^{\infty}2^{-\frac{2}{3}(i+1)}
    \\&\leq C(\sigma_j 2^k\kappa)^{\frac{1}{3}}(\|g^{x_k}\|_{L^\infty([0,1]\times\partial b)}+|a^{x_k}|).
\end{align*}
We thus find that~\eqref{eq:mass_escape_claim} holds when $\sigma_j^{-1}(t-s^i_{j+1})<\frac{1}{2}$.

When $\sigma_j^{-1}(t-s^i_{j+1})=\frac{1}{2}$, then for any $i\in\N$ we have that
\[\Big\|\sol_{0,\frac{1}{2}}^{U,\sigma_j2^k\kappa,g^{x_k}}a^{x_k}-a^{x_k}\Big\|_{L^1}\leq \Big\|\sol_{0,t_n}^{U,\sigma_j2^k\kappa,g^{x_k}}a^{x_k}-a^{x_k}\Big\|_{L^1}+\Big\|\sol_{t_n,\frac{1}{2}}^{U,\sigma_j2^k\kappa,g^{x_k}}a^{x_k}-a^{x_k}\Big\|_{L^1}.\]
As the first term on the right-hand side is bounded by the right-hand side of~\eqref{eq:mass_escape_claim} uniformly over $i$ and the second term converges to $0$ as $i\to\infty$, the bound also holds for this $t$.

Finally, for $\frac{1}{2}\leq \sigma_j^{-1}(t-s^i_{j+1})\leq \frac{1}{2}$ we use that $U(s,x)=0$ when $s\in[\frac{1}{2},1]$ by the definition of $V$. We thus find that
\[\Big\|\sol_{0,\sigma_j^{-1}(t-s^i_{j+1})}^{U,\sigma_j2^k\kappa,g^{x_k}}a^{x_k}-a^{x_k}\Big\|_{L^1}\leq\Big\|\sol_{0,\frac{1}{2}}^{U,\sigma_j2^k\kappa,g^{x_k}}a^{x_k}-a^{x_k}\Big\|_{L^1}+\Big\|\sol_{\frac{1}{2},\sigma_j^{-1}(t-s^i_{j+1})}^{0,\sigma_j2^k\kappa,g^{x_k}}a^{x_k}-a^{x_k}\Big\|_{L^1}.\]
The first term on the right-hand side is bounded by~\eqref{eq:mass_escape_claim}, while Lemma~\ref{lem:mean-preservation} implies that
\[\Big\|\sol_{\frac{1}{2},\sigma_j^{-1}(t-s^i_{j+1})}^{0,\sigma_j2^k\kappa,g^{x_k}}a^{x_k}-a^{x_k}\Big\|_{L^1}\leq C(\sigma_j 2^k\kappa)^{1/3}(\|g^{x_k}\|_{L^\infty([0,1],\partial B)}+|a^{x_k}|).\]
This concludes the proof of~\eqref{eq:mass_escape_claim} and thus the lemma.
\end{proof}

The following lemma, whose proof is deferred to Appendix~\ref{appendix:lemmas}, is also needed for Proposition~\ref{prop:no-change-by-partial-total-dissipator}.

\begin{lemma}\label{lem:diffusion-BV-L1-bound}
There exists $C>0$ so that for all $t\geq 0$
\[\|e^{t\Delta}-1\|_{BV(\T^2)\rightarrow L^1(\T^2)}\leq Ct^{1/2}.\]
\end{lemma}

The following proposition will be used to control $\sol^{V,\kappa,\T^2}_{0,t}$ in the time intervals $[s^{i+1}_{i+1}, s^i_{i+1}]$ in Theorem~\ref{thm:limit-is-close-pointwise-ell-dependent}. It states that for $t \in [s^{i+1}_{i+1},s^i_{i+1}],$
\[\sol^{V,\kappa,\T^2}_{s^{i+1}_{i+1},t}\Pi_{i+1} \approx \Pi_{i+1}.\]
This is sensible as the velocity field is constructed on $[s^{i+1}_{i+1},s^i_{i+1}]$ to be tangent to the interfaces along which $\Pi_{i+1} \theta$ varies. However, the proof is complicated by the presence of the diffusion and the infinitely many blow up times of $\|\nabla V(t,\cdot)\|_{L^\infty}$. That said, the proof is now fairly straightforward with Lemma~\ref{lem:Pi_k_preservation} and Lemma~\ref{lem:diffusion-BV-L1-bound} in hand.

\begin{proposition}
\label{prop:no-change-by-partial-total-dissipator}
There exists $C(\alpha)>0$ so that for all $i\in\N$, $\kappa\in(0,1)$, and $t\in[s^{i+1}_{i+1},s^i_{i+1}]$,
\[\Big\|\sol^{V,\kappa,\T^2}_{s^{i+1}_{i+1},t}\Pi_{i+1}-\Pi_{i+1}\Big\|_{L^1(\T^2)\rightarrow L^1(\T^2)}\leq C(2^{i(1+\alpha)/2}\kappa)^{1/3}\]
\end{proposition}

\begin{proof}
Without loss of generality we can assume that $\sigma_i2^i\kappa\leq 1$ as otherwise the bound holds trivially.

First we note that for any $r\in[0,\sigma_i]$, Lemma~\ref{lem:diffusion-BV-L1-bound} and~\eqref{eq:pi-l1-bv-bound} imply that
\begin{align}\label{eq:diffusion_error}
\notag\|e^{r\kappa\Delta}\Pi_{i+1}-\Pi_{i+1}\|_{L^1\rightarrow L^1}\leq \|e^{r\kappa \Delta}-1\|_{BV\rightarrow L^1}\|\Pi_{i+1}\|_{L^1\rightarrow BV}&\leq C(r\kappa)^{1/2} 2^{i/2}
\\\notag&\leq C(\sigma_i 2^i\kappa)^{1/2}
\\&\leq C(2^{i(1+\alpha)/2}\kappa)^{1/3}
\end{align}
where we've used the definition of $\sigma_i$. This implies the claimed bound for all $t\in[s^{i+1}_{i+1},s^i_\infty]$.

Next, suppose that $t\in[s^i_{j+1},s^{i}_{j}]$ for $j\geq i+1$. Then, fixing $k>j$, using that $\sol^{V,\kappa,\T^2}_{s,t}$ is an $L^1$ contraction

\begin{align}\label{eq:operator_splitting_no-change}
\notag\Big\|\sol^{V,\kappa,\T^2}_{s^{i+1}_{i+1},t}\Pi_{i+1}-\Pi_{i+1}\Big\|_{L^1\rightarrow L^1}&\leq \Big\|\sol^{V,\kappa,\T^2}_{s^{i+1}_{i+1},s^{i}_\infty}\Pi_{i+1}-\Pi_{i+1}\Big\|_{L^1\rightarrow L^1}+\Big\|\sol^{V,\kappa,\T^2}_{s^i_\infty,s^i_k}\Pi_{i+1}-\Pi_{i+1}\Big\|_{L^1\rightarrow L^1}
\\&\qquad+\sum_{\ell=j+1}^{k-1}\Big\|\sol^{V,\kappa,\T^2}_{s^i_{\ell+1},s^i_{\ell}}\Pi_{i+1}-\Pi_{i+1}\Big\|_{L^1\rightarrow L^1}+\Big\|\sol^{V,\kappa,\T^2}_{s^i_{j+1},t}\Pi_{i+1}-\Pi_{i+1}\Big\|_{L^1\rightarrow L^1}.
\end{align}
Using that $j\geq i+1$, Lemma~\ref{lem:Pi_k_preservation} implies that 

\begin{align}\label{eq:null_steps_error}
\notag \sum_{\ell=j+1}^{k-1}\Big\|\sol^{V,\kappa,\T^2}_{s^i_{\ell+1},s^i_{\ell}}\Pi_{i+1}-\Pi_{i+1}\Big\|_{L^1\rightarrow L^1}+\Big\|\sol^{V,\kappa,\T^2}_{s^i_{j+1},t}\Pi_{i+1}-\Pi_{i+1}\Big\|_{L^1\rightarrow L^1}&\leq C\sum_{\ell\geq i+1}(2^{i+1-(1-\alpha)\ell/2}\kappa)^{1/3}
 \\&\leq C(2^{i(1+\alpha)/2}\kappa)^{1/3},
\end{align}
where we have again used the definition of $\sigma_i$.
By the continuity of the operator
\[\lim_{k\rightarrow \infty}\Big\|\sol^{V,\kappa,\T^2}_{s^i_\infty,s^i_k}\Pi_{i+1}-\Pi_{i+1}\Big\|_{L^1\rightarrow L^1}=0,\]
we conclude by combining~\eqref{eq:operator_splitting_no-change},~\eqref{eq:null_steps_error}, and~\eqref{eq:diffusion_error} and sending $k$ to infinity.
\end{proof}

We note the following combination of Proposition~\ref{prop:rescaled-two-cell-dissipators} and Proposition~\ref{prop:no-change-by-partial-total-dissipator}, which will be useful for iterating in the proof of Theorem~\ref{thm:limit-is-close-pointwise-ell-dependent}.

\begin{corollary}
    \label{cor:partial-total-dissipator-works}
    There exists $C(\alpha)>0$ such that for all $i \in \N$ and $\kappa \in (0,1)$, we have the bound
    \[\Big\|\sol^{V,\kappa,\T^2}_{s^{i+1}_{i+1}, s^i_i} \Pi_{i+1} - \Pi_i\Big\|_{L^1(\T^2) \to L^1(\T^2)} \leq  C (2^{i(1+\alpha)/2} \kappa)^{(1-\alpha)/12}.\]
\end{corollary}

\begin{proof}
    Note that
    \[\sol^{V,\kappa,\T^2}_{s^{i+1}_{i+1}, s^i_i} \Pi_{i+1} - \Pi_i =\sol^{V,\kappa,\T^2}_{s^i_{i+1},s^i_i} \Big(\sol^{V,\kappa,\T^2}_{s^{i+1}_{i+1}, s^i_{i+1}} \Pi_{i+1} - \Pi_{i+1}\Big) + \sol^{V,\kappa,\T^2}_{s^i_{i+1},s^i_i} \Pi_{i+1} - \Pi_i,\]
    so 
    \begin{align*}\Big\|\sol^{V,\kappa,\T^2}_{s^{i+1}_{i+1}, s^i_i} \Pi_{i+1} - \Pi_i\Big\|_{L^1 \to L^1} &\leq \Big\| \sol^{V,\kappa,\T^2}_{s^{i+1}_{i+1}, s^i_{i+1}} \Pi_{i+1} - \Pi_{i+1}\Big\|_{L^1 \to L^1} + \Big\|\sol^{V,\kappa,\T^2}_{s^i_{i+1},s^i_i} \Pi_{i+1} - \Pi_i\Big\|_{L^1 \to L^1}
    \\&\leq C (2^{i(1+\alpha)/2} \kappa)^{1/3} +  C (2^{i(1+\alpha)/2} \kappa)^{(1-\alpha)/12} 
    \\&\leq C (2^{i(1+\alpha)/2} \kappa)^{(1-\alpha)/12},
    \end{align*}
    where we use Proposition~\ref{prop:no-change-by-partial-total-dissipator} for the first term and Proposition~\ref{prop:rescaled-two-cell-dissipators} for the second.
\end{proof}

\subsection{Proof of Theorem~\ref{thm:limit-is-close-pointwise-ell-dependent}}

We are now ready to prove Theorem~\ref{thm:limit-is-close-pointwise-ell-dependent}. It follows fairly directly now, combining Proposition~\ref{prop:rescaled-two-cell-dissipators}, Proposition~\ref{prop:first-time-close-to-proj}, Proposition~\ref{prop:no-change-by-partial-total-dissipator}, and Corollary~\ref{cor:partial-total-dissipator-works}.

\begin{proof}[Proof of Theorem~\ref{thm:limit-is-close-pointwise-ell-dependent}]
    We write
    \begin{align*}
    \sol_{0,t}^{V,\kappa,\T^2}-\sola_t e^{\kappa \sigma_\ell\Delta}\sol^{V,\kappa,\T^2}_{0,s^{\ell+1}_{\ell+1}} &= \sol^{V,\kappa,\T^2}_{s^{i+1}_{i+1},t} \sol_{s^\ell_\ell, s^{i+1}_{i+1
}}^{V,\kappa,\T^2}\Big(\sol_{s^{\ell+1}_{\ell+1}, s^\ell_\ell}^{V,\kappa,\T^2} 
 - \Pi_\ell e^{\kappa \sigma_\ell\Delta}\Big)\sol^{V,\kappa,\T^2}_{0,s^{\ell+1}_{\ell+1}}
 \\&\qquad+ \sol^{V,\kappa,\T^2}_{s^{i+1}_{i+1},t}\Big( \sol_{s^\ell_\ell, s^{i+1}_{i+1
}}^{V,\kappa,\T^2}\Pi_\ell - \Pi_{i+1}\Big) e^{\kappa \sigma_\ell \Delta}\sol^{V,\kappa,\T^2}_{0,s^{\ell+1}_{\ell+1}}
\\&\qquad + \Big(\sol^{V,\kappa,\T^2}_{s^{i+1}_{i+1},t}\Pi_{i+1} -\sola_t\Big)e^{\kappa \sigma_\ell \Delta}\sol^{V,\kappa,\T^2}_{0,s^{\ell+1}_{\ell+1}}.
    \end{align*}
    Thus, using that $\Big\|\sol^{V,\kappa,\T^2}_{0,s^{\ell+1}_{\ell+1}}\Big\|_{TV \to L^1} \leq 1,$
\begin{align*}
    &\Big\|\sol^{V,\kappa,\T^2}_{0,t} -\sol^{V,\kappa,\T^2}_{s^{i+1}_{i+1},t}\Pi_{i+1} e^{\kappa \sigma_\ell\Delta}\sol^{V,\kappa,\T^2}_{0,s^{\ell+1}_{\ell+1}}\Big\|_{TV \to L^1} 
    \\&\qquad\leq  \Big\|\sol_{s^{\ell+1}_{\ell+1}, s^\ell_\ell}^{V,\kappa,\T^2} 
 - \Pi_\ell e^{\kappa \sigma_\ell \Delta}\Big\|_{L^1 \to L^1} + \Big\|\sol_{s^\ell_\ell, s^{i+1}_{i+1
}}^{V,\kappa,\T^2}\Pi_\ell - \Pi_{i+1}\Big\|_{L^1 \to L^1}+ \Big\|\sol^{V,\kappa,\T^2}_{s^{i+1}_{i+1},t}\Pi_{i+1} -\sola_t\Big\|_{L^1 \to L^1}.
    \end{align*}
    First, by Proposition~\ref{prop:first-time-close-to-proj}, we have
    \begin{equation}
        \label{eq:close-to-ideal-ell-term-1}
        \|\sol_{s^{\ell+1}_{\ell+1}, s^\ell_\ell}^{V,\kappa,\T^2} 
 - \Pi_\ell e^{\kappa \sigma_\ell\Delta}\|_{L^1 \to L^1}  \leq C (2^{\ell(1+\alpha)/2} \kappa)^{(1-\alpha)^2/8}.
    \end{equation}
    Next we note that Corollary~\ref{cor:partial-total-dissipator-works} implies that
    \begin{align}
    \notag
         \Big\|\sol_{s^\ell_\ell, s^{i+1}_{i+1
}}^{V,\kappa,\T^2}\Pi_\ell - \Pi_{i+1}\Big\|_{L^1 \to L^1}&= \Big\|\sum_{k=i+1}^{\ell-1} \sol_{s^{k}_{k}, s^{i+1}_{i+1}}(\sol_{s^{k+1}_{k+1}, s^{k}_{k}} \Pi_{k+1} - \Pi_{k})\Big\|_{L^1 \to L^1}
\\&\leq
\notag\sum_{k=i+1}^{\ell-1} \Big\|\sol_{s^{k+1}_{k+1}, s^{k}_{k}} \Pi_{k+1} - \Pi_{k}\Big\|_{L^1 \to L^1}
\\&\leq  \sum_{k=i+1}^{\ell-1} C (2^{k(1+\alpha)/2} \kappa)^{(1-\alpha)/12} \leq C (2^{\ell(1+\alpha)/2} \kappa)^{(1-\alpha)/12}.
\label{eq:close-to-ideal-ell-term-2}
    \end{align}
    Finally, we note that
    \[\sola_{t} = \begin{cases}
       \Pi_{i+1}& t\in [s_{i+1}^{i+1}, s^i_{i+1}]\\
        \sol^{V,0,\T^2}_{s^i_{i+1},t} \Pi_{i+1}& t \in [s_{i+1}^{i}, s_i^i].
    \end{cases}\]
    Thus for $t \in [s_{i+1}^{i+1}, s^i_{i+1}],$ we have by Proposition~\ref{prop:no-change-by-partial-total-dissipator}
    \begin{equation}
    \label{eq:close-to-ideal-ell-term-3-1}
        \Big\|\sol^{V,\kappa,\T^2}_{s^{i+1}_{i+1},t}\Pi_{i+1} -\sola_t\Big\|_{L^1 \to L^1} = \Big\|\sol^{V,\kappa,\T^2}_{s^{i+1}_{i+1},t}\Pi_{i+1} -\Pi_{i+1}\Big\|_{L^1 \to L^1} \leq C(2^{i(1+\alpha)/2}\kappa)^{1/3}.
    \end{equation}
    For $t \in [s_{i+1}^i, s_i^i],$ using Proposition~\ref{prop:no-change-by-partial-total-dissipator} and Proposition~\ref{prop:rescaled-two-cell-dissipators}, we have
    \begin{align}
        \notag &\Big\|\sol^{V,\kappa,\T^2}_{s^{i+1}_{i+1},t}\Pi_{i+1} -\sola_t\Big\|_{L^1 \to L^1}
        \\\notag&\quad\quad\leq  \Big\|\sol^{V,\kappa,\T^2}_{s^{i}_{i+1},t}(\sol^{V,\kappa,\T^2}_{s^{i+1}_{i+1},s^{i}_{i+1}}\Pi_{i+1} - \Pi_{i+1})\Big\|_{L^1 \to L^1} +\Big\|(\sol^{V,\kappa,\T^2}_{s^{i}_{i+1},t}  -\sol^{V,0,\T^2}_{s^i_{i+1},t}) \Pi_{i+1}\Big\|_{L^1 \to L^1} 
        \\&\quad\quad\notag \leq  C(2^{i(1+\alpha)/2}\kappa)^{1/3} + C(2^{i(1+\alpha)/2} \kappa)^{(1-\alpha)/12}\big|\sigma_i^{-1}(t- s^i_{i+1} - \sigma_i/2)\big|^{-1/2}
        \\&\quad\quad\leq C(2^{i(1+\alpha)/2} \kappa)^{(1-\alpha)/12}\big|\sigma_i^{-1}(t- s^i_{i+1} - \sigma_i/2)\big|^{-1/2}.
        \label{eq:close-to-ideal-ell-term-3-2}
    \end{align}
    Putting together (\ref{eq:close-to-ideal-ell-term-1}--\ref{eq:close-to-ideal-ell-term-3-2}) and using that $\ell \geq i$, we conclude.
\end{proof}

\subsection{Proof of Corollary~\ref{cor:is-close-to-ideal-pointwise}}

We now consider Corollary~\ref{cor:is-close-to-ideal-pointwise}. The majority of the work will be done by Theorem~\ref{thm:limit-is-close-pointwise-ell-dependent}, but we need some results that will allow us to get rid of the $e^{\sigma_\ell \Delta}\sol^{V,\kappa,\T^2}_{0,s^{\ell+1}_{\ell+1}}$ term in Theorem~\ref{thm:limit-is-close-pointwise-ell-dependent}. This will be provided by Lemma~\ref{lem:soln-operator-doesnt-change-averages} below, but first we recall an estimate for drift-diffusion equations from the earlier work~\cite{hess-childs_universal_2025}. This lemma states that, provided the drift and diffusion are sufficiently small, the drift-diffusion equation leaves the average over the domain essentially invariant, regardless of the boundary data. We recall the notation of Definition~\ref{def:average-defn} used below.

\begin{lemma}[{\cite[Lemma 3.4]{hess-childs_universal_2025}}]
    \label{lem:mean-preservation}
    There exists $C>0$ so that for any divergence-free vector field $u \in L^1\big([0,1],L^\infty(B)\big)$, $\kappa \in (0,1)$, boundary data $f\in L^\infty([0,1]\times\partial B)$, and initial datum $\theta_0$, it holds that
    \[\Big|\big(\sol_{0,1}^{u,\kappa,f} \theta_0\big)_B- \big(\theta_0\big)_B\Big| \leq C\big(\|\theta_0\|_{L^\infty(B)}+\|f\|_{L^\infty([0,1]\times\partial B)}\big)\big(\|u\|_{L^1([0,1],L^\infty(B))}+\kappa^{1/2} \log (\kappa^{-1})\big).\]
\end{lemma}

By splitting our domain into small grid cells and applying Lemma~\ref{lem:mean-preservation} to each cell, we get the following lemma, which provides the missing component to go from Theorem~\ref{thm:limit-is-close-pointwise-ell-dependent} to Corollary~\ref{cor:is-close-to-ideal-pointwise}. We note that this estimate is given in $L^\infty$ as opposed to $L^1$ and in fact fails from $L^1 \to L^1$ as can be seen by placing an approximate $\delta$ mass near the boundary of two grid cells. We will actually want to use the estimate from $L^2 \to L^2$, but that will be a direct consequence of the $L^\infty \to L^\infty$ estimate and Riesz--Thorin interpolation.

\begin{lemma}
\label{lem:soln-operator-doesnt-change-averages}
There exists $C(\alpha)>0$ so that for any $i\geq k$ and $\kappa\in(0,1)$ it holds that
\[\Big\|\Pi_k\sol^{V,\kappa,\T^2}_{0,s^{i-1}_\infty}-\Pi_k\Big\|_{L^\infty(\T^2)\rightarrow L^\infty(\T^2)}\leq C\big((2^{k-i})^{1/2} +(2^{k-(1-\alpha)i/2}\kappa)^{1/3}\big).\]
\end{lemma}

\begin{proof}
For an initial data $\theta_0$
\[\Big\|\Pi_k\sol^{V,\kappa,\T^2}_{0,s^{i-1}_\infty}\theta_0-\Pi_k\theta_0\Big\|_{L^\infty}\leq \sup_{x_k\in\Lambda_k} \Big|\Big(\sol^{V,\kappa,\T^2}_{0,s^{i-1}_\infty}\theta_0\Big)_{A_k+x_k}-(\theta_0)_{A_k+x_k}\Big|.\]
Fixing $x_k$, then
\[\Big(\sol^{V,\kappa,\T^2}_{0,s^{i-1}_\infty}\theta_0\Big)_{x_k+A_k}=\Big(\sol^{U,s^{i-1}_\infty2^k\kappa,g^{x_k}}_{0,1}\tilde{\theta_0}\Big)_{B}\]
for the initial data $\tilde\theta(y)=\theta_0(\mathcal{R}^{-k}(y-x_k))$, boundary data $g^{x_k}:[0,T]\times \partial B\rightarrow\R$ given by
\[g^{x_k}(t,y)=\sol_{0,s^{i-1}_\infty t}^{V,\kappa,\T^2}\theta_0(\mathcal{R}^{-k}(y-x_k)),\]
and vector field $U:[0,1]\times B\rightarrow\R^2$ given by
\[U(t,x):=s^{i-1}_\infty \mathcal{R}^{-k}V(s^{i-1}_\infty t,\mathcal{R}^kx+x_k)).\]
By Lemma~\ref{lem:mean-preservation} we thus find that
\begin{align}\label{eq:lemma3.4_bound}
 \notag&\Big|\Big(\sol^{V,\kappa,\T^2}_{0,s^{i-1}_\infty}\theta_0\Big)_{A_k+x_k}-(\theta_0)_{A_k+x_k}\Big|=\Big|\Big(\sol^{U,s^{i-1}_\infty2^k\kappa,g^{x_k}}_{0,1}\tilde\theta_0\Big)_{B}-(\tilde\theta_0)_{B}\Big|
 \\&\qquad\qquad\leq C\big(\|\tilde\theta_0\|_{L^\infty(B)}+\|g^{x_k}\|_{L^\infty([0,1]\times\partial B)}\big)\big(\|U\|_{L^1([0,1],L^\infty(B))}+(s^{i-1}_\infty2^k\kappa)^{1/3}\big). 
\end{align}
Since $\sol^{V,\kappa,\T^2}_{0,t}$ is an $L^\infty$ contraction,
\begin{equation}\label{eq:Linfty_contraction_bound}
\|\tilde\theta_0\|_{L^\infty(B)}+\|g^{x_k}\|_{L^\infty([0,1]\times\partial B)}\leq \|\theta_0\|_{L^\infty}.
\end{equation}
On the other hand, using the definition of $U$ and $V$, we have that
\begin{align*}
\|U\|_{L^1([0,1], L^\infty(\T^d))}&= s^{i-1}_\infty2^{k/2}\int_0^1 \|V(s^{i-1}_\infty t,\cdot)\|_{L^\infty}\,dt
\\&=2^{k/2}\int_0^{s^{i-1}_\infty}\|V(t,\cdot)\|_{L^\infty}\,dt
\\&=2^{k/2} \sum_{j\geq i}\sum_{\ell\geq j}\sigma_{\ell}^{-1}2^{-\ell/2}\int_{s^j_{\ell+1}}^{s^j_{\ell}}\big\|v(\sigma_{\ell}^{-1}(t-s^j_{\ell+1}),\cdot)\big\|_{L^\infty}\,dt
\\&=2^{k/2} \sum_{j\geq i}\sum_{\ell\geq j}2^{-\ell/2}\int_0^1\|v(t,\cdot)\|_{L^\infty}\,dt
\\&\leq C(2^{k-i})^{1/2}
\end{align*}
where we use that $\int_0^1\|v(t,\cdot)\|_{L^\infty}\,dt\leq C$. Combining this with~\eqref{eq:lemma3.4_bound} and~\eqref{eq:Linfty_contraction_bound} and using that $\s^{i-1}_\infty\leq C2^{-(1-\alpha)i/2}$ we conclude the bound.
\end{proof}

We can now prove Corollary~\ref{cor:is-close-to-ideal-pointwise} by combining Theorem~\ref{thm:limit-is-close-pointwise-ell-dependent} and Lemma~\ref{lem:soln-operator-doesnt-change-averages}. However the first estimate is $L^1 \to L^1$ and the second is $L^\infty \to L^\infty$. As such, we ``meet in the middle'' and prove an estimate from $L^2 \to L^2$. This estimates is obtained by using Riesz-Thorin interpolation to interpolate Theorem~\ref{thm:limit-is-close-pointwise-ell-dependent} and Lemma~\ref{lem:soln-operator-doesnt-change-averages} with the trivial estimates on the opposite endpoints.

\begin{proof}[Proof of Corollary~\ref{cor:is-close-to-ideal-pointwise}]
    Let $t \in [0,1]$ and $i \in \N$ such that $t \in [s^{i+1}_{i+1}, s^i_i]$. Note that the lower bound on $t$ implies that 
    \[i \leq \lceil \tfrac{1}{8} \log_{\sqrt{2}}(\kappa^{-1}) \rceil + C.\]
    Define 
    \[\ell := \lceil \tfrac{1}{4} \log_{\sqrt{2}}(\kappa^{-1}) \rceil + C \in \N,\]
    so that $\ell > i$.  

    We note that $\Big\|\sol^{V,\kappa,\T^2}_{0,t} -\sola_t e^{\kappa \sigma_\ell\Delta}\sol^{V,\kappa,\T^2}_{0,s^{\ell+1}_{\ell+1}} \Big\|_{L^\infty \to L^\infty} \leq 2$, so by Riesz-Thorin interpolation with Theorem~\ref{thm:limit-is-close-pointwise-ell-dependent},
    \begin{align*}
    \Big\|\sol^{V,\kappa,\T^2}_{0,t} -\sola_t e^{\kappa \sigma_\ell\Delta}\sol^{V,\kappa,\T^2}_{0,s^{\ell+1}_{\ell+1}} \Big\|_{L^2 \to L^2} &\leq \begin{cases}
      C (2^{\ell(1+\alpha)/2} \kappa)^{\frac{(1-\alpha)^2}{24}}  & t \in [s^{i+1}_{i+1},s^i_{i+1}],\\
         C (2^{\ell(1+\alpha)/2} \kappa)^{\frac{(1-\alpha)^2}{24}}\Big|\frac{t- s^i_{i+1} - \sigma_i/2}{\sigma_i}\Big|^{-\frac{1}{4}}& t \in [s^i_{i+1}, s^i_{i}],
    \end{cases}
    \\&\leq  \begin{cases}
     C\kappa^{\frac{(1-\alpha)^2}{48}}  & t \in [s^{i+1}_{i+1},s^i_{i+1}],\\
         C\kappa^{\frac{(1-\alpha)^2}{48}}\Big|\frac{t- s^i_{i+1} - \sigma_i/2}{\sigma_i}\Big|^{-\frac{1}{4}}& t \in [s^i_{i+1}, s^i_{i}].
    \end{cases}
    \end{align*}
    We also note that $\Big\|\Pi_{i+1} \sol^{V,\kappa,\T^2}_{0,s^{\ell}_{\infty}} - \Pi_{i+1}\Big\|_{L^1 \to L^1} \leq 2$, so by Lemma~\ref{lem:soln-operator-doesnt-change-averages} and Riesz-Thorin interpolation,
    \[\Big\|\Pi_{i+1} \sol^{V,\kappa,\T^2}_{0,s^{\ell}_{\infty}} - \Pi_{i+1}\Big\|_{L^2 \to L^2}\leq C2^{(i-\ell)/4} +C(2^{i-(1-\alpha)\ell/2}\kappa)^{1/4} \leq C \kappa^{1/16}.\]
    Then
    \[ \big\|\sol^{V,\kappa,\T^2}_{0,t} -\sola_t\big\|_{L^2 \to L^2}\leq \Big\|\sol^{V,\kappa,\T^2}_{0,t} -\sola_t e^{\kappa \sigma_\ell \Delta}\sol^{V,\kappa,\T^2}_{0,s^{\ell+1}_{\ell+1}} \Big\|_{L^2 \to L^2} +\Big \|\sola_t \sol^{V,\kappa,\T^2}_{0,s^{\ell}_{\infty}} - \sola_t \Big\|_{L^2 \to L^2},\]
    and
    \[\Big\|\sola_t \sol^{V,\kappa,\T^2}_{0,s^{\ell}_{\infty}} - \sola_t \Big\|_{L^2 \to L^2} \leq \Big\|\Pi_{i+1} \sol^{V,\kappa,\T^2}_{0,s^{\ell}_{\infty}} - \Pi_{i+1} \Big\|_{L^2 \to L^2}.\]
    Thus combining the four displays above and using that $\kappa^{1/16} \leq \kappa^{(1-\alpha)^2/48}$, we conclude~\eqref{eq:soln-operator-close-to-ideal}.

    For~\eqref{eq:soln-operator-close-to-ideal-integrated}, we define the set of ``bad'' times
    \[B := [0,\kappa^{(1-\alpha)/8}] \cup \bigcup_{i=0}^\infty \big\{t \in [s^i_{i+1}, s^i_i : |\sigma_i^{-1} (t- s^i_{i+1} - \sigma_i/2)|^{-1/4} \geq \kappa^{(1-\alpha)^2/96}\big\}.\]
    We note that
    \[|B| \leq \kappa^{(1-\alpha)/8} + \kappa^{(1-\alpha)^2/24}\sum_{i=0}^\infty \sigma_i \leq C \kappa^{(1-\alpha)^2/24}.\]
    Then we note that for $t \in [0,1] - B$, we have that
    \[\big\|\sol^{V,\kappa,\T^2}_{0,t} -\sola_t\big\|_{L^2 \to L^2} \leq C \kappa^{(1-\alpha)^2/96}.\]
    For $t \in B$, we take the trivial bound $\big\|\sol^{V,\kappa,\T^2}_{0,t} -\sola_t\big\|_{L^2 \to L^2} \leq 2$. Then we compute
    \begin{align*}
    &\big\|\sol^{V,\kappa,\T^2}_{0,t} -\sola_t\big\|_{L^2 \to L^2([0,1],L^2)} =  \Big(\int_0^1 \|\sol^{V,\kappa,\T^2}_{0,t} -\sola_t\|_{L^2 \to L^2}^2\,dt\Big)^{1/2}
    \\&\qquad\leq \Big(\int_B \|\sol^{V,\kappa,\T^2}_{0,t} -\sola_t\|_{L^2 \to L^2}^2\,dt\Big)^{1/2} + \Big(\int_{[0,1] - B} \|\sol^{V,\kappa,\T^2}_{0,t} -\sola_t\|_{L^2 \to L^2}^2\,dt\Big)^{1/2}
    \\&\qquad \leq  2|B|^{1/2} + C\kappa^{(1-\alpha)^2/96} \leq C \kappa^{\frac{(1-\alpha)^2}{96}},
    \end{align*}
    as desired.    
\end{proof}

\section{Richardson dispersion for the asymptotic total dissipator}

\label{s:richardson}

With Theorem~\ref{thm:limit-is-close-pointwise-ell-dependent} and Corollary~\ref{cor:is-close-to-ideal-pointwise} in hand, the most involved technical step is behind us, and we are ready to start proving the turbulent phenomena that are the primary interest of this work. The first phenomenon we consider is that of Richardson dispersion, codified in Theorem~\ref{thm:richardson}. In order to prove Theorem~\ref{thm:richardson}, we will use Theorem~\ref{thm:limit-is-close-pointwise-ell-dependent}, which will give that down to the relevant time scale $\kappa^{\frac{1-\alpha}{1+\alpha}}$, the true solution is suitably close in $L^1$ to a function for which we can explicitly bound the variance. That will give the $t^{\frac{2}{1-\alpha}}$ term of Theorem~\ref{thm:richardson}. We also need to get the $\kappa t$ term, which is the dominant term on very short time scales. For that, the following $L^\infty$ decay estimate will be useful, which is standard for advection-diffusion equations.
\begin{proposition}[{\cite{nashContinuitySolutionsParabolic1958}}]
\label{prop:L1-Linfty-smoothing}
    Let $u \in L^\infty([0,1] \times \T^2)$ with $\nabla \cdot u =0$. Then for $\kappa \in (0,1), t>0,$ 
    \[\big\|\sol^{u,\kappa,\T^2}_{0,t}\big\|_{TV(\T^2) \to L^\infty(\T^2)} \leq C ((\kappa t)^{-1} +1).\]
\end{proposition}

We next note the following straightforward lower bound on the variance of a measure by the inverse $L^\infty$ norm of it's density. Note here we are using two-dimensionality in our computations and in general the appropriate power on the right hand side is dimensionally dependent. Recall that $\mathcal{P}(\T^2)$ is the space of probability measures on $\T^2$.

\begin{lemma}
\label{lem:Var-lower-bound-L-infty}
    Suppose $\mu \in \mathcal{P}(\T^2)$. Then
    \[\Var(\mu) \geq C^{-1} \|\mu\|_{L^\infty(\T^2)}^{-1}.\]
\end{lemma}

\begin{proof}
    Let $a \in \T^2$ arbitrary. Then for any $r>0$,
    \begin{align*}\int |x-a|^2\,d\mu(x) \geq \int_{\T^2 - B_r(a)} |x-a|^2\,d\mu(x) &\geq r^2 \mu(\T^2 - B_r(a)) 
    \\&= r^2 (1 - \mu(B_r(a))) \geq r^2 (1 - C r^2 \|\mu\|_{L^\infty}).
    \end{align*}
    Thus choosing $r = \frac{1}{2} C^{-1/2} \|\mu\|_{L^\infty}^{-1/2}$, we conclude.
\end{proof}

The following proposition, which is the central tool used to lower bound the variance in Theorem~\ref{thm:richardson}, gives that if $\mu$ has some non-trivial overlap with $\nu$, then we can lower bound the variance of $\mu$ by the inverse $L^\infty$ norm of the density of $\nu$.
 
\begin{proposition}
\label{prop:TV-close-Var-lower-bound}
    Let $\mu, \nu \in \mathcal{P}(\T^2)$ and suppose that $\|\mu -\nu\|_{TV(\T^2)} \leq 1$. Then 
    \[\Var(\mu) \geq C^{-1} \|\nu\|_{L^\infty(\T^2)}^{-1}.\]
\end{proposition}

Proposition~\ref{prop:TV-close-Var-lower-bound} is a direct consequence of Lemma~\ref{lem:Var-lower-bound-L-infty}, Lemma~\ref{lem:TV-close-decomp}, and Lemma~\ref{lem:Var-convex-sum} below.

\begin{lemma}
\label{lem:TV-close-decomp}
    Let $\mu,\nu \in \mathcal{P}(\T^2)$ and suppose $\|\mu - \nu\|_{TV(\T^2)} \leq 1$. Then we can find a decomposition of $\mu$
    \[\mu = \tfrac{1}{2}\alpha + \tfrac{1}{2}\beta,\]
    with $\alpha,\beta \in \mathcal{P}(\T^2)$ and $\|\alpha\|_{L^\infty(\T^2)} \leq \|\nu\|_{L^\infty(\T^2)}.$
\end{lemma}
\begin{proof}
    Let $\gamma, \lambda$ be the positive measures giving the Hahn-Jordan decomposition of $\mu -\nu$, so that
    \[\mu - \nu = \gamma - \lambda \quad \text{and} \quad \gamma(\T^2) + \lambda(\T^2) =  \|\mu - \nu\|_{TV} \leq 1.\]
    Note that $0=(\mu-\nu)(\T^2) = (\gamma - \lambda)(\T^2),$ thus $\gamma(\T^2) = \lambda(\T^2) \leq \frac{1}{2}.$

    We also have that $\mu-\gamma, \nu -\lambda$ are positive measures since $\mu - \gamma =\nu -\lambda$ and their negative parts---$\gamma, \lambda$ respectively---are mutually singular. Then we write
    \[\mu =  \frac{1}{2} \frac{(\nu - \lambda)}{(\nu- \lambda)(\T^2)} + \frac{1}{2} \Big(  \frac{2(\nu-\lambda)(\T^2)-1}{(\nu- \lambda)(\T^2)} (\nu - \lambda) + 2\gamma\Big) =: \tfrac{1}{2}\alpha + \tfrac{1}{2}\beta.\]
    Then we note that as $(\nu - \lambda)(\T^2) \geq 1 - \lambda(\T^2) \geq \frac{1}{2}$, we have that $\alpha, \beta \in \mathcal{P}(\T^2)$. Further, since $\nu - \lambda \geq 0$ and $\lambda \geq 0,$ we also have that
    \[\|\alpha\|_{L^\infty} \leq \|\nu -\lambda\|_{L^\infty} \leq \|\nu\|_{L^\infty},\]
    allowing us to conclude.
\end{proof}

\begin{lemma}
\label{lem:Var-convex-sum}
    Let $\mu,\alpha,\beta \in \mathcal{P}(\T^2)$ and suppose $\mu = \frac{1}{2}\alpha + \frac{1}{2} \beta$. Then 
    \[\Var(\mu) \geq \tfrac{1}{2} \Var(\alpha).\]
\end{lemma}

\begin{proof}
    Let $a \in \T^2$ arbitrary, then 
    \[\int |x-a|^2\,d\mu(x) \geq \frac{1}{2}\int |x-a|^2\,\alpha(dx) \geq \tfrac{1}{2} \Var(\alpha).\]
    Taking the infimum over $a$, we conclude.
\end{proof}

We are now ready to prove Theorem~\ref{thm:richardson} by combining Proposition~\ref{prop:L1-Linfty-smoothing} for short times, Theorem~\ref{thm:limit-is-close-pointwise-ell-dependent} for long times, and Proposition~\ref{prop:TV-close-Var-lower-bound} to give a lower bound on the variance.

\begin{proof}[Proof of Theorem~\ref{thm:richardson}]
    We note that since~\eqref{eq:intro_drift_diffusion_equation} is the Fokker-Planck equation for~\eqref{eq:stochastic-intro}, we have that
    \[\Var(X^\kappa_t) = \Var\big(\sol^{V,\kappa,\T^2}_{0,t} \delta_x\big).\]
    Then by Proposition~\ref{prop:L1-Linfty-smoothing} and Lemma~\ref{lem:Var-lower-bound-L-infty},
    \begin{equation}
    \label{eq:richardson-short-time}
    \Var\big(\sol^{V,\kappa,\T^2}_{0,t} \delta_x\big) \geq C^{-1} \kappa t.
    \end{equation}

    For any $t \leq C\kappa^{\frac{1-\alpha}{1+\alpha}}$, 
    \[t^{\frac{2}{1-\alpha}} = tt^{\frac{1+\alpha}{1-\alpha}} \leq Ct \kappa.\]
    Therefore, by~\eqref{eq:richardson-short-time}, in order to conclude, it suffices to prove that for all $t \geq K\kappa^{\frac{1-\alpha}{1+\alpha}}$, for some $K>0$ to be determined, 
    \[\Var\big(\sol^{V,\kappa,\T^2}_{0,t} \delta_x\big) \geq C^{-1} t^{\frac{2}{1-\alpha}}.\]
    Let $t \geq K\kappa^{\frac{1-\alpha}{1+\alpha}}$ and let $i \in \N$ such that $t \in [s^{i+1}_{i+1}, s^i_i]$. Then, applying Theorem~\ref{thm:limit-is-close-pointwise-ell-dependent} with $\ell = i+1$, we have that
    \[\Big\|\Big(\sol^{V,\kappa,\T^2}_{0,t \land s^i_{i+1}} -\Pi_{i+1}\sol^{V,\kappa,\T^2}_{0,s^{i+2}_\infty}\Big)\delta_x\Big\|_{L^1} \leq
      C (2^{i(1+\alpha)/2} \kappa)^{(1-\alpha)^2/12}.\]
    Note that since $t \geq K \kappa^{\frac{1-\alpha}{1+\alpha}}$, then $i \leq \lceil \frac{1}{1+\alpha} \log_{\sqrt{2}}(\kappa^{-1}) \rceil - N(K)$ for some $N(K) \in \N$ such that $N(K) \to \infty$ as $K \to\infty$. Thus
    \begin{equation}
    \label{eq:richardson-close-to-ideal}
    \Big\|\Big(\sol^{V,\kappa,\T^2}_{0,t \land s^i_{i+1}} -\Pi_{i+1}\sol^{V,\kappa,\T^2}_{0,s^{i+2}_\infty}\Big)\delta_x\Big\|_{L^1} \leq
      C ( 2^{-N(K)}\kappa^{-1} \kappa)^{(1-\alpha)^2/12} \leq 1,
      \end{equation}
    provided we choose $K$ sufficiently large so that $N(K)$ is sufficiently large.

    Then for $t \in [s^{i+1}_{i+1}, s^i_{i+1}]$, we have from~\eqref{eq:richardson-close-to-ideal}, Proposition~\ref{prop:TV-close-Var-lower-bound}, and~\eqref{eq:pi-l1-linfty-bound},
    \[\Var\big(\sol^{V,\kappa,\T^2}_{0,t} \delta_x\big) \geq C^{-1} \Big\|\Pi_{i+1} \sol^{V,\kappa,\T^2}_{0,s^{i+2}_\infty}\delta_x \Big\|_{L^\infty}^{-1} \geq C^{-1} \|\Pi_{i+1}\|_{L^1 \to L^\infty}^{-1} \geq C^{-1} 2^{-i} \geq C^{-1}t^{\frac{2}{1-\alpha}},\]
    as desired. For $t \in [s^i_{i+1},s^i_i]$, we write
    \[\sol^{V,\kappa,\T^2}_{0,t} \delta_x = \sol^{V,\kappa,\T^2}_{s^i_{i+1},t} \Big(\sol^{V,\kappa,\T^2}_{0,s^i_{i+1}}  - \Pi_{i+1} \sol^{V,\kappa,\T^2}_{0,s^{i+2}_\infty}\Big)\delta_x  + \sol^{V,\kappa,\T^2}_{s^i_{i+1},t}  \Pi_{i+1} \sol^{V,\kappa,\T^2}_{0,s^{i+2}_\infty}\delta_x.\]
    Thus by~\eqref{eq:richardson-close-to-ideal},
    \[\Big\|\sol^{V,\kappa,\T^2}_{0,t} \delta_x -  \sol^{V,\kappa,\T^2}_{s^i_{i+1},t}  \Pi_{i+1} \sol^{V,\kappa,\T^2}_{0,s^{i+2}_\infty}\delta_x\Big\|_{L^1} \leq \Big\|\Big(\sol^{V,\kappa,\T^2}_{0,s^i_{i+1}}  - \Pi_{i+1} \sol^{V,\kappa,\T^2}_{0,s^{i+2}_\infty}\Big)\delta_x\Big\|_{L^1} \leq 1.\]
Proposition~\ref{prop:TV-close-Var-lower-bound} and~\eqref{eq:pi-l1-linfty-bound} thus imply that
    \[\Var\big(\sol^{V,\kappa,\T^2}_{0,t} \delta_x\big) \geq C^{-1} \Big\| \sol^{V,\kappa,\T^2}_{s^i_{i+1},t}  \Pi_{i+1} \sol^{V,\kappa,\T^2}_{0,s^{i+2}_\infty}\delta_x\Big\|_{L^\infty}^{-1} \geq C^{-1} \|\Pi_{i+1}\|_{L^1 \to L^\infty}^{-1} \geq C^{-1} t^{\frac{2}{1-\alpha}},\]
    allowing us to conclude.
\end{proof}

\section{Fractional regularity spaces and interpolation}

\label{s:fractional-regularity-spaces}

In this section, we cover the relevant facts about fractional regularity spaces needed for Section~\ref{s:anomalous-regularization}. This treatment is far from comprehensive; see~\cite[Chapter 6]{bergh_interpolation_1976} for a standard and concise reference on fractional regularity spaces. We will need Corollary~\ref{cor:BV-embedding-bessel}, Proposition~\ref{prop:piecewise-constant-implies-not-in-some-spaces}, and most importantly Theorem~\ref{thm:interpolation-bound} in Section~\ref{s:anomalous-regularization}. We first introduce the definition of Besov norms. Besov norms are directly linked with the traditional measure of intermittency as given by structure functions, as noted in Subsection~\ref{sss:intermitency-intro}. 

\begin{definition}
   For any $f : \T^2\to \R$ with $\int f(x)\,dx =0$, $1 \leq p,q \leq \infty$, and $s \in (0,1)$, define the (homogeneous) Besov norm $B^{s,p}_q(\T^2)$ by
   \[\|f\|_{B^{s,p}_q(\T^2)} := \Big(\int_{\T^2} \frac{\|f(x) - f(x - h)\|_{L^p_x}^q}{|h|^{sq}}\, \frac{dh}{|h|^2}\Big)^{1/q},\]
   with the natural modification for $q =\infty$.
\end{definition}

For simplicity, we work primarily with Riesz potential norms to demonstrate intermittency. The following however shows that Riesz potential norms are comparable to Besov norms with the same $s$ and $p$ but varying $q$. One could thus deduce results about Besov regularity spaces from Theorem~\ref{thm:anomalous-regularization} and Theorem~\ref{thm:intermittency}. This relation is also useful as in many way Besov norms are more convenient to work with.

\begin{proposition}[{\cite[Theorem 6.4.4]{bergh_interpolation_1976} and~\cite[Theorem 2.5.12]{triebel_theory_1983}}]
\label{prop:BL-besov-bessel}
    For all $s \in (0,1)$ and $p \in (1,\infty)$, there exists $C(s,p)>0$ such that for all $f \in L^1(\T^2),$
    \[C^{-1} \|f\|_{B^{s,p}_\infty(\T^2)} \leq \|f\|_{H^{s,p}(\T^2)} \leq C\|f\|_{B^{s,p}_1(\T^2)}.\]
\end{proposition}

As an example of how Besov norms are easier to work with, the following embedding admits a rather straightforward proof, which we however still defer to Appendix~\ref{appendix:lemmas}.

\begin{proposition}
    \label{prop:BV-embedding-besov}
    For all $s \in [0,1)$ and all
    \[1 \leq p < \frac{2}{s+1},\]
    there exists $C(s,p)>0$ such that for all $f \in L^1(\T^2),$
    \[\|f\|_{B^{s,p}_1(\T^2)} \leq C \|f\|_{BV(\T^2)}.\]
\end{proposition}

The following is then direct from Proposition~\ref{prop:BV-embedding-besov} and Proposition~\ref{prop:BL-besov-bessel}. This will be useful as we will be able to explicitly bound BV norms but will want to control $H^{s,p}$ norms in order to use the interpolation results of Theorem~\ref{thm:interpolation-bound}.
\begin{corollary}
\label{cor:BV-embedding-bessel}
      For all $s \in [0,1)$ and all
    \[1 \leq p < \frac{2}{s+1},\]
    there exists $C(s,p)>0$ such that for all $f \in L^1(\T^2),$
    \[\|f\|_{H^{s,p}(\T^2)} \leq C \|f\|_{BV(\T^2)}.\]
\end{corollary}

We also defer the argument that a piecewise constant function fails to live in $H^{s,p}$ for $p>s^{-1}$ to Appendix~\ref{appendix:lemmas}. This is the essential fact that is needed for Theorem~\ref{thm:intermittency}.

\begin{proposition}
\label{prop:piecewise-constant-implies-not-in-some-spaces}
    Let $f : \T^2 \to \R$ and suppose that there exists a point $y \in \T^2$ and a radius $r>0$ such that
    \[f|_{B_r(y)} = a\indc_{A} + b \indc_B,\]
    for two sets $A \cap B = \emptyset, A \cup B = B_r(y)$, $|A|, |B| >0$ and two values $a \ne b$. Then for all $s \in (0,1)$ and $p > s^{-1}$,
    \[\|f\|_{H^{s,p}(\T^2)} = \|f\|_{ B^{s,p}_\infty(\T^2)} =\infty.\]
\end{proposition}

The following powerful interpolation theorem is the engine behind our proof of Theorem~\ref{thm:anomalous-regularization}, as explained in Subsection~\ref{ss:overview}.

\begin{theorem}[{\cite[Theorem 6.4.5]{bergh_interpolation_1976}}]
\label{thm:interpolation-bound}
    For all $s_0, s_1 \in \R$ with $s_0 \ne s_1$ as well as $p_0, p_1, q_0, q_1,$ and $\theta$ with $p_0, p_1 \in [1,\infty], q_0, q_1 \in (1,\infty),$ and $\theta \in (0,1)$, there exists $C(s_0,s_1,p_0,p_1,q_0,q_1,\theta)>0$ such that for all linear operators $T : L^{p_0}(\T^2) + L^{p_1}(\T^2) \to H^{s_0,q_0}(\T^2) + H^{s_1,q_1}(\T^2)$, we have the bound
    \[\|T\|_{L^{p_\theta}(\T^2) \to H^{s_\theta,q_\theta}(\T^2)} \leq C\|T\|_{L^{p_0}(\T^2) \to H^{s_0,q_0}(\T^2)}^{1-\theta} \|T\|_{L^{p_1}(\T^2) \to H^{s_1,q_1}(\T^2)}^\theta\]
    as well as, for any $f \in H^{s_0,q_0}(\T^2) + H^{s_1,q_1}(\T^2)$, the bound
    \[\|f\|_{H^{s_\theta,q_\theta}(\T^2)} \leq C\|f\|_{H^{s_0,q_0}(\T^2)}^{1-\theta} \|f\|_{H^{s_1,q_1}(\T^2)}^\theta,\]
    where
    \[p_\theta := (1-\theta) p_0 + \theta p_1, \quad q_\theta := (1-\theta) q_0 + \theta q_1,\quad\text{and}\quad s_\theta := (1-\theta) s_0 + \theta s_1.\]
\end{theorem}

\section{Anomalous regularization and intermittent regularity for the asymptotic total dissipator}

\label{s:anomalous-regularization}

In this section, we use the above facts about fractional regularity spaces together with Corollary~\ref{cor:is-close-to-ideal-pointwise} and the explicit form of the limiting solution operator $\sola_t$ to prove the anomalous regularization result of Theorem~\ref{thm:anomalous-regularization} and the intermittency result of Theorem~\ref{thm:intermittency}. 

The following lemma, whose proof is deferred to Appendix~\ref{appendix:lemmas}, gives a BV growth bound for transport equations with Lipschitz drifts. It is used in bounding the BV regularity of the limiting solution in Proposition~\ref{prop:limiting-solution-regularity}.

\begin{lemma}\label{lem:transport-bv-bv-bound}
Suppose that $u\in L^\infty([0,1],W^{1,\infty}(B))$ is divergence free and tangential to $\partial B$. Then for all $t\in[0,1]$ and any boundary data $f$
\[\big\|\sol^{u,0,f}_{0,t}\big\|_{BV(B)\rightarrow BV(B)}\leq \exp\bigg(\int_0^t \|\nabla u(s,\cdot)\|_{L^\infty(B)}\,ds\bigg).\]
\end{lemma}

The following proposition gives the pointwise-in-time spatial BV regularity of the limiting solution as given by $\sola_t$. It is essentially direct from~\eqref{eq:pi-l1-bv-bound}, Lemma~\ref{lem:transport-bv-bv-bound}, and computing the Lipschitz norm of the velocity field $V$.

\begin{proposition}
\label{prop:limiting-solution-regularity}
     There exists $C(\alpha)>0$ such that for all $t \in [0,1]$, we have the following regularity bound
     \[
     \|S_{0,t}\|_{L^1(\T^2)\rightarrow BV(\T^2)}\leq \begin{cases}
     C t^{-\frac{1}{1-\alpha}} & t\in[s^{i+1}_{i+2}+\sigma_{i+1}/2,s^i_{i+1}],\\
     Ct^{-\frac{1}{1-\alpha}}\Big(\frac{ s^i_{i+1} + \sigma_i/2 - t}{\sigma_i}\Big)^{- \frac{M}{1-\alpha}}&t\in[s^i_{i+1},s^i_{i+1}+\sigma_i/2].
     \end{cases}
     \]
\end{proposition}

\begin{proof}
When $t\in[s^{i+1}_{i+1}+\sigma_i/2,s^i_{i+1}]$, the definition of $S_{0,t}$ and~\eqref{eq:pi-l1-bv-bound} imply that
\[\|S_{0,t}\|_{L^1\rightarrow BV}=\|\Pi_{i+1}\|_{L^1\rightarrow BV}\leq C 2^{(i+1)/2}.\]

When $t\in[s^i_{i+1},s^i_{i+1}+\sigma_i/2]$ then by construction of $V$, we have that
\begin{align*}\int_{s^i_{i+1}}^t \|\nabla V(s,\cdot)\|_{L^\infty_x}\,dr = \int_0^{t - s^i_{i+1}} \sigma_i^{-1} \|\nabla v(\sigma_i^{-1} r,\cdot)\|_{L^\infty_x}\,dr &= \int_0^{\frac{t-s^i_{i+1}}{\sigma_i}} \|\nabla v(r,\cdot)\|_{L^\infty_x}\,dr 
\\&\leq \frac{M}{1-\alpha} \int_0^{\frac{t-s^i_{i+1}}{\sigma_i}} (1/2-r)^{-1}\,dr 
\\&\leq - \frac{M}{1-\alpha} \log\Big( \frac{ s^i_{i+1} + \sigma_i/2 - t}{\sigma_i}\Big).
\end{align*}
Lemma~\ref{lem:transport-bv-bv-bound} thus implies that
\[\Big\|\sol_{s^i_{j+1},t}^{V,\kappa,\T^2}\Big\|_{BV \to BV} \leq \exp\Big(\int_{s^i_{j+1}}^t \|\nabla V(s,\cdot)\|_{L^\infty_x}\,dr\Big)\leq \Big( \frac{ s^i_{i+1} + \sigma_i/2 - t}{\sigma_i}\Big)^{-\frac{M}{1-\alpha}}.\]
In total, this gives us that
\[\|S_{0,t}\|_{L^1\rightarrow BV}\leq \Big\|\sol_{s^i_{i+1},t}^{V,0,\T^2}\Big\|_{BV\rightarrow BV}\|\Pi_{i+1}\|_{L^1\rightarrow BV}\leq C 2^{(i+1)/2}\Big( \frac{ s^i_{i+1} + \sigma_i/2 - t}{\sigma_i}\Big)^{-\frac{M}{1-\alpha}}.\]
Using the definition of $s^i_j$, we conclude.
\end{proof}

We now seek to bound the integrated-in-time spatial regularity in an $H^s(\T^2)$ space of the limiting solution as given by $\sola_t$. This follows from the BV bound given by Proposition~\ref{prop:limiting-solution-regularity}, the $BV \to H^{s,p}$ embedding of Corollary~\ref{cor:BV-embedding-bessel}, and the interpolation of Theorem~\ref{thm:interpolation-bound}.

\begin{proposition}
\label{prop:ideal-soln-operator-regularity}
    There exist $C(\alpha)>0$ such that
    \[\|\sola_t\|_{L^2(\T^2) \to L^2([0,1],H^{\frac{1-\alpha}{8(M+1)}}(\T^2))} \leq C.\]
\end{proposition}

\begin{proof}
    By Corollary~\ref{cor:BV-embedding-bessel} and Proposition~\ref{prop:limiting-solution-regularity}, there exists $p>1$ such that
    \[
     \|S_{0,t}\|_{L^1\rightarrow H^{1/2,p}}\leq \begin{cases}
     C t^{-\frac{1}{1-\alpha}} & t\in[s^{i+1}_{i+2}+\sigma_{i+1}/2,s^i_{i+1}],\\
     Ct^{-\frac{1}{1-\alpha}}\Big(\frac{ s^i_{i+1} + \sigma_i/2 - t}{\sigma_i}\Big)^{- \frac{M}{1-\alpha}}&t\in[s^i_{i+1},s^i_{i+1}+\sigma_i/2].
     \end{cases}
     \]
     Using Theorem~\ref{thm:interpolation-bound}, for any $\theta \in (0,1/2)$ we have that
     \[\|S_{0,t}\|_{L^2 \to H^{\theta/2}} \leq C\|S_{0,t}\|_{L^{ \frac{2(1-\theta)}{1-2\theta}} \to L^{ \frac{2(1-\theta)}{1-2\theta}}}^{1-\theta} \|S_{0,t}\|_{L^1 \to H^{1/2,p}}^{\theta} \leq C\|S_{0,t}\|_{L^1 \to H^{1/2,p}}^{\theta}.\]
     Taking $\theta := \frac{1-\alpha}{4(M+1)}$, we then get that
     \[\|S_{0,t}\|_{L^2 \to H^{\frac{1-\alpha}{8(M+1)}}} \leq \begin{cases}
     C t^{-\frac{1}{4}} & t\in[s^{i+1}_{i+2}+\sigma_{i+1}/2,s^i_{i+1}],\\
     Ct^{-\frac{1}{4}}\Big(\frac{ s^i_{i+1} + \sigma_i/2 - t}{\sigma_i}\Big)^{- \frac{1}{4}}&t\in[s^i_{i+1},s^i_{i+1}+\sigma_i/2].
     \end{cases}\]
    Then
    \begin{align*}
        \|\sola_t\|_{L^2 \to L^2_tH^{\frac{1-\alpha}{8(M+1)}}_x}^2 &= \int_0^1\|\sola_t\|_{L^2 \to H^{\frac{1-\alpha}{8(M+1)}}}^2\,dt
        \\&\leq C\sum_{i=0}^\infty (s_i^i)^{-1/2} \sigma_i + (s_i^i)^{-1/2} \int_{s^i_{i+1}}^{s^i_{i+1} + \sigma_i/2}  \Big(\frac{ s^i_{i+1} + \sigma_i/2 - t}{\sigma_i}\Big)^{- \frac{1}{2}}\,dt
        \\&\leq C\sum_{i=0}^\infty (s^i_i)^{-1/2} \sigma_i \int_0^{1/2}  r^{-1/2}\,dr
        \\&\leq C \sum_{i=0}^\infty 2^{-i(1-\alpha)/4} \leq C,
    \end{align*}
    as desired.
\end{proof}

We recall the standard $L^2$ energy identity for the advection-diffusion equation
\[\frac{d}{dt}\big \|\sol^{V,\kappa,\T^2}_{0,t} \theta\big\|_{L^2}^2 = - 2\kappa\big \|\nabla\sol^{V,\kappa,\T^2}_{0,t} \theta\big\|_{L^2}^2.\]
Integrating this in time and interpolating we get the following bound, which gives a regularity bound for the advection-diffusion solution that degenerates as $\kappa\rightarrow 0$.
\begin{proposition}
\label{prop:interpolated-energy-identity}
    For all $\kappa >0$, we have the bound
    \[\big\|\sol^{V,\kappa,\T^2}_{0,t}\big\|_{L^2(\T^2) \to L^2([0,1],H^{\frac{1-\alpha}{8(M+1)}}(\T^2))} \leq C\kappa^{-\frac{1-\alpha}{16(M+1)}}.\]
\end{proposition}

\begin{proof}
    Integrating the energy identity, we get.
     \begin{equation}
     \label{eq:integrated-energy-identity}
     \big\|\sol^{V,\kappa,\T^2}_{0,t}\big\|_{L^2 \to L^2_tH^1_x} \leq \kappa^{-1/2}.
     \end{equation}
     Then we compute using Theorem~\ref{thm:interpolation-bound} and H\"older's inequality,
     \begin{align*}
         & \big\|\sol^{V,\kappa,\T^2}_{0,t}\big\|_{L^2 \to L^2_t H^{\frac{1-\alpha}{8(M+1)}}_x}
         \\&\qquad= \bigg(\int_0^1 \big\|\sol^{V,\kappa,\T^2}_{0,t}\big\|_{L^2 \to H^{\frac{1-\alpha}{8(M+1)}}}^2\,dt\bigg)^{1/2}
         \\&\qquad\leq \bigg(\int_0^1 \big\|\sol^{V,\kappa,\T^2}_{0,t}\big\|_{L^2 \to L^2}^{2(1-\frac{1-\alpha}{8(M+1)})} \big\|\sol^{V,\kappa,\T^2}_{0,t}\big\|_{L^2 \to H^{1}}^{\frac{2(1-\alpha)}{8(M+1)} }\,dt\bigg)^{1/2}
         \\&\qquad \leq C\Bigg( \bigg(\int_0^1 \big\|\sol^{V,\kappa,\T^2}_{0,t}\big\|_{L^2 \to L^2}^{2}\,dt\bigg)^{(1-\frac{1-\alpha}{8(M+1)})}\bigg(\int_0^1 \big\|\sol^{V,\kappa,\T^2}_{0,t}\big\|_{L^2 \to H^{1}}^{2 }\,dt\bigg)^{\frac{1-\alpha}{8(M+1)}}\Bigg)^{1/2}
         \\&\qquad\leq  C\big\|\sol^{V,\kappa,\T^2}_{0,t}\big\|_{L^2 \to L^2_tH^{1}_x}^{\frac{1-\alpha}{8(M+1)}} 
         \\&\qquad\leq C\kappa^{-\frac{1-\alpha}{16(M+1)}},
     \end{align*}
     where we use~\eqref{eq:integrated-energy-identity} and that $\big\|\sol^{V,\kappa,\T^2}_{0,t}\big\|_{L^2 \to L^2} \leq 1.$
\end{proof}

We are now ready to implement the scheme sketched in Subsection~\ref{ss:overview} to prove the anomalous regularization of Theorem~\ref{thm:anomalous-regularization}. We use the estimates provided Corollary~\ref{cor:is-close-to-ideal-pointwise}, Proposition~\ref{prop:ideal-soln-operator-regularity}, and Proposition~\ref{prop:interpolated-energy-identity} as the inputs into the interpolation estimate of Theorem~\ref{thm:interpolation-bound}.

\begin{proof}[Proof of Theorem~\ref{thm:anomalous-regularization}]
    Let $\theta \in (0,1)$ and write
    \begin{equation}
    \label{eq:intermittent-regularity-term-1}
    \big\|\sol^{V,\kappa,\T^2}_{0,t}\big\|_{L^2 \to L^2_tH^{\frac{\theta(1-\alpha)}{8(M+1)}}_x} \leq \|\sola_t\|_{L^2 \to L^2_tH^{\frac{\theta(1-\alpha)}{8(M+1)}}_x} + \big\|\sol^{V,\kappa,\T^2}_{0,t} - \sola_t\big\|_{L^2 \to L^2_tH^{\frac{\theta(1-\alpha)}{8(M+1)}}_x}.
    \end{equation}
    Then we note that by Theorem~\ref{thm:interpolation-bound} and H\"older's inequality,
    \begin{align*}
    &\big\|\sol^{V,\kappa,\T^2}_{0,t} - \sola_t\big\|_{L^2 \to L^2_tH^{\frac{\theta(1-\alpha)}{8(M+1)}}_x} 
    \\&\qquad= \Big(\int_0^1 \big\|\sol^{V,\kappa,\T^2}_{0,t} - \sola_t\big\|_{L^2 \to H^{\frac{\theta(1-\alpha)}{8(M+1)}}}^2\,dt\Big)^{1/2}
    \\&\qquad\leq \Big(\int_0^1 \big\|\sol^{V,\kappa,\T^2}_{0,t} - \sola_t\big\|_{L^2 \to L^2}^{2(1-\theta)} \big\|\sol^{V,\kappa,\T^2}_{0,t} - \sola_t\big\|_{L^2 \to H^{\frac{1-\alpha}{8(M+1)}}}^{2\theta}\,dt\Big)^{1/2}
    \\&\qquad\leq \bigg(\Big(\int_0^1 \big\|\sol^{V,\kappa,\T^2}_{0,t} - \sola_t\big\|_{L^2 \to L^2}^{2} \,dt\Big)^{1-\theta}\Big(\int_0^1  \big\|\sol^{V,\kappa,\T^2}_{0,t} - \sola_t\big\|_{L^2 \to H^{\frac{1-\alpha}{8(M+1)}}}^{2}\,dt\Big)^\theta\bigg)^{1/2}
    \\&\qquad= \big\|\sol^{V,\kappa,\T^2}_{0,t} - \sola_t\big\|_{L^2 \to L^2_tL^2_x}^{1-\theta} \big\|\sol^{V,\kappa,\T^2}_{0,t} - \sola_t\big\|_{L^2 \to L^2_tH^{\frac{1-\alpha}{8(M+1)}}_x}^\theta
    \end{align*}
    The triangle inequality, Proposition~\ref{prop:ideal-soln-operator-regularity} and Proposition~\ref{prop:interpolated-energy-identity} imply that 
    \begin{align*}
     \big\|\sol^{V,\kappa,\T^2}_{0,t} - \sola_t\big\|_{L^2 \to L^2_tH^{\frac{1-\alpha}{8(M+1)}}_x}&\leq \big\|\sol^{V,\kappa,\T^2}_{0,t}\big\|_{L^2 \to L^2_tH^{\frac{1-\alpha}{8(M+1)}}_x} + \|\sola_t\|_{L^2 \to L^2_tH^{\frac{1-\alpha}{8(M+1)}}_x}
     \\&\leq C\kappa^{-\frac{1-\alpha}{16(M+1)}}.   
    \end{align*}
    Combining the previous two displays and~\eqref{eq:soln-operator-close-to-ideal-integrated} in Corollary~\ref{cor:is-close-to-ideal-pointwise} we have that
    \[\|\sol^{V,\kappa,\T^2}_{0,t} - \sola_t\|_{L^2 \to L^2_tH^{\frac{1-\alpha}{8(M+1)}}_x}\leq C\kappa^{\frac{(1-\theta)(1-\alpha)^2}{96}- \frac{\theta(1-\alpha)}{16(M+1)}},\]
    thus~\eqref{eq:intermittent-regularity-term-1} and Proposition~\ref{prop:ideal-soln-operator-regularity} imply that 
    \[\|\sol^{V,\kappa,\T^2}_{0,t}\|_{L^2 \to L^2_tH^{\frac{\theta(1-\alpha)}{8(M+1)}}_x} \leq C + C\kappa^{\frac{(1-\theta)(1-\alpha)^2}{96}- \frac{\theta(1-\alpha)}{16(M+1)}}.\]
    Choosing
    \[\theta := \frac{(1-\alpha)(M+1)}{6+(1-\alpha)(M+1)},\]
    we get
    \[\|\sol^{V,\kappa,\T^2}_{0,t}\|_{L^2 \to L^2_tH^{\frac{(1-\alpha)^2}{8( M+7)}}_x} \leq \|\sol^{V,\kappa,\T^2}_{0,t}\|_{L^2 \to L^2_tH^{\frac{(1-\alpha)^2(M+1)}{8(M+1)(6 + (1-\alpha)(M+1))}}_x} \leq C,\]
    as desired.
\end{proof}

In order to conclude, we now only need to prove the intermittent regularity statement of Theorem~\ref{thm:intermittency}. We will want the following soft version of Corollary~\ref{cor:is-close-to-ideal-pointwise} which has the advantage of being for every time $t \in (0,1]$ as opposed to Corollary~\ref{cor:is-close-to-ideal-pointwise} which trivializes at the singular times $t = s^i_{i+1} + \sigma_i/2$.

\begin{lemma}
    \label{lem:pointwise-weak-convergence}
    Fix $\theta_0 \in L^2(\T^2)$. Then for all $t \in (0,1]$, we have that as $\kappa \to 0$,
    \[\sol^{V,\kappa,\T^2}_{0,t} \theta_0 \stackrel{L^2}{\rightharpoonup} \sola_t \theta_0.\]
\end{lemma}

\begin{proof}
    Since strong convergence in $L^2$ implies weak convergence in $L^2$, Corollary~\ref{cor:is-close-to-ideal-pointwise} directly implies the lemma for all $t \in (0,1]$ except when $t = s^i_{i+1} + \sigma_i/2$ for $i \in \N$. So fix $i \in \N$ and let $t = s^i_{i+1} + \sigma_i/2$. Note then that
    \[e^{\kappa \sigma_i \Delta/2} \sol_{0,t}^{V,\kappa,\T^2} \theta_0 = \sol_{0,t+ \sigma_i/2}^{V,\kappa,\T^2} \theta_0.\]
    Let $\kappa_j \to 0$ be an arbitrary sequence. Then we have that $\sol_{0,t}^{V,\kappa_j,\T^2} \theta_0$ is bounded in $L^2$, so there is a subsequence $\kappa_{k_j}$ so that $\sol_{0,t}^{V,\kappa_{k_j},\T^2} \theta_0 \stackrel{L^2}{\rightharpoonup} \phi$ for some  $\phi \in L^2$. Thus by compactness $\sol_{0,t}^{V,\kappa_{k_j},\T^2} \theta_0 \stackrel{H^{-1}}{\to} \phi$. Then $e^{\kappa_{k_j} \sigma_i \Delta/2} \sol_{0,t}^{V,\kappa_{k_j},\T^2} \theta_0 \stackrel{H^{-1}}{\to} \phi$ but also
    \[e^{\kappa_{k_j} \sigma_i\Delta/2} \sol_{0,t}^{V,\kappa_{k_j},\T^2} \theta_0 = \sol_{0,t+ \sigma_i/2}^{V,\kappa_{k_j},\T^2} \theta_0 \stackrel{L^2}{\rightharpoonup} \sola_{t+\sigma_i/2} \theta_0 = \sola_t \theta_0.\]
    Thus $\sol_{0,t}^{V,\kappa_{k_j},\T^2} \theta_0 \stackrel{L^2}{\rightharpoonup} \sola_t \theta_0$. Since the sequence $\kappa_j$ was arbitrary, this allows us to conclude. 
\end{proof}

Finally, we conclude by noting that Theorem~\ref{thm:intermittency} is a direct consequence of the following more precise proposition, which characterizes the time $t_*$ in Theorem~\ref{thm:intermittency}. The below proposition is direct from norm lower semi-continuity, Lemma~\ref{lem:pointwise-weak-convergence}, and Proposition~\ref{prop:piecewise-constant-implies-not-in-some-spaces} with the explicit form of $\sola_t$ to get the $\|\sola_t \theta_0\|_{H^{\beta,p}} = \infty$. 

\begin{proposition}
    \label{prop:intermittency}
    Fix $\theta_0 \in L^2(\T^2)$ with $\theta_0 \ne 0$ and $\int \theta_0(x)\,dx =0$. Let 
    \[i_* := \sup \{ i \in \N : \Pi_i \theta_0 =0\} \geq 0.\]
    Then $i_* < \infty$. Define $t_* := s^{i_*}_{i_*+1} + \sigma_{i_*}/2.$ Then for all $t \in (0,t_*)$, $\beta \in (0,1)$, and  $p> \beta^{-1}$, we have that
    \[\lim_{\kappa \to 0} \|\sol^{V,\kappa,\T^2}_{0,t} \theta_0\|_{H^{\beta,p}(\T^2)} = \infty.\]
\end{proposition}

\begin{proof}
    We note $i_* < \infty$ as otherwise $\Pi_i \theta_0 =0$ for all $i \in \N$, which implies that $\theta_0 = 0$, contradicting our assumption.

    Let $t \in (0,t_*)$. Then by Lemma~\ref{lem:pointwise-weak-convergence}, 
    \[\sol^{V,\kappa,\T^2}_{0,t} \theta_0 \stackrel{L^2}{\rightharpoonup} \sola_t \theta_0.\]
    By Banach-Alaoglu and weak lower semi-continuity of norms, it suffices then to prove that 
    \[\|\sola_t \theta_0\|_{H^{\beta,p}} = \infty.\]
    Then by the assumption that $t  \in (0,t_*)$ and the definition of $t_*$, we have that for some $i \in \N$ with $i \geq i_*$, either $\sola_t \theta_0 = \Pi_{i+1} \theta_0$ or $\sola_t \theta_0 = \sol_{s^i_{i+1},t}^{V,0,\T^2} \Pi_{i+1} \theta_0$ with $t < s^i_{i+1} + \sigma_i/2$. Since $i \geq i_*$, we have that $\Pi_{i+1} \theta_0 \ne 0$. Thus in the former case that $\sola_t \theta_0 = \Pi_{i+1} \theta_0$, we conclude immediately by Proposition~\ref{prop:piecewise-constant-implies-not-in-some-spaces}. In the latter case, we note that $\sol^{V,0,\T^2}_{s^i_{i+1},t}$ is given by precomposition with a $C^\infty$ diffeomorphism, so we can again conclude by Proposition~\ref{prop:piecewise-constant-implies-not-in-some-spaces}.
\end{proof}

\appendix

\section{Technical lemmas}
\label{appendix:lemmas}

In this section of the appendix we prove the technical lemmas used in the proofs of the main propositions and theorems. This includes various estimates on projection/advection-diffusion operators, and various properties of fractional regularity Besov spaces.

\subsection{\texorpdfstring{Bounds on the projectors $\Pi_j$}{Bounds on the projectors Pi\_j}}

Here we prove Lemma~\ref{lem:pi-bounds} which gives bounds on the projector operator $\Pi_j$.

\begin{proof}[Proof of Lemma~\ref{lem:pi-bounds}]
    \eqref{eq:pi-l1-l1-bound} follows from a direct computation. 

    For~\eqref{eq:pi-l1-linfty-bound}, letting $f \in C^\infty(\T^2),$ we note that for all $x_j \in \Lambda_j$ and $x \in A_j + x_j$ we have that
    \[|\Pi_j f(x)| = \frac{1}{|A_j|}\Big| \int_{A_j +x_j} f(y)\,dy\Big| \leq C2^j \|f\|_{L^1}.\]

    For~\eqref{eq:pi-l1-bv-bound}, let $f \in C^\infty(\T^2)$ and write
    \[\Pi_j f = \sum_{x_j \in \Lambda_j} c_{x_j} \indc_{A_j + x_j}.\]
    We then note that
    \begin{equation}
        \label{eq:bv-comp-in-F_j}
    \|\Pi_j f\|_{BV} \leq C 2^{-j/2} \sum_{x_j \in \Lambda_j} \sum_{i \in \{1,2,3,4\}}|c_{x_j} - c_{x_j^i}|,
    \end{equation}
    where the $x^i_j \in \Lambda_j$ are the ``neighboring'' boxes to $x_j.$ 

    Then we have that
    \[ 2^{-j/2} \sum_{x_j \in \Lambda_j} \sum_{i \in \{1,2,3,4\}}|c_{x_j} - c_{x_j^i}| \leq C 2^{-j/2} \sum_{x_j\in \Lambda_j} |c_{x_j}| \leq C2^{j/2} \|\Pi_j f\|_{L^1} \leq C 2^{j/2} \|f\|_{L^1},\]
    where we use~\eqref{eq:pi-l1-l1-bound}. Then~\eqref{eq:bv-comp-in-F_j} allows us to conclude~\eqref{eq:pi-l1-bv-bound}, using that $C^\infty(\T^2)$ is dense in $L^1(\T^2)$.

    Finally, for $f \in C^\infty(\T^2)$, then using the Poincar\'e inequality,
    \begin{align*}\|(1-\Pi_k)  f\|_{L^1} &= \Big|\sum_{x \in \Lambda_{k}} \int_{A_{k} + x}  f(y) - (f)_{A_{k} + x}\,dy\Big|
    \\&\leq \sum_{x \in \Lambda_{k}} \|  f - ( f)_{A_{k} + x}\|_{L^1(A_{k}+x)}
    \\&\leq C 2^{-k/2} \sum_{x \in \Lambda_{k}} \| \nabla f\|_{L^1(A_{k}+x)}
    \\&=  C 2^{-k/2}\| \nabla f\|_{L^1},
    \end{align*}
    as claimed. For general $f \in W^{1,1}$, we conclude~\eqref{eq:pi_projecting-bound} by approximating with mollifications.
\end{proof}

\subsection{Estimates for advection-diffusion and transport equations}

In this subsection we prove a number of estimates for the solution operators of advection-diffusion and transport equations. Most of these proofs use the stochastic characteristic representation of advection-diffusion equations or the characteristic representation of transport equations.

First, we prove Lemma~\ref{lem:box-bounds-sharp}, which we use to control errors introduced when splitting our domain into a sub grid of cells.

\begin{proof}[Proof of Lemma~\ref{lem:box-bounds-sharp}]
    Without loss of generality, we can suppose by approximation that $u$ is in $ L^1([0,1], C^\infty(\T^2))$. We let $X^\kappa_t(x)$ be the stochastic flow solving the backward SDE 
    \[\begin{cases}dX^\kappa_t(x) = u(t,X^\kappa_t(x))dt + \sqrt{2\kappa}\,dw_t \\ X^\kappa_1(x) = x, \end{cases}\]
    where $w_t$ is a standard Brownian motion in $\R^2$. We note that 
    \[\sol^{u,\kappa,\T^2}_{0,1} \Pi_j \theta(y) = \E \big[\Pi_j \theta(X^\kappa_1(y))\big].\]
     Note then that
    \[X^\kappa_0(x) = x+\int_0^1 u(t,X^\kappa_t(x))\,dt + \sqrt{2\kappa} w_1,\]
    thus
    \[\big|x -X^\kappa_0(x)\big| \leq \|u\|_{L^1_tL^\infty_x} + \sqrt{2\kappa} |w_1|.\]
    
    Let $x_j \in \Lambda_j$ such that $x \in A_j+x_j$. Then we have that by the bounds $\|u\|_{L^1_t L^\infty_x} \leq N 2^{-j/2}$ and $\kappa^{-1/2} \leq 2^{-j/2}$ as well as the Gaussian tail bounds,
    \begin{align*}
    \big|\sol^{u,\kappa,\T^2}_{0,1} \Pi_j \theta(x)\big| &\leq \sum_{n=1}^\infty \P\big( (n-1)/\sqrt{2} \leq |w_1| < n/\sqrt{2} \big) \sup_{|x-y| \leq N2^{-j/2} + \sqrt{\kappa} n} |\Pi_j \theta(y)|
    \\&\leq \sum_{n=1}^\infty C e^{-n^2/C} \sup_{|x_j- y| \leq (N+n+1) 2^{-j/2}} |\Pi_j \theta(y)|. 
    \end{align*}
    Write then
    \[\Pi_j \theta = \sum_{y_j \in \Lambda_j} c_{y_j} \indc_{A_j + y_j},\]
    so that
    \[ \sup_{|x_j- y| \leq (N+n+1) 2^{-j/2}} |\Pi_j \theta(y)| \leq \sup_{y_j \in \Lambda_j, |x_j-y_j| \leq (N+n+2) 2^{-j/2}} |c_{y_j}|.\]
    Thus together we have
    \[ |\sol^{u,\kappa,\T^2}_{0,1} \Pi_j \theta(x)| \leq \sum_{n=1}^\infty C e^{-n^2/C} \sup_{y_j \in \Lambda_j, |x_j-y_j| \leq (N+n+2) 2^{-j/2}} |c_{y_j}|.\]
    We note that the same argument applies to $ |\sol^{V,\kappa,\T^2}_{0,t} \Pi_j \theta(x)|$ for all $t \in [0,1]$, giving the same bound. Thus we in fact get
    \begin{align*}
        \sum_{x_j \in \Lambda_j} 2^{-j}\sup_{t \in [0,1]}\sup_{x \in A_j+x_j} \big|\sol_{0,t}^{u,\kappa,\T^2} \Pi_j \theta(x)\big|&\leq \sum_{x_j\in \Lambda_j} 2^{-j} \sum_{n=1}^\infty C e^{-n^2/C} \sup_{y_j \in \Lambda_j, |x_j-y_j| \leq (N+n+2) 2^{-j/2}} |c_{y_j}|
        \\&\leq  \sum_{x_j \in\Lambda_j} 2^{-j}\sum_{n=1}^\infty C e^{-n^2/C} \sum_{y_j \in \Lambda_j, |x_j-y_j| \leq (N+n+2) 2^{-j/2}} |c_{y_j}|
        \\&\leq C\sum_{y_j \in \Lambda_j} 2^{-j} |c_{y_j}| \sum_{x_j \in \Lambda_j} \sum_{n \geq (2^{j/2}|x_j-y_j| - N -2) \lor 0}  e^{-n^2/C}
        \\&\leq C\sum_{y_j \in \Lambda_j} 2^{-j} |c_{y_j}| \sum_{x_j \in \Lambda_j}  e^{-((2^{j/2}|x_j-y_j| - N -2) \lor 0)^2/C}
        \\&\leq C \sum_{y_j \in \Lambda_j} 2^{-j} |c_{y_j}| \sum_{\ell \in \Z^2}  e^{-((|\ell| - N -2) \lor 0)^2/C}
        \\&\leq  C(N+1)^2\sum_{y_j \in \Lambda_j} 2^{-j} |c_{y_j}| = C (N+1)^2\|\Pi_j \theta\|_{L^1},
    \end{align*}
    as desired.
\end{proof}

Next, we prove Lemma~\ref{lem:diffusion-BV-L1-bound}, which controls how much $BV$ functions are perturbed in $L^1$ by diffusion. This follows as a simple corollary of the following lemma.

\begin{lemma}\label{lem:translation}
For $y \in \T^2$, let $\tau_y : L^1(\T^2) \to L^1(\T^2)$ be the linear operator $\tau_y \theta(x) := \theta(x-y)$.
Then we have the operator bound
\[\|\tau_y - 1\|_{BV(\T^2) \to L^1(\T^2)} \leq |y|. \]
\end{lemma}
\begin{proof}
    Let $f \in C^\infty(\T^2)$. Then
    \begin{align*}\|\tau_y f - f\|_{L^1} &= \int |f(x+y) - f(x)|\,dx
    \\&= \int \Big|\int_0^1 y \cdot \nabla f(x+ty)\,dt\big|dx
    \\&\leq |y|\int \int_0^1 |\nabla f(x+ty)|\,dtdx \leq |y| \|\nabla f\|_{L^1} = |y| \|f\|_{W^{1,1}}.
    \end{align*}
    We conclude for $f \in BV$ by mollifying.
\end{proof}

\begin{proof}[Proof of Lemma~\ref{lem:diffusion-BV-L1-bound}]
For any $\theta\in BV(\T^2)$
\begin{align*}
\|e^{t\Delta}\theta-\theta\|_{L^1}\leq \int\int |\theta(x-y)-\theta(x)|\Phi_t(y)\,dx\,dy&\leq \int \|\tau_y\theta-\theta\|_{L^1}\Phi_t(y)\,dy
\\&\leq \|\theta\|_{BV}\int |y|\Phi_t(y)\,dy
\\&\leq Ct^{1/2}\|\theta\|_{BV},
\end{align*}
where $\Phi_t(y)$ is the standard heat kernel on $\T^2$ and the second inequality follows by Lemma~\ref{lem:translation}.
\end{proof}

We now prove Lemma~\ref{lem:drift-diffusion-small}, which controls the displacement of some $W^{1,1}$ data in $L^1$ under the action of the advection-diffusion equation.

\begin{proof}[Proof of Lemma~\ref{lem:drift-diffusion-small}]
    Without loss of generality, we can suppose by approximation that $u$ is in $L^1([0,1], C^\infty(\T^2))$ and considering the action of the operators on an arbitrary $\theta \in C^\infty(\T^2)$. Then, we let $X^\kappa_t(x)$ be the stochastic flow solving the backward SDE 
    \[\begin{cases}dX^\kappa_t(x) = u(t,X^\kappa_t(x))dt + \sqrt{2\kappa}\,dw_t \\ X^\kappa_1(x) = x, \end{cases}\]
    where $w_t$ is a standard Brownian motion in $\R^2$. We note that 
    \[\sol^{u,\kappa,\T^2}_{0,1} \theta(x) = \E \theta(X^\kappa_1(x)).\]
    As such, we have that
    \[\big\|(\sol^{u,\kappa,\T^2}_{0,1} - 1) \theta\big\|_{L^1} \leq \E \| \theta \circ X^\kappa_1 - \theta\|_{L^1}.\]
    Define
    \[Y^\kappa_t(x):= X^\kappa_t(x) - \sqrt{2\kappa}  (w_t - w_1),\]
    and note that
    \[\begin{cases}
        dY^\kappa_t(x) = u\big(t, Y^\kappa_t(x) +\sqrt{2\kappa} (w_t - w_1)\big) dt\\
        Y^\kappa_1(x) =x.
    \end{cases}\]
    Then we have that
    \[\theta \circ Y^\kappa_1(x) - \theta(x)=\int_0^1 \frac{d}{dt} \theta \circ Y^\kappa_t(x)\,dt = \int_0^1 u\big(t, Y^\kappa_t(x) +\sqrt{2\kappa} (w_t - w_1)\big) \cdot \nabla \theta \circ Y^\kappa_t(x)\,dt.\]
    Thus
    \[\|\theta \circ X^\kappa_1 -\tau_{\sqrt{2\kappa} w_1} \theta\|_{L^1} = \|\theta \circ Y^\kappa_1 - \theta\|_{L^1} \leq \int_0^1 \|u(t,\cdot)\|_{L^\infty_x} \int |\nabla \theta \circ Y^\kappa_t(x)|\,dxdt = \|u\|_{L^1_tL^\infty_x} \|\nabla \theta\|_{L^1},\]
    where we use that $Y^\kappa_t$ is a volume preserving diffeomorphism to change variables for the last equality. Then we compute using Lemma~\ref{lem:translation}
    \begin{align*}
         \| \theta \circ X^\kappa_1 - \theta\|_{L^1} &\leq  \| \theta \circ X^\kappa_1 -\tau_{\sqrt{ 2\kappa} w_1} \theta\|_{L^1}+  \| \tau_{\sqrt{ 2\kappa} w_1} \theta - \theta\|_{L^1}
         \\&\leq \big(\|u\|_{L^1_tL^\infty_x} + \sqrt{2 \kappa} |w_1| \big) \|\nabla \theta\|_{L^1}.
    \end{align*}
    Putting it together, 
    \[\big\|(\sol^{u,\kappa,\T^2}_{0,1} - 1) \theta\big\|_{L^1} \leq \E  \big(\|u\|_{L^1_tL^\infty_x} + \sqrt{2 \kappa} |w_1| \big) \|\nabla \theta\|_{L^1} \leq \Big(\|u\|_{L^1_tL^\infty_x} + \sqrt{\pi \kappa}\Big) \|\nabla \theta\|_{L^1},\]
    allowing us to conclude.
\end{proof}

Before proving Lemma~\ref{lem:constant_error}, we recall the following estimate.

\begin{lemma}\label{eq:old_boundary_error_lemma}
    For any divergence-free vector field $u \in L^1\big([0,1],W^{1,\infty}(B)\big)$ tangent to $\partial B$, there exists a constant $C>0$ so that for all boundary data $f\in L^\infty([0,1]\times \partial B)$, $\kappa \in (0,1)$, and $\beta<\frac{1}{2}$ we have the estimate
    \[\big\|\sol_{0,t}^{u,\kappa,f}a-a\big\|_{L^1(B)} \leq C\big(\|f\|_{L^\infty([0,1]\times\partial B)}+|a|\big)(t\kappa)^{\beta}\exp\bigg(\int_0^t \|\nabla u(s,\cdot)\|_{L^\infty(B)}\,ds\bigg).\]
\end{lemma}

\begin{proof}
This follows from the proof of Lemma 3.2 in~\cite{hess-childs_universal_2025} with $\partial E=\emptyset$, and keeping more careful track of the constants.
\end{proof}

\begin{proof}[Proof of Lemma~\ref{lem:constant_error}]
For any integer $n$, using Lemma~\ref{eq:old_boundary_error_lemma} and the fact that the solution operator is an $L^1$ contraction, we have that
\begin{align*}
    \big\|\sol^{u,f,\kappa}_{0,t}a-a\big\|_{L^1}&\leq \sum_{i=1}^n \Big\|\sol^{u,f,\kappa}_{\frac{i-1}{n}t,\frac{i}{n}t}a-a\Big\|_{L^1}
    \\&\leq C(\|f\|_{L^\infty([0,1]\times\partial B)}+|a|)\sum_{i=1}^n \Big(\frac{\kappa t}{n}\Big)^{\beta}\exp\Big(\int_{\frac{(i-1)}{n}t}^{\frac{i}{n}t}\|\nabla u(s,x)\|_{L^\infty}\,ds\Big)
    \\&\leq C(t\kappa)^\beta(\|f\|_{L^\infty([0,1]\times\partial B)}+|a|)n^{1-\beta}\exp\Big(\frac{t}{n}\|\nabla u\|_{L^\infty([0,1],L^\infty(B))}\Big).
\end{align*}
Letting $n=\lceil t\|\nabla u\|_{L^\infty([0,1],L^\infty(B))}\rceil$, we thus find that
\[\big\|\sol^{u,f,\kappa}_{0,1}a-a\big\|_{L^1}\leq C\kappa^\beta(\|f\|_{L^\infty([0,1]\times\partial B)}+|a|) ((t\|\nabla u\|_{L^\infty([0,1],L^\infty(B))})^{1-\beta}+1)\]
as claimed.
\end{proof}

Finally, we conclude with the proof of Lemma~\ref{lem:transport-bv-bv-bound} which bounds the rate transport by a Lipschitz vector field can increase the $BV$ norm of some data.

\begin{proof}[Proof of Lemma~\ref{lem:transport-bv-bv-bound}]
First, suppose that $\theta\in C^1(\T^d)$. Then, letting $\Psi_t$ be the flow map
\[
\begin{cases}
\frac{d}{dt}\Psi_t(x)=u(t,\Psi_t(x))\\
\Psi_0(x)=x,
\end{cases}
\]
it holds that $\sol^{u,0,f}_{0,t}\theta=\theta\circ\Psi_t^{-1}$ and that
\[\frac{d}{dt}\nabla\Psi_t^{-1}(x)=-\nabla\Psi_t^{-1}(x)\nabla u(t,x).\]
Applying Gr\"onwall's inequality, this implies that
\[\|\nabla \Psi_t^{-1}\|_{L^\infty}\leq\exp\bigg(\int_0^t\|\nabla u(s,\cdot)\|_{L^\infty}\,ds\bigg).\]
We thus have that
\[\int |\nabla\theta_t|(x)\,dx\leq\int|\nabla\theta\circ\Psi_t^{-1}(x)||\nabla\Psi_t^{-1}(x)|\,ds\leq \exp\bigg(\int_0^t\|\nabla u(s,\cdot)\|_{L^\infty}\,ds\bigg)\int|\nabla\theta(x)|\,ds,\]
where in the last line we've used that $u$ is divergence free to change coordinates. Since, $\sol_{0,t}^{u,0,f}$ preserves $L^1$ norms, we have thus shown that
\[\big\|\sol_{0,t}^{u,0,f}\big\|_{W^{1,1}\rightarrow W^{1,1}}\leq \exp\bigg(\int_0^t\|\nabla u(s,\cdot)\|_{L^\infty}\,ds\bigg).\]
To conclude for $\theta\in BV(\T^d)$ we note that, letting $\theta_\eps:=\theta*\phi_\eps$ where $\phi_\eps$ is a family of standard mollifiers, we have
\begin{align*}
\big\|\nabla \sol^{u,0,f}_{0,t}\theta\big\|_{TV}\leq \liminf_{\eps\rightarrow 0}\|\nabla(\theta_\eps\circ\Psi_t^{-1})\|_{L^1}&\leq \exp\bigg(\int_0^t\|\nabla u(s,\cdot)\|_{L^\infty}\bigg)\lim_{\eps\rightarrow 0}\|\nabla\theta_\eps\|_{L^1}
\\&=\exp\bigg(\int_0^t\|\nabla u(s,\cdot)\|_{L^\infty}\bigg)\|\nabla \theta\|_{TV},
\end{align*}
since $\nabla(\theta_\eps\circ\Psi_{t}^{-1})\rightarrow \nabla(\theta\circ\Psi_{t}^{-1})=\nabla \sol^{u,0,f}_{0,t}\theta$ in distribution as $\eps\rightarrow 0$. This concludes the claim.
\end{proof}

\subsection{Results on fractional regularity spaces}

In this subsection we prove Propositions~\ref{prop:BV-embedding-besov} and~\ref{prop:piecewise-constant-implies-not-in-some-spaces} which respectively show that $BV(\T^d)$ embeds into certain Besov spaces, and that functions with jump discontinuities are not in certain Besov spaces.

In order to prove Proposition~\ref{prop:BV-embedding-besov} we need to introduce the alternative definition of Besov spaces based on the Littlewood-Paley decomposition. Here we follow the exposition of~\cite[Chapter 6]{bergh_interpolation_1976}.

We define a Littlewood-Paley decomposition as follows. Fix $\gamma \in C_c^\infty( \{ 2^{-1} < |\xi| < 2\})$ with $\gamma(\xi)>0$ for all $2^{-1} < |\xi| < 2$. Let
    \[\rho(\xi) :=  \frac{\gamma(\xi)}{\sum_{k=-\infty}^\infty \gamma(2^{-k} \xi)}.\]
Note that $\rho \in C_c^\infty(\{2^{-1} < |\xi| < 2\})$ and that 
\[\sum_{k=-\infty}^\infty \rho(2^{-k} \xi) = 1.\]
For $k \geq -1$, let
\begin{equation*}
    \rho_k(\xi) := \begin{cases} \rho(2^{-k} \xi) & k \geq 0,\\  \sum_{j=-\infty}^{-1} \rho(2^{-j} \xi) & k = -1.\end{cases}
\end{equation*}

\begin{definition}
    For $k\geq -1$ and $f: \T^2 \to \R$, define the $k$-th Littlewood-Paley projection, $\Delta_k f : \T^2 \to \R$, by
    \[\Delta_k f := \F^{-1} \big( \rho_k \F f\big).\]
\end{definition}

Note that if $\int f(x)\,dx =0$, then $\Delta_{-1} f =0.$

\begin{definition}
    For $f : \T^2 \to \R$ with $\int f(x)\,dx =0$ and for any $1 \leq p,q \leq\infty$, and $s \in \R$, we define the (homogeneous) Besov norm $\tilde B^{s,p}_q(\T^2)$ by
    \[\|f\|_{\tilde B^{s,p}_q(\T^2)} := \Big(\sum_{k=0}^\infty \big(2^{sk} \|\Delta_k f\|_{L^p(\T^2)}\big)^q\Big)^{1/q}.\]
\end{definition}

The following classical result gives that this gives an equivalent norm on the Besov spaces.

\begin{proposition}[{\cite[Theorem 2.5.12]{triebel_theory_1983}}]
    \label{prop:besov-equivalent}
    For all $1 \leq p,q \leq \infty$, and $s\in (0,1)$, there exists $C(s,p,q) >0$ such that for all $f : \T^2 \to \R$ with $\int f(x)\,dx =0$, 
    \[C^{-1}\|f\|_{B^{s,p}_q(\T^2)} \leq  \|f\|_{\tilde B^{s,p}_q(\T^2)} \leq C \|f\|_{B^{s,p}_q(\T^2)}.\]
\end{proposition}

Now, by essentially direct computation, we can prove the following lemma.

\begin{lemma}
    \label{lem:control-of-LP-by-BV}
    There exists $C>0$ such that for all $k \geq 0$, $1 \leq p \leq\infty$, and $f \in BV(\T^2)$, we have that
    \[\|\Delta_k f\|_{L^p(\T^2)} \leq C 2^{k(1-2/p)} \|f\|_{BV(\T^2)}.\]
\end{lemma}

\begin{proof}
    Note that
    \[\Delta_k f = \F^{-1} \rho_k * f  = \nabla \Delta^{-1}\F^{-1} \rho_k *\nabla f,\]
    thus
    \begin{equation}\label{eq:LP-youngs}
    \|\Delta_k f\|_{L^p} \leq \|\nabla \Delta^{-1} \F^{-1} \rho_k\|_{L^p} \|f\|_{BV}.
    \end{equation}
    Define
    \[\eta_k(\xi) := -i \frac{\xi}{|\xi|^2} \rho_k(\xi),\]
    so that $\nabla \Delta^{-1} \F^{-1} \rho_k = \F^{-1} \eta_k$. Then
    \[\|\F^{-1} \eta_k\|_{L^\infty} \leq \|\eta_k\|_{L^1}  = \sum_{\xi \in \Z^2} |\xi|^{-1} \rho(2^{-k} \xi) \leq C \int_{\R^2} |\xi|^{-1} \rho(2^{-k} \xi)\,d\xi \leq  C  2^{k} \int_{\R^2} |\xi|^{-1} \rho(\xi)\,d\xi\leq C2^k.\]
    On the other hand, letting $\F_{\R^2}$ denote the Fourier transform of a function $g : \R^2 \to \R$ to a function $\F_{\R^2} g : \R^2 \to \R$. Then by Poisson summation, we have that
    \[\|\F^{-1} \eta_k(x)\|_{L^1_x} = \Big\|\sum_{j \in \Z^2} \F^{-1}_{\R^2} \eta_k(x-j)\Big\|_{L^1_x} \leq \|\F^{-1}_{\R^2} \eta_k\|_{L^1}.\]
    Note that $\eta_k(\xi) = 2^{-k}\eta_0(2^{-k} \xi),$ so $\F^{-1}_{\R^2} \eta_k(x) = 2^{k} (\F^{-1}_{\R^2}\eta_0)(2^k x),$ thus 
    \[\|\F^{-1} \eta_k\|_{L^1} \leq C2^{-k}.\]
    Interpolating, we see that
    \[\|\nabla \Delta^{-1} \F^{-1} \rho_k\|_{L^p} = \|\F^{-1} \eta_k\|_{L^p} \leq \|\F^{-1} \eta_k\|_{L^1}^{1/p} \|\F^{-1} \eta_k\|_{L^\infty}^{1-1/p} \leq 2^{k (1-2/p)}.\]
    Combining with~\eqref{eq:LP-youngs}, we conclude.
\end{proof}

Proposition~\ref{prop:BV-embedding-besov} follows as a direct consequence.

\begin{proof}[Proof of Proposition~\ref{prop:BV-embedding-besov}]
    By Proposition~\ref{prop:besov-equivalent} and Lemma~\ref{lem:control-of-LP-by-BV},
    \[\|f\|_{B^{s,p}_1}  \leq C\|f\|_{\tilde B^{s,p}_1}  \leq C\sum_{k=0}^\infty 2^{sk} \|\Delta_k f\|_{L^p} \leq C \|f\|_{BV} \sum_{k=0}^\infty 2^{k (s + 1 - 2/p)},\]
    using the upper bound on $p$, we conclude.
\end{proof}

Finally, we prove Proposition~\ref{prop:piecewise-constant-implies-not-in-some-spaces}.

\begin{proof}[Proof of Proposition~\ref{prop:piecewise-constant-implies-not-in-some-spaces}]
     We prove the statement for $\|f\|_{B^{s,p}_\infty}$ as the statement for $\|f\|_{H^{s,p}}$ then follows from Proposition~\ref{prop:BL-besov-bessel}.

    Since 
    \[\|f\|_{ B^{s,p}_\infty}=\sup_{|h|>0} \frac{\|f(x)-f(x-h)\|_{L^p_x}}{|h|^s}\]
    and $ps>1$, it suffices for us to show that
    \begin{equation}\label{eq:contradiction_hypothesis}
    \lim_{h\rightarrow 0} \frac{\|f(x)-f(x-h)\|_{L^p_x}^p}{|h|}=0
    \end{equation}
    gives a contradiction.
    
    Fixing $h \in \T^2$ so that $0 < |h| < r/2$, we have that
    \[\|f(x) - f(x-h)\|_{L^p_x}^p \geq \int_{x \in B_{r-|h|}(y)} |f(x) - f(x-h)|^p\,dx \geq |a-b|^p\int_{x \in B_{r-|h|}(y)} |\indc_A(x) - \indc_A(x-h)|\,dx.\]
    Combined with~\eqref{eq:contradiction_hypothesis}, this implies that for all $\ep>0$,
    \[\lim_{h\to0} \int_{x \in B_{r-\ep}(y)} \frac{|\indc_A(x-h) - \indc_A(x)|}{|h|}\,dx =0.\]
    Then, for all $\phi \in C_c^\infty(B_r(y))$ and coordinate directions $i$
    \begin{align*}
    \int \indc_A(x) \partial_{i}\phi(x)\, dx &=\int \indc_A(x) \lim_{h\to0} \frac{\phi(x+he_i)-\phi(x)}{h}\,dx
    \\&=\lim_{h\to0}\int \indc_A(x) \frac{\phi(x+he_i)-\phi(x)}{h}\,dx
    \\&=-\lim_{h\to0}\int \frac{\indc_{A}(y)-\indc_A(y-he_i)}{h} \phi(y)\,dx,
    \end{align*}
    where the second equality follows by the dominated convergence theorem, and the last equality follows by a coordinate transform. Letting $\eps>0$ so that $\supp(\phi)\subset B_{r-\eps}$, we then have that
    \[\bigg|\int \indc_A(x) \partial_i\phi(x)\, dx \bigg|\leq \|\phi\|_{L^\infty}\lim_{h\to0} \int_{x \in B_{r-\ep}(y)} \frac{|\indc_A(x) - \indc_A(x-h)|}{|h|}\,dx=0.\]
    We have thus shown that for all $\phi\in C_c(B_r)$,
    \[\int \indc_A\partial_i\phi=0,\]
    thus $\indc_A\in W^{1,1}(B_r)$ with derivative equal to $0$. As a consequence $\indc_A$ must be almost everywhere a constant in $B_r$, which gives a contradiction to the fact that $0<|A|< |B_r|$.
\end{proof} 

\section{Additional proofs for Richardson dispersion}\label{appendix:variance}

In this section we prove Proposition~\ref{prop:variance_upper_bound} and Corollary~\ref{cor:richardson}. We will repeatedly use the following stability estimate for solutions of ODEs on the torus. Note that we are implicitly associating $\R^2$ with $\T^2$.

\begin{lemma}\label{lem:Bihari--LaSalle}
There exists $C(\alpha)>0$ so that for all vector fields $u\in L^\infty([0,T],C^\alpha(\T^2))$, $x,y\in\T^2$, and $g,h:[0,T]\rightarrow \R^2$, if
\[x_t=x_0+\int_0^t u(s,x_s)\,ds+h(t),\]
\[y_t=y_0+\int_0^t u(s,y_s)\,ds+g(t),\]
then
\[|y_t-x_t|\leq C\Big(|x_0-y_0|+\sup_{s\in[0,t]} |h(s)-g(s)|+(\|u\|_{L^\infty_tC^\alpha_x}t)^\frac{1}{1-\alpha}\Big).\]
\end{lemma}

\begin{proof}
    We have that
    \begin{align*}
        |x_t-y_t|&\leq |x_0-y_0|+\int_0^t |u(s,x_s)-u(s,y_s)|\,ds+|h(t)-g(t)|
        \\&\leq  |x_0-y_0|+\sup_{s\in[0,t]}|h(s)-g(s)|+\|u\|_{L^\infty_tC^\alpha_x}\int_0^t|x_s-y_s|^\alpha\,ds.
    \end{align*}
    The Bihari--LaSalle inequality then immediately implies the claim.
\end{proof}

Proposition~\ref{prop:variance_upper_bound} now follows as a simple corollary.

\begin{proof}[Proof of Proposition~\ref{prop:variance_upper_bound}]

Let $\mu_t$ be the law of $X_t^\kappa$. Then
\[\Var(X_t^\kappa)=\inf_a \int|x-a|\,d\mu_t(x)\leq \iint |x-y|\,d\mu_t(x)\,d\mu_t(y).\]
That is, if $Y_t^\kappa$ is an independent copy of $X_t^\kappa$ then
\[\Var(X_t^\kappa)\leq \E[|X_t^\kappa-Y_t^\kappa|^2].\]
Suppose that $w_t$ and $\tilde{w}_t$ are respectively the generating noises for $X_t^\kappa$ and $Y_t^\kappa$. Then Lemma~\ref{lem:Bihari--LaSalle} implies that
\[|X_t^\kappa-Y_t^\kappa|\leq C \big(\sqrt{\kappa}\sup_{s\in[0,t]}|w_s-\tilde w_s|+(\|u\|_{L^\infty_tC^\alpha_x}t)^{\frac{1}{1-\alpha}}\big).\]
Taking expectations of the square, and using that
\[\E\Big[\sup_{s\in[0,t]}|w_s-\tilde w_s|^2\Big] \leq Ct,\]
this implies that
\[\mathbb{E}|X_t^\kappa-Y_t^\kappa|^2\leq C(\kappa t+(\|u\|_{L^\infty_tC^\alpha_x}t)^{\frac{2}{1-\alpha}}),\]
as claimed.
\end{proof}

Finally, we prove Corollary~\ref{cor:richardson}.

\begin{proof}[Proof of Corollary~\ref{cor:richardson}]
Let $X^\kappa_{t}$ and $Y^\kappa_{t}$ be two solutions to~\eqref{eq:stochastic-intro} with respective independent driving noises $w_t,\tilde{w}_t$ and initial conditions $x_0,y_0$. Then the upper bound of the corollary follows almost exactly as Proposition~\ref{prop:variance_upper_bound}. Letting $R_t^\kappa:=|X^\kappa_t-Y^\kappa_t|$, Lemma~\ref{lem:Bihari--LaSalle} implies that
\[R_t\leq C \big(R_0+\sqrt{\kappa}\sup_{s\in[0,t]}|w_s-\tilde w_s|+t^{\frac{1}{1-\alpha}}\big)\]
for some $C(\alpha)>0$, thus taking expectations of the square
\[\mathbb{E}[(R_t^\kappa)^2]\leq C\big((R_0^\kappa)^2+\kappa t+t^\frac{1}{1-\alpha}\big)),\]
as desired.

For the lower bound, using the definition of the variance and Theorem~\ref{thm:richardson}, we have that 
\begin{equation}\label{eq:difference_lower_bound}
\mathbb\E[(R^\kappa_t)^2]\geq\text{Var}(R_t^\kappa)\geq \Var(X_{t,1}^\kappa)\wedge \Var(X_{t,2}^\kappa)\geq C^{-1}(\kappa t +t^{\frac{2}{1-\alpha}}).
\end{equation}
To conclude the corollary it thus suffices to prove that  $\mathbb\E[(R^\kappa_t)^2]\geq C^{-1} (R^\kappa_0)^2$. Reversing time, Lemma~\ref{lem:Bihari--LaSalle} also implies that
\[R_0^\kappa\leq C( R_t^\kappa+\sqrt{\kappa}\sup_{s\in[0,t]}|w_s-\tilde{w}_s|+t^\frac{2}{1-\alpha}),\]
thus, taking expectations of the square and rearranging, it holds that
\[C^{-1}(R^\kappa_0)^2-C(\kappa t+t^\frac{2}{1-\alpha})\leq \E[(R^\kappa_t)^2].\]
If $\kappa t+t^\frac{2}{1-\alpha}\leq C^{-1}(R^\kappa_0)^2$ for large enough $C$, then this implies that $C^{-1}(R^\kappa_0)^2\leq  \E[(R^\kappa_t)^2]$ as desired. On the other hand, if $\kappa t+t^\frac{2}{1-\alpha}\geq C^{-1}(R^\kappa_0)^2$, then~\eqref{eq:difference_lower_bound} immediately implies that $\mathbb\E[(R^\kappa_t)^2]\geq C^{-1} \E[(R^\kappa_0)^2]$, thus in either case we have the claimed lower bound.
\end{proof}

\section{Sharpness of the intermittent Obukhov-Corrsin bounds}

\label{appendix:intermittency}

The goal of this section is to prove Theorem~\ref{thm:sharpness-of-theo-for-p-equal-infty}. We consider the vector field $v^\alpha$ constructed in~\cite[Definition 2.7]{hess-childs_universal_2025} and the initial data $\Theta_0$ given in~\cite[Theorem 2.6]{hess-childs_universal_2025}. We recall that this vector field and data is essentially the quasi-self-similar perfect mixing construction of~\cite[Section 8]{alberti_exponential_2019}. Then we consider the solution
\[\theta^\alpha_t(x) := \sol_{0,t}^{v^\alpha,0, \T^2} \big(\Theta_0-\tfrac{1}{2}\big).\]
Then examining the time rescaling of~\cite[Definition 2.7]{hess-childs_universal_2025} and using~\cite[Items 4 and 5, Theorem 2.6]{hess-childs_universal_2025} to control $\|\theta^\alpha_{t_j}\|_{BV} \leq C 5^j$ for $j \in \N$ (with $t_j$ the geometric sequence of times given in~\cite[Definition 2.7]{hess-childs_universal_2025}) and using the Lipchitz bound on $U$ to control intermediate times $t \in (t_j,t_{j+1}),$ we extract for $t \leq \frac{1}{2}$,
\begin{align}
\label{eq:theta-Linfty-bound-OC}
\|\theta^\alpha_t\|_{L^\infty} &\leq 1 \\
\label{eq:theta-BV-bound-OC}
\|\theta^\alpha_t\|_{BV} &\leq C\big(\tfrac{1}{2} -t\big)^{- \frac{1}{1-\alpha}}\\
\label{eq:v-Linfty-bound-OC}
\|v^\alpha_t\|_{L^\infty} &\leq C\big(\tfrac{1}{2} -t\big)^{\frac{\alpha}{1-\alpha}}\\
\label{eq:v-Lip-bound-OC}
\|v^\alpha_t\|_{W^{1,\infty}} &\leq C \big(\tfrac{1}{2} -t\big)^{- 1}.
\end{align}
For also recall that for $t > \frac{1}{2}$,
\[\theta^\alpha_t = v^\alpha_t = 0.\]
Interpolating~\eqref{eq:v-Linfty-bound-OC} and~\eqref{eq:v-Lip-bound-OC}, we see for any $\sigma \in (0,1),$
\begin{equation}
    \label{eq:v-bound-interpolated-OC}
\|v^\alpha_t\|_{C^\sigma} \leq C \big(\tfrac{1}{2} - t\big)^{\frac{\alpha-\sigma}{1-\alpha}}.
\end{equation}
Then by Corollary~\ref{cor:BV-embedding-bessel}, for any $s<1$, there exists $p>1$, we have that
\[\|\theta^\alpha\|_{H^{s,p}} \leq C \big(\tfrac{1}{2} - t)^{-\frac{1}{1-\alpha}}.\]
Then using that for any $q <\infty$, $\|\theta^\alpha_t\|_{L^q} \leq \|\theta^\alpha_t\|_{L^\infty} \leq 1$ and Theorem~\ref{thm:interpolation-bound}, we have that for any $\beta \in (0,1), \ep>0$, 
\begin{equation}
\label{eq:theta-bound-interpolated-OC}
\|\theta^\alpha_t\|_{H^{\beta,\beta^{-1} -\ep}} \leq C \big(\tfrac{1}{2} -t\big)^{-\frac{\beta}{1-\alpha}-\ep}.
\end{equation}
We then have that $v^\alpha,\theta^\alpha$ live in the following spaces.

\begin{proposition}
\label{prop:OC-alpha-inclusions}
    For all $\alpha, \beta,\sigma \in (0,1), p \in [1,\infty], q \in [1,\infty),$
    \[v^\alpha \in L^p([0,1], C^\sigma(\T^2)) \subseteq L^p([0,1], B^{\sigma,p}_\infty(\T^2)) \;\; \text{and} \;\;\theta^\alpha_t \in L^q([0,1], H^{\beta,q}(\T^2)) \subseteq L^q([0,1], B^{\beta,q}_\infty(\T^2)),\]
    provided
    \[\beta < \frac{1-\alpha}{q}\quad \text{and} \quad  \sigma  <\frac{1+ \alpha (p-1)}{p}.\]
\end{proposition}
\begin{proof}
    We integrate the bounds~\eqref{eq:v-bound-interpolated-OC} and~\eqref{eq:theta-bound-interpolated-OC}, recalling that $\big(\tfrac{1}{2}-t\big)^\gamma$ is integrable if and only if $\gamma > -1$. This then gives the conditions 
    \[\frac{\alpha - \sigma}{1-\alpha} > -\frac{1}{p}, \quad q < \frac{1}{\beta}, \quad \text{and}\quad \frac{\beta}{1-\alpha} < \frac{1}{q}.\]
    We note that the final condition implies the second. Thus manipulating the first and last conditions, we conclude.
\end{proof}

\begin{corollary}
\label{cor:OC-no-alpha-inclusions}
    For all $\beta, \sigma \in (0,1), p \in [1,\infty], q \in [1,\infty)$ such that
    \[\beta < \frac{1}{q} \quad \text{and}\quad \sigma < 1- \frac{q(p-1)}{p} \beta,\]
    there exists $\alpha \in (0,1)$ such that \[v^\alpha \in L^p([0,1], B^{\sigma,p}_\infty(\T^2)) \quad \text{and} \quad \theta^\alpha_t \in L^q([0,1], B^{\beta,q}_\infty(\T^2)).\]
\end{corollary}
\begin{proof}
    Using that $0<\beta < \frac{1}{q}$, we can choose $\alpha = 1- q\beta - \ep \in (0,1)$ for some $\ep>0$ sufficiently small, and then compute 
    \[\frac{1-\alpha}{q} = \frac{q\beta + \ep}{q} > \beta\]
    and
    \[\frac{1+ \alpha(p-1)}{p} = 1 -\frac{ q(p - 1)}{p} \beta- \frac{\ep (p-1)}{p} > \sigma,\]
    choosing $\ep>0$ sufficiently small. Thus we conclude using Proposition~\ref{prop:OC-alpha-inclusions}. 
\end{proof}

\begin{proof}[Proof of Theorem~\ref{thm:sharpness-of-theo-for-p-equal-infty}]
    We note that since $v^\alpha \in C^\infty_{\mathrm{loc}}([0,1/2) \times \T^2)$, $\|\theta^\alpha_0\|_{L^2} >0$, and $\|\theta^\alpha_t\|_{L^2} =0$ for all $t>\frac{1}{2}$, we get that for the dissipation distribution $D$ defined by
    \[D := \partial_t \theta^2 + \nabla \cdot (u \theta^2),\]
    we have that $D \ne 0$ is a non-trivial negative Radon measure supported on $S := \{1/2\} \times \T^2$. It is clear that $S$ has Hausdorff dimension $2$. 

    Thus to conclude it suffices to apply Corollary~\ref{cor:OC-no-alpha-inclusions} with $p =\infty$ and note that for the conditions it gives---$\beta < \frac{1}{q}$ and $\sigma < 1- q\beta$---the first is implied by the second (together with $\sigma>0$) and the second is equivalent to the condition in Theorem~\ref{thm:sharpness-of-theo-for-p-equal-infty}. 
\end{proof}

{\small
\bibliographystyle{alpha}
\bibliography{keefer-references,references2}

\newcommand{\etalchar}[1]{$^{#1}$}
\begin{thebibliography}{HPSRY24}

\bibitem[ABN22]{albritton_enhanced_2022}
Dallas Albritton, Rajendra Beekie, and Matthew Novack.
\newblock Enhanced dissipation and {H\"ormander}'s hypoellipticity.
\newblock {\em Journal of Functional Analysis}, 283(3):109522, 2022.

\bibitem[ACM19a]{alberti_exponential_2019}
Giovanni Alberti, Gianluca Crippa, and Anna Mazzucato.
\newblock Exponential self-similar mixing by incompressible flows.
\newblock {\em Journal of the American Mathematical Society}, 32(2):445--490, 2019.

\bibitem[ACM19b]{alberti_loss_2019}
Giovanni Alberti, Gianluca Crippa, and Anna~L. Mazzucato.
\newblock Loss of regularity for the continuity equation with non-{L}ipschitz velocity field.
\newblock {\em Ann. PDE}, 5(1):Paper No. 9, 19, 2019.

\bibitem[AV25]{armstrong_anomalous_2025}
Scott Armstrong and Vlad Vicol.
\newblock Anomalous diffusion by fractal homogenization.
\newblock {\em Annals of PDE}, 11(1):2, 2025.

\bibitem[BBPS21]{bedrossian_almost-sure_2021}
Jacob Bedrossian, Alex Blumenthal, and Sam Punshon-Smith.
\newblock Almost-sure enhanced dissipation and uniform-in-diffusivity exponential mixing for advection–diffusion by stochastic {Navier}–{Stokes}.
\newblock {\em Probability Theory and Related Fields}, 179(3):777--834, 2021.

\bibitem[BBPS22]{bedrossian_almost-sure_2022}
Jacob Bedrossian, Alex Blumenthal, and Samuel Punshon-Smith.
\newblock Almost-sure exponential mixing of passive scalars by the stochastic {Navier}–{Stokes} equations.
\newblock {\em The Annals of Probability}, 50(1):241--303, 2022.

\bibitem[BCZ17]{bedrossian_enhanced_2017}
Jacob Bedrossian and Michele Coti~Zelati.
\newblock Enhanced dissipation, hypoellipticity, and anomalous small noise inviscid limits in shear flows.
\newblock {\em Archive for Rational Mechanics and Analysis}, 224(3):1161--1204, 2017.

\bibitem[BDLIS15]{buckmaster_anomalous_2015}
Tristan Buckmaster, Camillo De~Lellis, Philip Isett, and L\'aszl\'o Sz\'ekelyhidi, Jr.
\newblock Anomalous dissipation for {$1/5$}-{H}\"older {E}uler flows.
\newblock {\em Ann. of Math. (2)}, 182(1):127--172, 2015.

\bibitem[BGK98a]{Bernard1998}
Denis Bernard, Krzysztof Gawedzki, and Antti Kupiainen.
\newblock Slow modes in passive advection.
\newblock {\em Journal of Statistical Physics}, 90(3--4):519--569, 1998.

\bibitem[BGK98b]{bernard_slow_1998}
Denis Bernard, Krzysztof Gawedzki, and Antti Kupiainen.
\newblock Slow {Modes} in {Passive} {Advection}.
\newblock {\em Journal of Statistical Physics}, 90(3):519--569, 1998.

\bibitem[BL76]{bergh_interpolation_1976}
J\"oran Bergh and J\"orgen L\"ofstr\"om.
\newblock {\em Interpolation spaces. {A}n introduction}, volume No. 223 of {\em Grundlehren der Mathematischen Wissenschaften}.
\newblock Springer-Verlag, Berlin-New York, 1976.

\bibitem[BSJW23]{burczak_anomalous_2023}
Jan Burczak, László Székelyhidi~Jr., and Bian Wu.
\newblock Anomalous dissipation and {Euler} flows, 2023.
\newblock arXiv:2310.02934.

\bibitem[BV19]{buckmaster_convex_2019}
Tristan Buckmaster and Vlad Vicol.
\newblock Convex integration and phenomenologies in turbulence.
\newblock {\em EMS Surv. Math. Sci.}, 6(1-2):173--263, 2019.

\bibitem[BZG23]{blumenthal_exponential_2023}
Alex Blumenthal, Michele~Coti Zelati, and Rishabh~S. Gvalani.
\newblock Exponential mixing for random dynamical systems and an example of {Pierrehumbert}.
\newblock {\em The Annals of Probability}, 51(4):1559--1601, 2023.

\bibitem[CCS23]{colombo_anomalous_2023}
Maria Colombo, Gianluca Crippa, and Massimo Sorella.
\newblock Anomalous dissipation and lack of selection in the {Obukhov}–{Corrsin} theory of scalar turbulence.
\newblock {\em Annals of PDE}, 9(2):21, 2023.

\bibitem[CET94]{constantin_onsagers_1994}
Peter Constantin, Weinan E, and Edriss~S. Titi.
\newblock Onsager's conjecture on the energy conservation for solutions of {E}uler's equation.
\newblock {\em Comm. Math. Phys.}, 165(1):207--209, 1994.

\bibitem[CGH{\etalchar{+}}03]{chaves_lagrangian_2003}
Marta Chaves, Krzysztof Gawedzki, Peter Horvai, Antti Kupiainen, and Massimo Vergassola.
\newblock Lagrangian {Dispersion} in {Gaussian} {Self}-{Similar} {Velocity} {Ensembles}.
\newblock {\em Journal of Statistical Physics}, 113(5):643--692, 2003.

\bibitem[CIRS25]{cooperman_exponentially_2025}
William Cooperman, Gautam Iyer, Keefer Rowan, and Seungjae Son.
\newblock Exponentially mixing flows with slow enhanced dissipation, July 2025.
\newblock arXiv:2507.21305 [math].

\bibitem[CIS24]{cooperman_harris_2024}
William Cooperman, Gautam Iyer, and Seungjae Son.
\newblock A {Harris} theorem for enhanced dissipation, and an example of {Pierrehumbert}, 2024.
\newblock arXiv:2403.19858.

\bibitem[CKRZ08]{constantin_diffusion_2008}
Peter Constantin, Alexander Kiselev, Lenya Ryzhik, and Andrej Zlatoš.
\newblock Diffusion and mixing in fluid flow.
\newblock {\em Annals of Mathematics}, 168(2):643--674, 2008.

\bibitem[Cor51]{corrsin_spectrum_1951}
Stanley Corrsin.
\newblock On the {Spectrum} of {Isotropic} {Temperature} {Fluctuations} in an {Isotropic} {Turbulence}.
\newblock {\em Journal of Applied Physics}, 22(4):469--473, 1951.

\bibitem[CS14]{cheskidov_euler_2014}
A.~Cheskidov and R.~Shvydkoy.
\newblock Euler equations and turbulence: analytical approach to intermittency.
\newblock {\em SIAM J. Math. Anal.}, 46(1):353--374, 2014.

\bibitem[CS23]{cheskidov_volumetric_2023}
Alexey Cheskidov and Roman Shvydkoy.
\newblock Volumetric theory of intermittency in fully developed turbulence.
\newblock {\em Arch. Ration. Mech. Anal.}, 247(3):Paper No. 45, 35, 2023.

\bibitem[CZD21]{zelati_stochastic_2021}
Michele Coti~Zelati and Theodore~D. Drivas.
\newblock A stochastic approach to enhanced diffusion.
\newblock {\em Annali Scuola Normale Superiore - Classe Di Scienze}, pages 811--834, 2021.

\bibitem[CZDE20]{zelati_relation_2020}
Michele Coti~Zelati, Matias~G. Delgadino, and Tarek~M. Elgindi.
\newblock On the relation between enhanced dissipation timescales and mixing rates.
\newblock {\em Communications on Pure and Applied Mathematics}, 73:1205--1244, 2020.

\bibitem[CZDG24]{zelati_statistically_2023}
Michele Coti~Zelati, Theodore~D. Drivas, and Rishabh~S. Gvalani.
\newblock Mixing by statistically self-similar {G}aussian random fields.
\newblock {\em J. Stat. Phys.}, 191(5):Paper No. 61, 11, 2024.

\bibitem[CZG23]{coti_zelati_enhanced_2023}
Michele Coti~Zelati and Thierry Gallay.
\newblock Enhanced dissipation and {Taylor} dispersion in higher-dimensional parallel shear flows.
\newblock {\em Journal of the London Mathematical Society}, 108(4):1358--1392, 2023.

\bibitem[DE17a]{drivas_lagrangian_2017}
Theodore~D. Drivas and Gregory~L. Eyink.
\newblock A {Lagrangian} fluctuation–dissipation relation for scalar turbulence. {Part} {I}. {Flows} with no bounding walls.
\newblock {\em Journal of Fluid Mechanics}, 829:153--189, 2017.

\bibitem[DE17b]{drivas_lagrangian_2017-1}
Theodore~D. Drivas and Gregory~L. Eyink.
\newblock A {Lagrangian} fluctuation–dissipation relation for scalar turbulence. {Part} {II}. {Wall}-bounded flows.
\newblock {\em Journal of Fluid Mechanics}, 829:236--279, 2017.

\bibitem[DEIJ22]{drivas_anomalous_2022}
Theodore~D. Drivas, Tarek~M. Elgindi, Gautam Iyer, and In-Jee Jeong.
\newblock Anomalous dissipation in passive scalar transport.
\newblock {\em Archive for Rational Mechanics and Analysis}, 243(3):1151--1180, 2022.

\bibitem[DL89]{diperna_ordinary_1989}
Ronald~J DiPerna and Pierre-Louis Lions.
\newblock Ordinary differential equations, transport theory and {Sobolev} spaces.
\newblock {\em Inventiones mathematicae}, 98(3):511--547, 1989.

\bibitem[DLS12]{de_lellis_h-principle_2012}
Camillo De~Lellis and L\'aszl\'o Sz\'ekelyhidi, Jr.
\newblock The {$h$}-principle and the equations of fluid dynamics.
\newblock {\em Bull. Amer. Math. Soc. (N.S.)}, 49(3):347--375, 2012.

\bibitem[DLS13]{de_lellis_dissipative_2013}
Camillo De~Lellis and L\'aszl\'o Sz\'ekelyhidi, Jr.
\newblock Dissipative continuous {E}uler flows.
\newblock {\em Invent. Math.}, 193(2):377--407, 2013.

\bibitem[DLS22]{de_lellis_weak_2022}
Camillo De~Lellis and L\'aszl\'o Sz\'ekelyhidi, Jr.
\newblock Weak stability and closure in turbulence.
\newblock {\em Philos. Trans. Roy. Soc. A}, 380(2218):Paper No. 20210091, 16, 2022.

\bibitem[DRDI24]{de_rosa_support_2024}
Luigi De~Rosa, Theodore~D. Drivas, and Marco Inversi.
\newblock On the support of anomalous dissipation measures.
\newblock {\em J. Math. Fluid Mech.}, 26(4):Paper No. 56, 24, 2024.

\bibitem[DRDII25]{rosa_intermittency_2025}
Luigi De~Rosa, Theodore~D. Drivas, Marco Inversi, and Philip Isett.
\newblock Intermittency and {Dissipation} {Regularity} in {Turbulence}, February 2025.
\newblock arXiv:2502.10032 [math].

\bibitem[Dri22]{drivas_self-regularization_2022}
Theodore~D. Drivas.
\newblock Self-regularization in turbulence from the {K}olmogorov 4/5-law and alignment.
\newblock {\em Philos. Trans. Roy. Soc. A}, 380(2226):Paper No. 20210033, 15, 2022.

\bibitem[DRI24]{de_rosa_intermittency_2024}
Luigi De~Rosa and Philip Isett.
\newblock Intermittency and lower dimensional dissipation in incompressible fluids.
\newblock {\em Arch. Ration. Mech. Anal.}, 248(1):Paper No. 11, 37, 2024.

\bibitem[ED18]{Eyink_Drivas_2018}
Gregory~L. Eyink and Theodore~D. Drivas.
\newblock A lagrangian fluctuation–dissipation relation for scalar turbulence. part iii. turbulent rayleigh–bénard convection.
\newblock {\em Journal of Fluid Mechanics}, 836:560–598, 2018.

\bibitem[EL24]{elgindi_norm_2024}
Tarek~M. Elgindi and Kyle Liss.
\newblock Norm growth, non-uniqueness, and anomalous dissipation in passive scalars.
\newblock {\em Archive for Rational Mechanics and Analysis}, 248(6):120, 2024.

\bibitem[ELM23]{elgindi_optimal_2023}
Tarek~M. Elgindi, Kyle Liss, and Jonathan~C. Mattingly.
\newblock Optimal enhanced dissipation and mixing for a time-periodic, {Lipschitz} velocity field on {T}{\textasciicircum}2, 2023.
\newblock arXiv:2304.05374.

\bibitem[Fan03]{Fannjiang2003}
Albert~C. Fannjiang.
\newblock Invariance principle for inertial-scale behavior of scalar fields in kolmogorov-type turbulence.
\newblock {\em Physica D: Nonlinear Phenomena}, 179(3--4):161--182, May 2003.

\bibitem[FGP10]{flandoli_well-posedness_2010}
F.~Flandoli, M.~Gubinelli, and E.~Priola.
\newblock Well-posedness of the transport equation by stochastic perturbation.
\newblock {\em Invent. Math.}, 180(1):1--53, 2010.

\bibitem[FGV01]{falkovichParticlesFieldsFluid2001}
Gregory Falkovich, Krzysztof Gawedzki, and Massimo Vergassola.
\newblock Particles and fields in fluid turbulence.
\newblock {\em Reviews of Modern Physics}, 73(4):913--975, 2001.
\newblock Publisher: American Physical Society.

\bibitem[FI19]{feng_dissipation_2019}
Yuanyuan Feng and Gautam Iyer.
\newblock Dissipation enhancement by mixing.
\newblock {\em Nonlinearity}, 32(5):1810, 2019.

\bibitem[Fla11]{flandoli_random_2011}
Franco Flandoli.
\newblock {\em Random perturbation of {PDE}s and fluid dynamic models}, volume 2015 of {\em Lecture Notes in Mathematics}.
\newblock Springer, Heidelberg, 2011.
\newblock Lectures from the 40th Probability Summer School held in Saint-Flour, 2010, \'Ecole d'\'Et\'e{} de Probabilit\'es de Saint-Flour. [Saint-Flour Probability Summer School].

\bibitem[FMV98]{frisch_intermittency_1998}
U.~Frisch, A.~Mazzino, and M.~Vergassola.
\newblock Intermittency in {Passive} {Scalar} {Advection}.
\newblock {\em Physical Review Letters}, 80(25):5532--5535, 1998.

\bibitem[Fri95]{frisch_turbulence_1995}
Uriel Frisch.
\newblock Turbulence: {The} {Legacy} of {A}. {N}. {Kolmogorov}, 1995.

\bibitem[Ges18]{gess_regularization_2018}
Benjamin Gess.
\newblock Regularization and well-posedness by noise for ordinary and partial differential equations.
\newblock In {\em Stochastic partial differential equations and related fields}, volume 229 of {\em Springer Proc. Math. Stat.}, pages 43--67. Springer, Cham, 2018.

\bibitem[GGM24]{galeati_anomalous_2024}
Lucio Galeati, Francesco Grotto, and Mario Maurelli.
\newblock Anomalous {Regularization} in {Kraichnan}'s {Passive} {Scalar} {Model}, 2024.
\newblock arXiv:2407.16668.

\bibitem[GK95]{gawedzki_anomalous_1995}
Krzysztof Gawedzki and Antti Kupiainen.
\newblock Anomalous {Scaling} of the {Passive} {Scalar}.
\newblock {\em Physical Review Letters}, 75(21):3834--3837, 1995.

\bibitem[GV00]{GawedzkiVergassola2000}
Krzysztof Gawedzki and Massimo Vergassola.
\newblock Phase transition in the passive scalar advection.
\newblock {\em Physica D: Nonlinear Phenomena}, 138(1--2):63--90, 2000.

\bibitem[HCR25]{hess-childs_universal_2025}
Elias Hess-Childs and Keefer Rowan.
\newblock A universal total anomalous dissipator, 2025.
\newblock arXiv:2501.18526.

\bibitem[HPSRY24]{hairer_lower_2024}
Martin Hairer, Sam Punshon-Smith, Tommaso Rosati, and Jaeyun Yi.
\newblock Lower bounds on the top {Lyapunov} exponent for linear {PDEs} driven by the {2D} stochastic {Navier}--{Stokes} equations, 2024.
\newblock arXiv:2411.10419.

\bibitem[HPZZ23]{hofmanova_anomalous_2023}
Martina Hofmanová, Umberto Pappalettera, Rongchan Zhu, and Xiangchan Zhu.
\newblock Anomalous and total dissipation due to advection by solutions of randomly forced {Navier}--{Stokes} equations, 2023.
\newblock arXiv:2305.08090.

\bibitem[Ise18]{isett_proof_2018}
Philip Isett.
\newblock A proof of {O}nsager's conjecture.
\newblock {\em Ann. of Math. (2)}, 188(3):871--963, 2018.

\bibitem[JR02]{jan_integration_2002}
Yves~Le Jan and Olivier Raimond.
\newblock Integration of {Brownian} vector fields.
\newblock {\em The Annals of Probability}, 30(2):826--873, 2002.

\bibitem[JS24]{johansson_anomalous_2024}
Carl Johan~Peter Johansson and Massimo Sorella.
\newblock Anomalous dissipation via spontaneous stochasticity with a two-dimensional autonomous velocity field, 2024.
\newblock arXiv:2409.03599.

\bibitem[Kol41a]{kolmogorov_dissipation_1941}
Andrei~Nikolaevich Kolmogorov.
\newblock Dissipation of energy in locally isotropic turbulence.
\newblock {\em Akademiia Nauk SSSR Doklady}, 32:16, 1941.

\bibitem[Kol41b]{kolmogorov_local_1941}
Andrei~Nikolaevich Kolmogorov.
\newblock The local structure of turbulence in incompressible viscous fluid for very large {Reynolds}' numbers.
\newblock {\em Akademiia Nauk SSSR Doklady}, 30:301--305, 1941.

\bibitem[Kol41c]{kolmogorov_degeneration_1941}
Andrej~Nikolaevich Kolmogorov.
\newblock On the degeneration of isotropic turbulence in an incompressible viscous fluid.
\newblock {\em Dokl. Akad. Nauk SSSR}, 31(6):319--323, 1941.

\bibitem[Kra68]{kraichnanSmallScaleStructure1968}
Robert~H. Kraichnan.
\newblock Small‐{Scale} {Structure} of a {Scalar} {Field} {Convected} by {Turbulence}.
\newblock {\em The Physics of Fluids}, 11(5):945--953, 1968.

\bibitem[Lor69]{lorenz_predictability_1969}
Edward~N. Lorenz.
\newblock The predictability of a flow which possesses many scales of motion.
\newblock {\em Tellus A: Dynamic Meteorology and Oceanography}, 21(3), January 1969.

\bibitem[Mai16a]{mailybaev_spontaneous_2016}
Alexei~A. Mailybaev.
\newblock Spontaneous stochasticity of velocity in turbulence models.
\newblock {\em Multiscale Model. Simul.}, 14(1):96--112, 2016.

\bibitem[Mai16b]{mailybaev_spontaneously_2016}
Alexei~A. Mailybaev.
\newblock Spontaneously stochastic solutions in one-dimensional inviscid systems.
\newblock {\em Nonlinearity}, 29(8):2238--2252, 2016.

\bibitem[MD18]{miles_diffusion-limited_2018}
Christopher~J. Miles and Charles~R. Doering.
\newblock Diffusion-limited mixing by incompressible flows.
\newblock {\em Nonlinearity}, 31(5):2346, 2018.

\bibitem[MHSW22]{myers_hill_exponential_2022}
Joe Myers~Hill, Rob Sturman, and Mark C.~T. Wilson.
\newblock Exponential mixing by orthogonal non-monotonic shears.
\newblock {\em Physica D: Nonlinear Phenomena}, 434:133224, 2022.

\bibitem[Nas58]{nashContinuitySolutionsParabolic1958}
John Nash.
\newblock Continuity of {Solutions} of {Parabolic} and {Elliptic} {Equations}.
\newblock {\em American Journal of Mathematics}, 80(4):931--954, 1958.

\bibitem[NFS25]{navarro-fernandez_exponential_2025}
Víctor Navarro-Fernández and Christian Seis.
\newblock Exponential mixing by random cellular flows, February 2025.
\newblock arXiv:2502.17273 [math].

\bibitem[NV23]{novack_intermittent_2023}
Matthew Novack and Vlad Vicol.
\newblock An intermittent {O}nsager theorem.
\newblock {\em Invent. Math.}, 233(1):223--323, 2023.

\bibitem[Obu49]{obukhov_structure_1949}
Alexander~M. Obukhov.
\newblock Structure of {Temperature} {Field} in {Turbulent} {Flow}.
\newblock {\em Izv. Akad. Nauk. SSSR, Ser. Geogr. i Geofiz.}, 13:58--69, 1949.

\bibitem[Ric26]{richardson_atmospheric_1926}
Lewis~Fry Richardson.
\newblock Atmospheric diffusion shown on a distance-neighbour graph.
\newblock {\em Proceedings of the Royal Society of London. Series A, Containing Papers of a Mathematical and Physical Character}, 110(756):709--737, 1926.

\bibitem[Row24a]{rowan_accelerated_2024}
Keefer Rowan.
\newblock Accelerated relaxation enhancing flows cause total dissipation.
\newblock {\em Nonlinearity}, 37(9):095010, 2024.

\bibitem[Row24b]{rowan_anomalous_2024}
Keefer Rowan.
\newblock On anomalous diffusion in the {K}raichnan model and correlated-in-time variants.
\newblock {\em Archive for Rational Mechanics and Analysis}, 248(5):93, 2024.

\bibitem[Tri83]{triebel_theory_1983}
Hans Triebel.
\newblock {\em Theory of function spaces}, volume~78 of {\em Monographs in Mathematics}.
\newblock Birkh\"auser Verlag, Basel, 1983.

\bibitem[Vil25]{villringer_enhanced_2024}
David Villringer.
\newblock Enhanced dissipation via the {M}alliavin calculus.
\newblock {\em Electron. Commun. Probab.}, 30:Paper No. 26, 11, 2025.

\end{thebibliography}
}

\end{document}